\documentclass[reqno,11pt]{amsart}
\usepackage{amsmath, amsthm, amssymb,bbm, stmaryrd, tikz, mathrsfs} 
\usepackage[normalem]{ulem}
\usepackage[utf8]{inputenc} 
\usepackage{graphicx} 
\usepackage[small]{caption}
\usetikzlibrary{arrows,calc}
\numberwithin{equation}{section}

\usepackage{xspace}
\usepackage{upgreek}

\usepackage{enumerate} 
\usepackage[shortlabels]{enumitem}
\usepackage{verbatim} 
\usepackage{color}
\usepackage{tabularx}

\usepackage[colorlinks]{hyperref} 

\usepackage[capitalise]{cleveref}

\usepackage[letterpaper]{geometry}
\usepackage[colorinlistoftodos]{todonotes}
\presetkeys{todonotes}{inline, color=green}{}

\newcommand{\Prob}{\mathbb{P}}

\newcommand{\R}{\mathbb{R}}
\newcommand{\E}{\mathbb{E}}

\newcommand{\fM}{\mathfrak{M}}

\newcommand{\cE}{\mathcal{E}}
\newcommand{\cF}{\mathcal{F}}
\newcommand{\cG}{\mathcal{G}}

\newcommand{\cU}{\mathcal{U}}

\newcommand{\sF}{\mathsf{F}}

\newcommand{\Rone}{\mathtt{r}_n}
\newcommand{\Rtwo}{\widehat{\mathtt{r}}_n}

\renewcommand{\d}{\mathrm{d}}
\renewcommand{\P}{\Prob}

\newcommand{\alt}{\mathrm{alt}}
\newcommand{\blt}{\mathrm{blt}}
\newcommand{\rainbow}{\mathrm{rbw}}
\newcommand{\knight}{\mathrm{knt}}
\newcommand{\corr}{\mathrm{corr}}
\newcommand{\dist}{\operatorname{dist}}
\newcommand{\tmix}{t_\textsc{mix}}

\newcommand{\tv}{{\textsc{tv}}}
\renewcommand{\u}{\mathsf{u}}
\renewcommand{\v}{\mathsf{v}}
\newcommand{\x}{\mathsf{x}}
\newcommand{\y}{\mathsf{y}}

\newcommand{\gKap}{\mathfrak{c}_0}
\newcommand{\gAlpha}{\upalpha}
\newcommand{\good}{\texttt{good}\xspace}

\newcommand{\autI}{\cA^{(1)}}
\newcommand{\autII}{\cA^{(2)}}

\newcommand\nleftrightsquigarrow{
\mathrel{\mkern6mu\not\mkern-6mu\leftrightsquigarrow}}

\newtheorem{theorem}{Theorem}[section]
\newtheorem*{theorem*}{Theorem}
\newtheorem{lemma}[theorem]{Lemma}
\newtheorem{claim}[theorem]{Claim}
\newtheorem{proposition}[theorem]{Proposition}

\newtheorem{fact}[theorem]{Fact}
\newtheorem{corollary}[theorem]{Corollary}

\theoremstyle{definition}{
\newtheorem{example}[theorem]{Example}

\newtheorem*{definition*}{Definition}

\newtheorem{question}[theorem]{Question}
\newtheorem*{question*}{Question}
\newtheorem*{example*}{Example}
\newtheorem*{examples*}{Examples}
\newtheorem{remark}[theorem]{Remark}
\newtheorem*{remark*}{Remark}

}

\renewcommand{\bar}[1]{\overline{#1}}

\newcommand{\Cov}{\operatorname{Cov}}
\newcommand{\Var}{\operatorname{Var}}
\newcommand{\Vol}{\operatorname{Vol}}
\newcommand{\one}{\mathbbm{1}}
\newcommand{\bfone}{\mathbf{1}}
\newcommand{\diam}{\operatorname{diam}}

\newcommand{\Z}{\mathbb{Z}}
\newcommand{\C}{\mathbb{C}}

\newcommand{\Cq}{\mathfrak{C}_q}

\newcommand{\cA}{\mathcal{A}}
\newcommand{\cR}{\mathcal{R}}

\renewcommand{\restriction}{\mathord{\upharpoonright}}
\renewcommand{\epsilon}{\varepsilon}

\makeatletter
\crefname{step}{Step}{Steps}
\crefname{question}{Question}{Questions}
\makeatother

\usepackage{xcolor}
\title{The noisy voter model with general initial conditions}

\author[P. Caddeo]{Patrizio Caddeo}
\address{P. Caddeo \hfill\break
Courant Institute\\ New York University\\
251 Mercer Street\\ New York, NY 10012, USA.}
\email{patrizio.caddeo@courant.nyu.edu}

\author[E. Lubetzky]{Eyal Lubetzky}
\address{E. Lubetzky \hfill\break
Courant Institute\\ New York University\\
251 Mercer Street\\ New York, NY 10012, USA.}
\email{eyal@courant.nyu.edu}

\begin{document}

\begin{abstract}
We study the noisy voter model with $q\geq 2$ states and noise probability $\theta$ on arbitrary bounded-degree $n$-vertex graphs $G$ with subexponential growth of balls (e.g., finite subsets of~$\Z^d$). Cox, Peres and Steif (2016) showed for the binary case $q=2$ (and a wider class of chains) that, when starting from a worst-case initial state, this Markov chain has total variation cutoff at $t_n=\frac1{2\theta}\log n$. 
The second author and Sly (2021) analyzed faster initial conditions for Glauber dynamics for the 1D Ising model, which is the noisy voter for $q=2$ and $G=\Z/n\Z$. They showed that the ``alternating'' initial state is the fastest one if $\theta\geq \frac23$, and conjectured that this holds for all values of the noise $\theta$.

Here we show that for every graph $G$ as above and all $\theta,q$ and initial states $\x_0$, the noisy voter model exhibits cutoff at an explicit function of the autocorrelation of the model started at~$\x_0$. Consequently, for $G=\Z/n\Z$ and $q=2$ (Glauber dynamics for the 1D Ising model), we
 confirm the conjecture of [LS21] that 
the alternating initial condition is asymptotically fastest for all $\theta$. Analogous results hold in $\Z_n^d$ for $q=2$ and all $d\geq 1$  (``checkerboard'' initial conditions are fastest) as well as for $d=1$ and all $q\geq 2$ (``rainbow'' initial conditions are fastest).
\end{abstract}

\maketitle

\section{Introduction}

The noisy voter model on a connected graph $G=(V,E)$, with  noise parameter $0<\theta<1$ and $q\geq 2$ states, is the Markov chain $(X_t)$ on the space $\Omega=\{0,\ldots,q-1\}^V$ (assigning one of $q$ different states, or opinions, to each vertex), evolving in the following way.
An independent rate-$1$ Poisson clock is attached to every $u\in V$. Once the clock at $u$ rings, the value of $X_t(u)$ is updated as follows: 
\begin{enumerate}[(i),wide=0pt, itemsep=0ex, topsep=0ex, ref=(\roman*)]
\item\label[step]{def:walk-step} with probability $1-\theta$, the new value is $X_t(v)$ for a uniformly chosen neighbor $v\sim u$;
\item\label[step]{def:kill-step} with probability $\theta$, the new value is uniformly chosen out of $\{0,\ldots,q-1\}$.
\end{enumerate}
Let $\mu_G$ be the stationary distribution of this Markov chain (ergodic but not necessarily reversible).

The classical voter model is the special case $\theta=0$, and $q=2$ (noiseless, binary), introduced in the 1970's independently by Clifford and Sudbury~\cite{CliffordSudbury73} and by Holley and Liggett~\cite{HolleyLiggett75} to model the evolution systems with two competing population types; see, e.g.,~\cite{LiggettBook99} for the rich literature on it. 
The noisy voter model was introduced by Granovsky and Madras~\cite{GranovskyMadras95} in 1995, where properties of the equilibrium $\mu_G$ were studied (cf.~the recent works~\cite{BHZ23,PymarRevera24} on features of the stationary measure). Our focus is on the asymptotic time it takes $(X_t)$ to converge to $\mu_G$ in total variation distance. 

Denote the total variation mixing time to within distance $\epsilon$ from $\mu_G$ by
\[ \tmix^{\x_0}(\epsilon) = \inf\left\{t \geq 0 \,:\; \left\|\P_{\x_0}(X_t\in\cdot)- \mu_G\right\|_\tv < \epsilon\right\}\quad\mbox{and}\quad \tmix(\epsilon) = \max_{\x_0}\tmix^{\x_0}(\epsilon)\,.\]
 In the context of a sequence of finite Markov chains $\{X_t^{(n)}\}$, the \emph{cutoff phenomenon}, discovered by Aldous and Diaconis in the 1980's~\cite{Aldous83,AldousDiaconis86,DiaconisShahshahani81,Diaconis96}, refers to the situation where there exists a sequence $t_n$ such that, for every fixed $0<\epsilon<1$, one has $\tmix(\epsilon) /t_n \to 1$ as $n\to\infty$. That is, a sharp transition along a $o(t_n)$ window occurs, whereby the distance from equilibrium drops from near $1$ to near $0$.
When cutoff occurs at $t_n$ for $\tmix^{\x_0}$---that is, starting from a sequence $\{\x_0^{(n)}\}$ of initial conditions---the asymptotic mixing time from $\x_0$ is well-defined as $(1+o(1))t_n$ independently of the choice of~$\epsilon$, and one can compare it to other sequences of initial configurations and characterize which is the fastest. 
We will establish cutoff for the noisy voter from any $\x_0$ on $G$ with subexponential growth of balls.

Consider the binary case $q=2$. Cox, Peres and Steif~\cite{CPS16} proved in 2016 that for every sequence of bounded-degree graphs $G$ on $n$ vertices and fixed noise parameter $0<\theta<1$, the noisy voter model has $\tmix(\epsilon) = (\frac1{2\theta}+o(1))\log n$ for every fixed $0<\epsilon<1$; that is, when started at a worst-case initial state, there is cutoff at $t_n=\frac1{2\theta}\log n$
(in fact~\cite{CPS16} allowed a wider family of rules for selecting the neighbor $v\sim u$ in the voting update of \cref{def:walk-step} above). Their lower bound on $\tmix$---the asymptotically worst initial condition $\x_0$---was realized by a monochromatic (all-$0$ or all-$1$) state.

In the special case of $q=2$ and $G=\Z/n\Z$ (the $n$-cycle), the noisy voter model is well-known to be equivalent to Glauber dynamics for the Ising model on the 1D torus at inverse-temperature $\beta$, via the correspondence\footnote{In an update of Glauber dynamics for the Ising model, if a site $x$ has $a$ neighbors with $+$ spins and $b$ neighbors with $-$ spins, the probability that the dynamics sets $x$ to $+$  is $(1+\tanh(\beta \sigma))/2$ for $\sigma=a-b$, hence the equivalence.} $\theta=1-\tanh(2\beta)$. (NB.\ deleting but a single edge from $G$ would break this equivalence, and make the stationary distribution of the noisy voter model be non-reversible.)
Thus, said result of~\cite{CPS16} recovers the 1D special case of the result~\cite{LubetzkySly13} of the second author and Sly, where it was shown that there is cutoff for the Ising model on the torus $(\Z/n\Z)^d$ in any dimension $d$, at high temperature (covering all $\beta$ at $d=1$ and all $\beta<\beta_c$ at $d=2$; its companion paper~\cite{LubetzkySly14} generalized the method to arbitrary bounded-degree graphs with sub-exponential growth rate). A later work~\cite{LubetzkySly16} by the same authors established cutoff for the Ising model on $(\Z/n\Z)^d$ for all $\beta<\beta_c$ in any dimension $d$, using a different approach dubbed there \emph{information percolation} (see also~\cite{LubetzkySly15,LubetzkySly17}) that carefully studies the space-time clusters of the nearest-neighbor interaction, rolled backward in time. (See \cref{fig:nvm-sim} for an illustration for noisy voter 
and \cref{sec:duality} for more details.) This opened the door to studying $\tmix^{\x_0}$ for non-worst-case initial conditions $\x_0$, where very little was previously known; e.g., it was shown there that for $d=1$, a $(1-o(1))$-fraction of the initial states $\x_0$ is asymptotically as slow as the worst-case one, whereas a random i.i.d.\ initial condition is asymptotically twice faster. 

\begin{figure}
\vspace{-0.1in}
\includegraphics[width=0.48\textwidth]{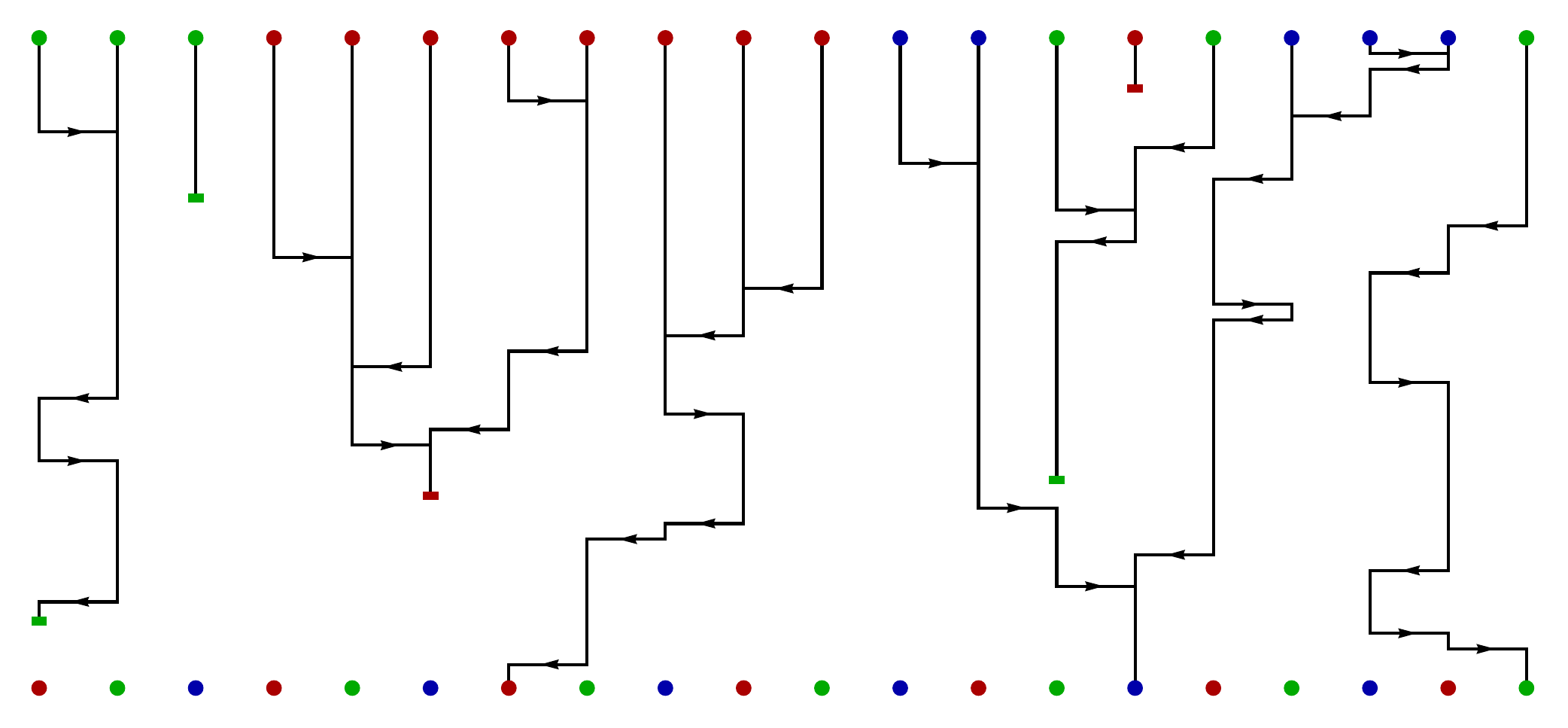}
\raisebox{-0.12in}{\includegraphics[width=0.4\textwidth]{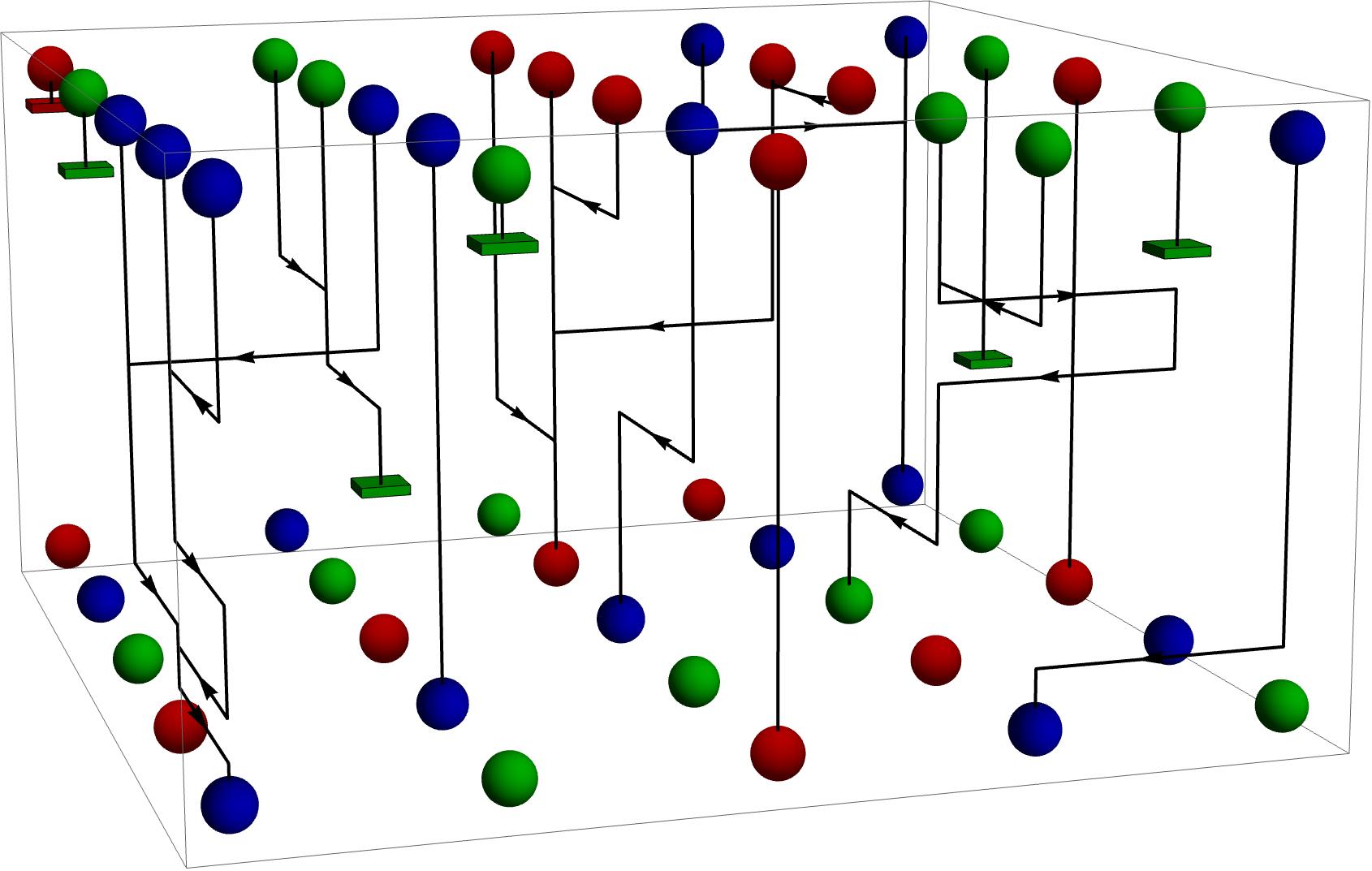}}
\vspace{-0.15in}
    \caption{Noisy voter model for $q=3$ on the lattice in dimensions 1 and 2, when revealing the update history (from the final (top) state $X_t$ back in time to the initial state $X_0=\x_0$).}
    \vspace{-0.15in}
    \label{fig:nvm-sim}
\end{figure}

Thereafter, the second author and Sly~\cite{LubetzkySly21} investigated, for $q=2$ and $G=\Z/n\Z$ (i.e., Glauber dynamics for the 1D Ising model at inverse-temperature $\beta$, equivalent to the noisy voter on~$\Z/n\Z$ with $q=2$ and $\theta=1-\tanh(2\beta)$), mixing from specific deterministic initial conditions, and showed:
\begin{enumerate}[(a),left=0pt, itemsep=0ex, topsep=0ex]
    \item started from the alternating $\x_\alt = (0,1,0,1,\ldots)$, there is cutoff at  $t_n=\max\{\tfrac1{4-2\theta},\tfrac1{4\theta}\}\log n$;
    \item started from the bi-alternating $\x_{\blt}=(0,0,1,1,\ldots)$, there is cutoff at $t_n=\max\{\tfrac12,\tfrac1{4\theta}\}\log n$ (matching the mixing time from $\x_\alt$ for $\theta\leq\frac12$ but strictly dominated by it for $\theta>\frac12$);
    \item the initial state $\x_\alt$ is asymptotically fastest at high enough temperature (specifically, if $\theta \geq \frac23$).
\end{enumerate}
The authors of~\cite{LubetzkySly21} conjectured that $x_\alt$ is asymptotically fastest  at every temperature (for all $\theta$). 

\subsection{Results}\label{sec:results}

We show that the noisy voter model for every fixed $q\geq 2$ exhibits cutoff under general starting states, in arbitrary families of graphs with a sub-exponential growth rate of balls (e.g., subsets of the lattice). Specifically, cutoff occurs at an explicit function of the autocorrelation of the model started at $\x_0$. Consequently, we confirm the conjecture of the second author and Sly that, in the special case of $G=\Z/n\Z$, the initial state $\x_\alt$ is asymptotically  fastest for all $\theta$ (or $\beta$). More generally, for the noisy voter model on $(\Z/n\Z)^d$ with $q=2$ and any $d\geq1$, the ``checkerboard'' initial conditions are fastest; for $q \geq 2$ and $d=1$, the ``rainbow'' initial conditions are fastest.

Let $G=(V,E)$ be a finite graph on $n$ vertices. 
Our assumption on its geometry is the following \textit{sub-exponential growth of balls} property: there exist constants $\gKap >0$ and $\gAlpha \in(0,1)$ such that 
\begin{equation}\label{eq:expansion}
|B_v(r)| \leq  \gKap\exp(r^{1-\gAlpha})\quad\mbox{for every $v\in V$ and $r\geq 1$}\,, 
\end{equation}
where $B_v(r)$ denotes the ball of radius $r$, centered at the vertex $v$, with respect to graph distance.
Consider the noisy voter model on such a graph $G$ with $q$ states and noise probability $\theta\in(0,1]$, and let $\x_0$ be an arbitrary initial configuration. Denoting
the degree of a vertex $v\in V$ by $d_G(v)$ and letting $\pi_G(v) = d_G(v)/ (2|E|)$ be the stationary measure of a simple random walk on $G$,
we define the \emph{autocorrelation} function of the noisy voter model 
\begin{equation}\label{def:autocorrelationL1}
\autI_t(\x_0) = \sum_v\pi_G(v)\Big(\P_{\x_0}(X_t(v)=\x_0(v))-\frac{1}{q}\Big)\,,
\end{equation}
as well as an $L^2$ analog of it which will turn out to be nothing but a time-rescaled version of $\autI_t$:
\begin{align}
\autII_t(\x_0) &
= \sum_{v \in V} \pi_G(v) \sum_{j=0}^{q-1} \Big(\P_{\x_0}\left(X_t(v) = j\right) - \frac{1}{q}\Big)^2\,.\label{def:autocorrelationL2}
\end{align}
Indeed, despite their seemingly different formulation, one has $\autII_t(\x_0) = \autI_{2t}(\x_0)$ (see \cref{eq:autI-autII}); hence, we will refer to both objects as the autocorrelation function (see \cref{sec:autocorrelation} for more details). A key quantity for the behavior of $\tmix^{\x_0}$ is the
 time $T_{\x_0}$ required for $\autII_t(\x_0)$ to drop below $\frac1n$:
\begin{equation}
T_{\x_0} = 
\frac12\inf\Big\{t \geq 0 \, : \; \autI_t(\x_0) \leq \frac{1}{n} \Big\}
=\inf\Big\{t \geq 0 \, : \; \autII_t(\x_0) \leq \frac{1}{n} \Big\}\,.\label{def:Tx0}
\end{equation}
Our main theorem states that waiting for the maximum between $T_{\x_0}$ and $\frac1{4\theta}\log n$---which will correspond to a quantity $T_{\corr}$ as explained later (see \cref{sec:proof-ideas})---is both necessary and sufficient for the model to be mixed. That is, $\max\{T_{\x_0},T_{\corr}\}$ gives the correct asymptotic mixing time on any family of graphs with limited expansion, under arbitrary initial conditions $\x_0$.
\begin{theorem}\label{thm:main}
Fix $\theta\in(0,1]$ and $q\geq2$. Let $G=(V,E)$ be a connected graph on $n$ vertices satisfying \cref{eq:expansion} for some $\gKap>0$ and $\gAlpha\in(0,1)$, and let $\x_0$ be an arbitrary initial condition. Then the $q$-state noisy voter model on $G$  
and noise probability $\theta$ started from $\x_0$ satisfies
\begin{equation}
\tmix^{\x_0}(\epsilon) = (1+
o(1)
)\max\Big\{T_{\x_0}, \frac{1}{4\theta}\log n\Big\}\,,\label{eq:mixing}
\end{equation}
where the $o(1)$-term is $ O\big(\log^{-\gAlpha/2} n\big)
$ with a constant 
%
depending on $\epsilon,\theta,q,\gKap,\gAlpha$ (but not $G$ or $\x_0$).

In particular, for any sequence of connected graphs $G^{(n)}$ with $n$ vertices satisfying \cref{eq:expansion} for the same $\gAlpha,\gKap$, and any sequence of initial conditions $\x_0^{(n)}$, the $q$-state noisy voter model on $G^{(n)}$ with noise probability $\theta$ has cutoff at $\max\{T_{\x_0^{(n)}}, \frac{1}{4\theta}\log n\}$.
\end{theorem}
\cref{thm:main} shows that a dynamical phase transition in the parameter $\theta$ can occur under certain initial conditions $\x_0$: we will see that in the case of $G=\Z_n^d$, if $\x_0$ is chosen by repeating a balanced pattern of the $q$ colors, then $T_{\x_0} \sim \frac1{2[(1-\lambda) + \lambda\theta]}\log n$ for some $\lambda<1$ (here and in what follows, we let $f_n\sim g_n$ denote $f_n=(1+o(1))g_n$). It follows that $\tmix$ changes from $\frac{1}{4\theta}\log n$ to $\frac{1}{2[(1-\lambda) + \lambda\theta]}\log n$ at some $\theta_0\in(0,1)$, implying a singularity in the analytical behavior of $\tmix^{\x_0}$ as a function of $\theta$. 

The full \cref{eq:mixing} was only known \cite{LubetzkySly21} for the 1D Ising model on $\Z_n$ with periodic boundary conditions and specific initial conditions $\x_0$. While \cite{LubetzkySly21} showed that $\tmix^{\x_0} \geq \frac{1-o(1)}{4-2\theta}\log n$ for all $\x_0$ (implying, via their analysis of the alternating initial condition $\x_\alt = (0,1,0,1,\ldots)\in\Z_n$ that it is asymptotically fastest for $\theta\geq \frac23$), the fact that one also has $\tmix^{\x_0} \geq \frac{1-o(1)}{4\theta}\log n$ for all $\x_0$ (implying that $\x_\alt$ is always asymptotically fastest) was left there as a conjecture. 
Here not only do we confirm this, but we extend the result to the noisy voter model on any bipartite graph with $n=|V|$ large satisfying \cref{eq:expansion} for some $\gAlpha,\gKap$. Indeed, if $G$ is bipartite with parts $V_1,V_2$, 
we can define 
\begin{equation}\label{eq:alternatingbipartite}
\x_{\alt}(v) := \one_{\{v\in V_1\}}
\end{equation}
(see \cref{fig:checkboard} for an illustration on $\Z^d$) and have the following.
\begin{figure}
\vspace{-0.05in}
  \includegraphics[width=0.35\textwidth]
  {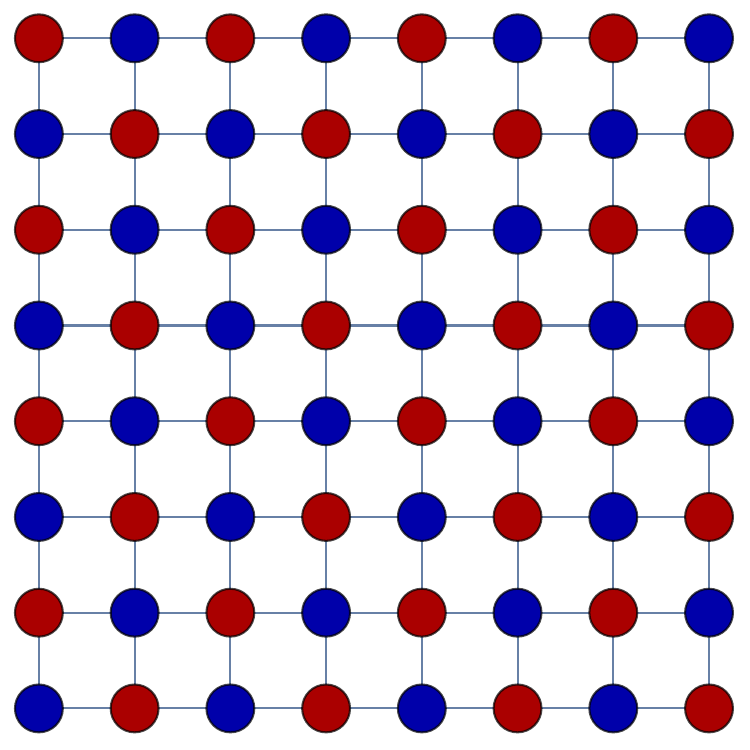}
\hspace{0.1in}
\raisebox{-0.45in}{
  \includegraphics[width=0.45\textwidth]
  {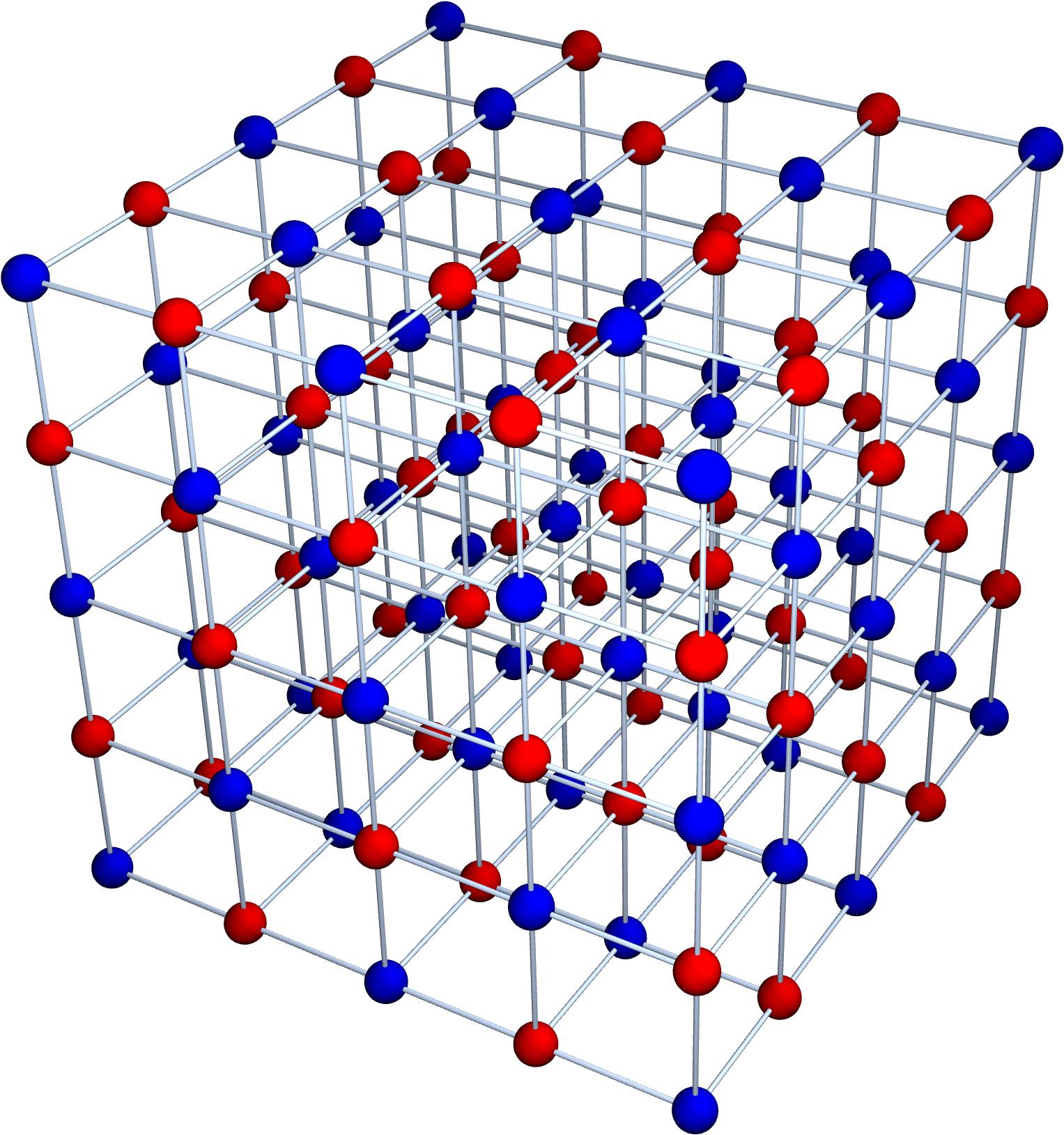}}

\vspace{0.1in}
  \includegraphics[width=0.6\textwidth]
  {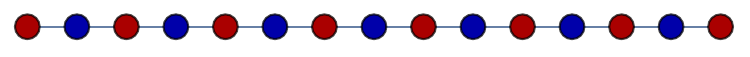}
\vspace{-0.1in}  
\caption{When $q=2$, alternating/checkerboard starting configurations are asymptotically the fastest in all dimensions.}
\label{fig:checkboard}
\vspace{-0.1in}
\end{figure}
\begin{corollary}\label{cor:bipartite}
Fix $\theta\in(0,1]$. Let $G^{(n)}$ be a sequence of bipartite graphs on $n$ vertices satisfying \cref{eq:expansion} for the same $\gAlpha,\gKap$. Let $\x_\alt = \x_{\alt}^{(n)}$ denote the alternating condition on $G^{(n)}$ as in \cref{eq:alternatingbipartite}. Then, for every fixed $0<\epsilon<1$, the binary noisy voter model on $G^{(n)}$ with noise $\theta$ satisfies
\begin{equation}
\tmix^{\x_{\alt}}(\epsilon) = \frac{1+o(1)}{\min\{4-2\theta,4\theta\}}\log n\,,\label{eq:mixingbipartite}
\end{equation}
where the $o(1)$-term is $ O\big(\log^{-\gAlpha/2} n\big)
$ with a constant depending on $\epsilon,\theta,\gAlpha,\gKap$ (but not on $G^{(n)}$).
Furthermore, if $\x_0=\x_0^{(n)}$ is any arbitrary sequence of initial conditions, we have
\[
\liminf_{n \to \infty}\frac{\tmix^{\x_0}(\epsilon)}{\tmix^{\x_{\alt}}(\epsilon)} \geq 1\,,
\]
thus $\x_{\alt}$ is the asymptotically fastest initial condition on $G^{(n)}$. In particular, for Glauber dynamics for the Ising model on $\Z_n$ (the case $G^{(n)}=\Z_n$) at any inverse temperature $\beta>0$, the alternating initial condition $\x_{\alt}$ is the asymptotically fastest initial condition. 
\end{corollary}

\Cref{cor:bipartite} shows that when $q=2$, alternating between the $2$ colors (when possible) gives the asymptotic fastest rate of convergence. In the case of the $d$-dimensional lattice $\Z_n^d$ for any $d\geq 1$, this results in ``checkerboard'' configurations being the fastest initial conditions.
However, when $q>2$, the answer can become more complicated. Consider again the $d$-dimensional lattice $G=\Z_n^d$: we can define a generalization of $\x_{\alt}$ by cycling over the $q$ colors in all directions, creating the \emph{rainbow} configuration (depicted in \cref{fig:2drainbowknight} on the left)
\begin{equation}\label{eq:rainbow}
\x_{\rainbow}(i_1,\ldots,i_d) = i_1+\ldots+i_d \pmod q\,, \quad (i_1,\ldots,i_d)\in\Z_n^d\,.
\end{equation}
It turns out that $\x_{\rainbow}$ does not always achieve the fastest rate. Indeed, we will see in \Cref{sec:applications} that
\[
T_{\x_{\rainbow}} \sim \frac{d}{2\big[1-(1-\theta)\cos(\frac{2\pi}{q})\big]}\log n\,,
\]
while for $d=2$ and $q=5$, if we define the ``knight'' configuration (depicted in \cref{fig:2drainbowknight} on the right)
\[
\x_{\knight}(i,j) = i+2j \pmod 5
\]
(taking after a chess knight move), then we will get
\[
T_{\x_{\knight}} \sim  
\frac{1}{1+(1-\theta)/4}\log n\,,
\]
and in particular $T_{\x_{\knight}} < (1-\delta) T_{\x_{\rainbow}}$ with $\delta=\frac{\sqrt5 (1-\theta)}{5-\theta} >0$ for all $\theta<1$ and every $n$ in this setting. The reason for this is that when $d>1$, there is more room to spread the colors while avoiding early repetitions (which $\x_\rainbow$ does not do effectively in $d\geq 2$, since we can encounter the same color we are currently at after making just two moves).

\begin{figure}
\centering
  \includegraphics[width=0.4\textwidth]{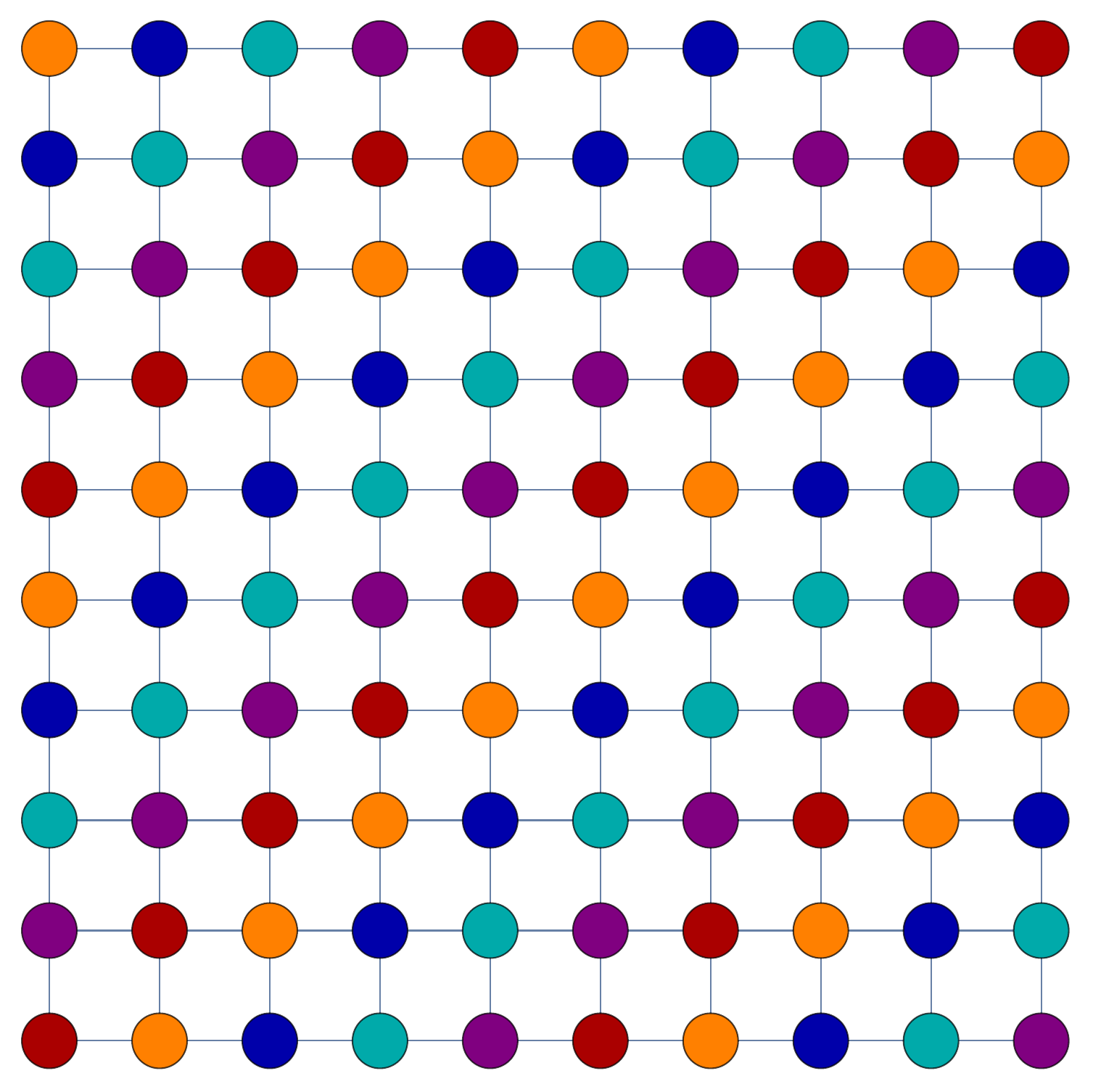}
\hspace{0.25in}  \includegraphics[width=0.4\textwidth]{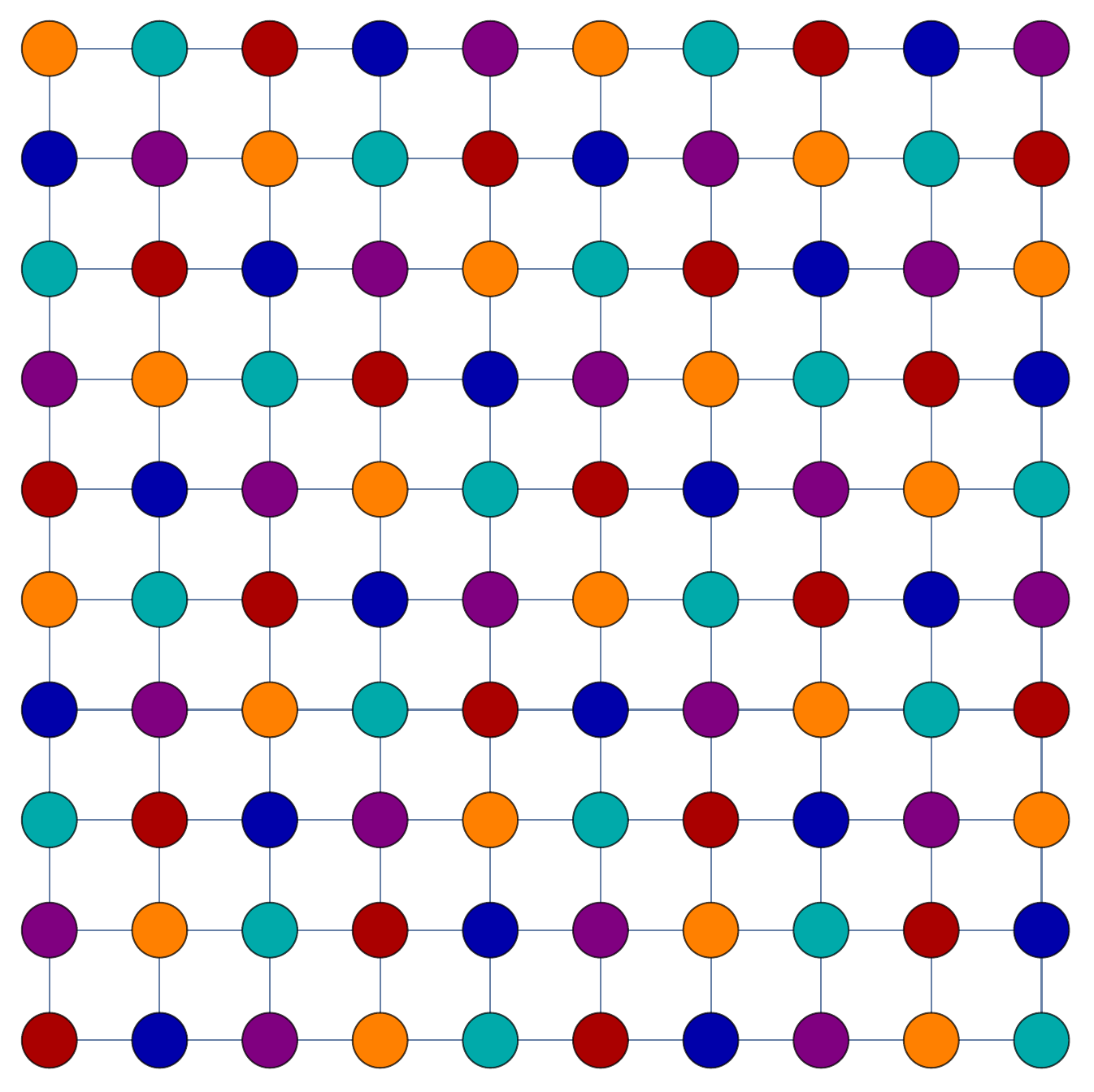}
\caption{In $\Z_n^2$ with $q=5$, starting from the rainbow configuration (left) leads to slower mixing compared to starting from the knight configuration (right).}
\label{fig:2drainbowknight}
\end{figure}
\begin{figure}
\centering
\includegraphics[width=0.55\textwidth]{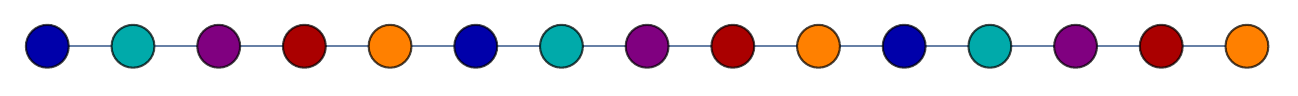}
\caption{Rainbow configuration on $\Z_n$ for $q=5$. On $\Z_n$, the rainbow initial condition is asymptotically the fastest for all $q \geq 2$ .}
\vspace{-0.1in}
\label{fig:rainbow}
\end{figure}

Nonetheless, we can show that when $d=1$, $\x_\rainbow=(0,1,\ldots,q-1,0,1,\ldots,q-1,0,1,\ldots)$ (depicted in \cref{fig:rainbow}) does indeed achieve the fastest possible rate (see also \cref{fig:Z-Tx0}), as we next formulate.

\begin{figure}
\vspace{-0.15in}
\centering
  \begin{tikzpicture}[font=\tiny]
  \node (fig1) at (-0.5,0) {  \includegraphics[width=0.4\textwidth]{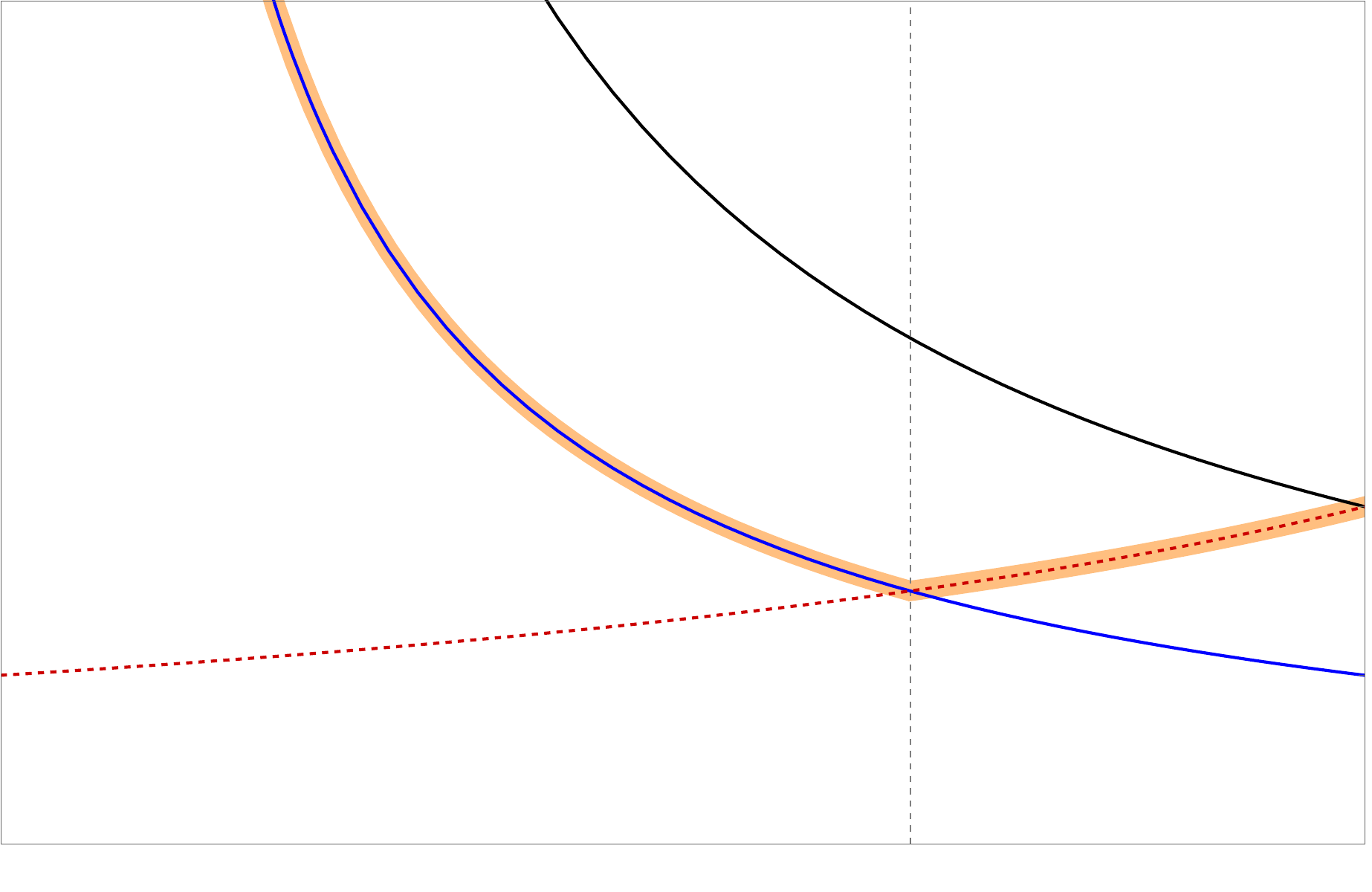}
  };
  \node at (-3.73,-2.1) {$0$};
  \node at (0.58,-2.1) {$\frac23$};
  \node at (2.75,-2.1) {$1$};
  \node at (3.2,-1.1) {$\frac14\log n$};
  \node at (3.2,-0.25) {$\frac12\log n$};
  
  \node (fig2) at (7,0) {  \includegraphics[width=0.4\textwidth]{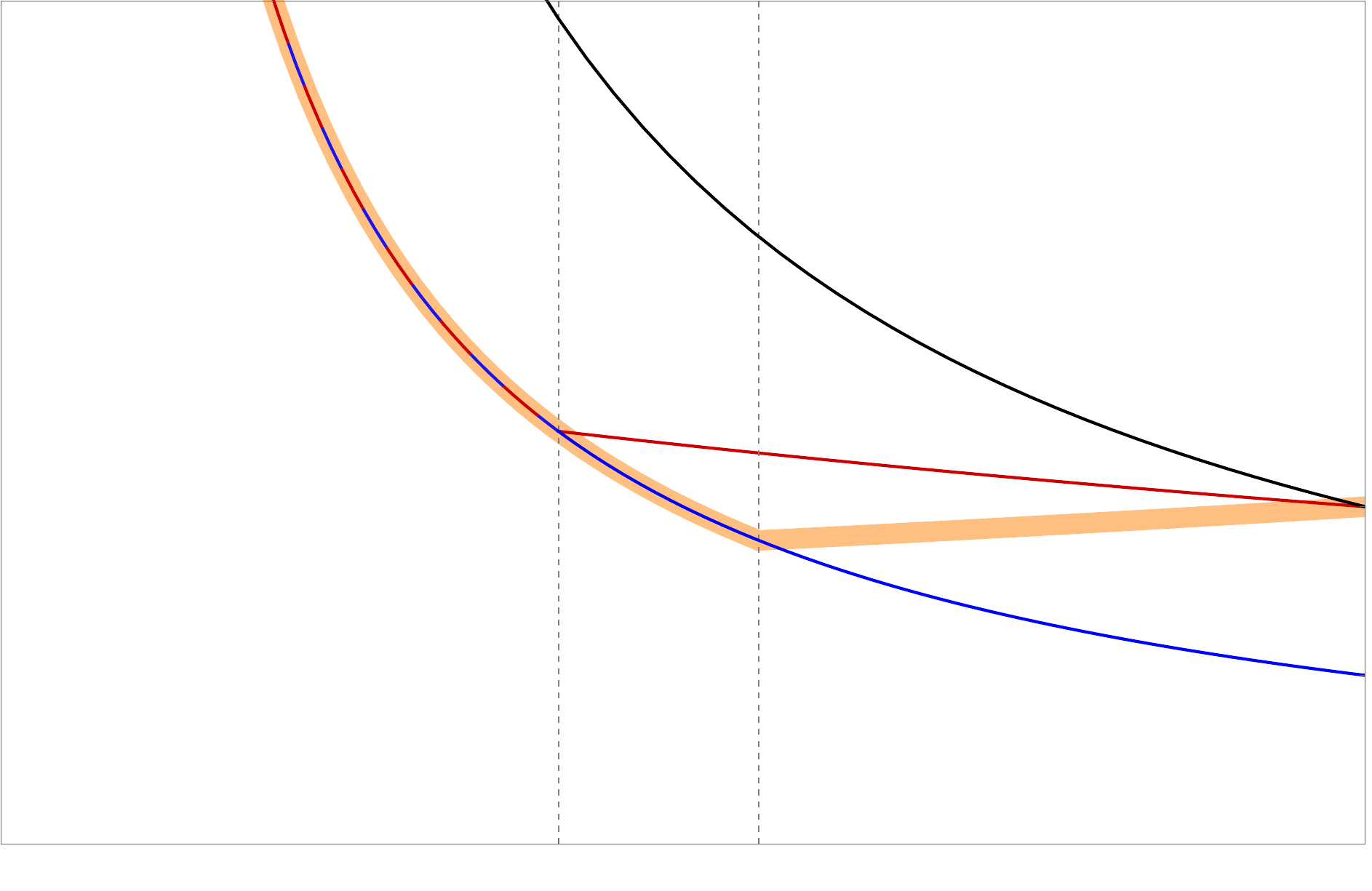}
  };

  \node at (3.73,-2.1) {$0$};
  \node at (6.35,-2.1) {$\frac{10-\sqrt{5}}{19}$};
  \node at (7.35,-2.1) {$\frac59$};
  \node at (10.25,-2.1) {$1$};
  \node at (10.7,-1.1) {$\frac14\log n$};
  \node at (10.7,-0.25) {$\frac12\log n$};
  
  \end{tikzpicture}
\vspace{-0.17in}
\caption{Mixing as a function of $\theta$ in $(\Z/n\Z)^d$. Left: $d=1$ and $q=2$ (previously known lower bound in red, $\x_\alt$ in orange). Right:  $d=2$ and $q=5$ ($\x_\rainbow$ in red, $\x_\knight$ in orange).}\label{fig:Z-Tx0}
\vspace{-0.1in}
\end{figure}

\begin{theorem}\label{thm:rainbow}
Let $G = \Z_n$. There exists $c>0$ such that, for every $q\geq 3$, every $\theta\in(0,1]$ and every initial state $\x_0$,
\begin{equation}\label{eq:rainbowlb}
\autII_t(\x_0) \geq 
\frac{c_0}{q^2}\, \autII_t(\x_\rainbow)\,.
\end{equation}
Consequently, $T_{\x_0} \geq T_{\x_\rainbow}+O(1)$, and thus $\x_\rainbow$ is the asymptotically fastest initial condition.
\end{theorem}

The results so far only pertained to deterministic initial states. The effect of a random initial condition was explored in \cite{LubetzkySly16} for the 1D Ising model: instead of starting from a deterministic state $\x_0$, one assigns an i.i.d.\ uniform color to each site. This means starting from the initial configuration $\cU$, the uniform distribution over $\Omega$, or equivalently replacing $\P_{\x_0}(X_t\in\cdot)$ with its annealed version $\P_\cU(X_t \in \cdot):=\frac1{|\Omega|}\sum_{\x_0\in\Omega}\P_{\x_0}(X_t\in\cdot)$. Denote by $\tmix^{\cU}(\epsilon)$ the mixing time of the chain started at $\cU$ (the time for $\P_\cU(X_t \in \cdot)$ to be within $\epsilon$ of the equilibrium measure $\mu_G$ in total variation distance).
It was shown in~\cite{LubetzkySly16} that, for any fixed $0<\epsilon<1$, Glauber dynamics for Ising on $(\Z/n\Z)$ has 
\[ \tmix^{\cU}(\epsilon) \sim \frac{1}{4\theta}\log n\]  (that is, the bound in \cref{eq:mixing} but without the autocorrelation term $T_{\x_0}$). We can now extend this result to the noisy voter model in the same setting as \Cref{thm:main}. Furthermore, as the lower bound $\tmix^{\x_0} \geq \frac{1-o(1)}{4\theta}\log n$ in \cref{eq:mixing} holds for all deterministic starting states, we can deduce that the uniform initial condition is asymptotically fastest compared to all deterministic starting states.

\begin{corollary}\label{cor:uniform}
Fix $\theta\in(0,1]$ and $q\geq2$. Let $G=(V,E)$ be a connected graph on $n$ vertices satisfying \eqref{eq:expansion} for some $\gKap>0$ and $0<\gAlpha<1$, and $\cU$ be the uniform distribution on $\{0,\ldots,q-1\}^V$. Then, for every fixed $0<\epsilon<1$, the $q$-state noisy voter model on $G$ with noise $\theta$ started from~$\cU$ has
\begin{equation}
\tmix^{\cU}(\epsilon) = 
\frac{1+o(1)}{4\theta}\log n\,,\label{eq:uniformmixing}
\end{equation}
where the $o(1)$-term is $ O\big(\log^{-\gAlpha/2} n\big)
$ with a constant 
%
depending on $\epsilon,\theta,q,\gKap,\gAlpha$ (but not $G$), and
\begin{equation}\label{eq:tmix-asymp} \tmix(\epsilon)=\frac{1+o(1)}{2\theta}\log n\,,\end{equation}
i.e., the asymptotic worst-case mixing time (attained by the monochromatic $\x_0=\bf{1}$) is twice slower.
Moreover, if $G^{(n)}$ is a sequence graphs satisfying \cref{eq:expansion} for the same $\gAlpha,\gKap$, then for any $\x_0=\x_0^{(n)}$,
\[
\liminf_{n \to \infty}\frac{\tmix^{\x_0}(\epsilon)}{\tmix^{\cU}(\epsilon)} \geq 1\,,
\]
i.e., asymptotically, no deterministic initial state $\x_0$ is faster than the uniform initial condition $\cU$.
\end{corollary}
(NB. \cref{eq:tmix-asymp} extends~\cite[Thm~1,~Eq.(1.2)]{CPS16}---which was valid for all bounded-degree graphs---from $q=2$ to all $q\geq 2$, albeit for graphs that further have subexponential growth of balls.) 

\subsection{Heuristics and main proof ideas}\label{sec:proof-ideas}
\subsubsection*{The autocorrelation time}
As mentioned above, in the binary case $q=2$, it was shown in~\cite{CPS16} that the worst-case mixing time for bounded-degree $n$-vertex graphs is $(1+o(1))\frac{1}{2\theta}\log n$, i.e., the same as the chain where the noise probability is $1$ (no vote-copying mechanism) and sites are independently refreshed at rate $\theta$ (a random walk on the hypercube).
This can be understood as follows: looking at the dynamics backward in time (as discussed in \cref{sec:duality}) shows that sites coalesce into clusters (arising from the copying mechanism) that receive an identical color. Under a monochromatic initial configuration $\x_0$, say $\x_0= \bfone$, a cluster roughly behaves like a single site, attaining its color by either getting a noisy update (with rate $\theta$) or surviving the noise to accept the color (in this case, $1$) assigned by $\x_0$. In graphs with bounded degree, said clusters typically have $O(1)$ size, thus one sees $c n$ clusters behaving as an infinite temperature (independent spins) single-site dynamics that evolve at rate $\theta$, leading to the aforementioned mixing time. 

The role of the autocorrelation function $\autII_t(\x_0)$ is to quantify the effect of a general initial condition $\x_0$ under the same assumption that the single-site marginals govern the mixing time. Indeed, the high temperature nature of the noisy voter model and the worst-case intuition above suggest that one could treat the latter as a product chain, and it is known that the total variation ($L^1$) distance of a product chain can be bounded, from below and above, by the sum of (squares of) the $L^2$ distances to equilibrium of its coordinates (see \cite[Prop.~7]{LubetzkySly14}). This motivates the introduction of the function $\autII_t(\x_0)$, which is nothing but that $L^2$ quantity for the noisy voter model (see \cref{def:autocorrelationL2}, where the sum over $j$ is simply $q \| \P_{\x_0}(X_t(v)\in\cdot)- \nu\|_{L^2(\nu)}^2$ for $\nu$ the uniform distribution on $\{0,\ldots,q-1\}$).
Hence, if we were to run $n$ i.i.d.\ copies of the chain $(X_t^v)_{v\in V}$, and take the value of each $v\in V$ from its own dedicated instance $X_t^v(v)$, we would have cutoff at $T_{\x_0}$. This approximation turns out to be valid in the true model for \textit{sufficiently large} $\theta$ (and all $q$), as confirmed by \cref{thm:main} for every initial condition $\x_0$.
However, as the next subsection shows, for some choices of $\x_0$ and small $\theta$, the function $\autII_t(\x_0)$ does not govern the mixing time.

\subsubsection*{The correlation time}
One key feature that the heuristics of single-site marginals fails to capture is, by its very definition, the correlations between sites in the equilibrium measure $\mu_G$. This aspect  can be measured by the \emph{coupling from the past} description of $\mu_G$; deferring details to \cref{sec:duality,sec:couplingfromthepast}, one can perfectly simulate $\mu_G$ via coalescing continuous-time random walks from every site, which move at rate $1-\theta$ and are killed at rate $\theta$, generating a single uniformly chosen color to their entire coalescent cluster as they die. 
Via this description, one can read the correlations in $\mu_G$ from the coalescent clusters (conditional on the partition to clusters, every cluster has a single uniform color,
and distinct ones are independent). 
When attempting to generate $\P_{\x_0}(X_t\in\cdot)$ via the same recipe, there is a time constraint: once we reach time $t$, the coalescing process ends (surviving clusters take on the colors prescribed by $\x_0$).

With that in mind, let $\big\{u \overset{>t}{\leftrightsquigarrow} v\big\}$ denote the event that the continuous-time walks from $u$ and $v$ (moving at rate $1-\theta$ and being killed at rate $\theta$) do not coalesce within time $t$ but do so afterwards. The above discussion indicates that the quantities $
\P\big(u \overset{>t}{\leftrightsquigarrow} v\big)
$
must be suitably small in order to avoid the situation whereby $u,v$ are more correlated under $\mu_G$ than under $\P_{\x_0}(X_t\in\cdot)$.
We will see that, for graphs with subexponential growth of balls as per \cref{eq:expansion}, averaging over nearest-neighbors is enough to asymptotically characterize this effect. Namely, let
\begin{equation}\label{def:taucorrelation}
T_{\corr} = \inf\Big\{ t\geq 0\,:\; \sum_{uv\in E}
\P\big(u \overset{>t}{\leftrightsquigarrow} v\big)
< \sqrt{n}\Big\}\,. 
\end{equation}

Clearly, $\P\big(u \overset{>t}{\leftrightsquigarrow} v\big)\leq e^{-2\theta t}$, simply by asking that the two random walks survive (no killing) separately in $[0,t]$. For our class of graphs, a matching lower bound holds up to sub-exponential corrections, thus $T_{\corr}$ will reduce (up to lower order terms) to $\frac{1}{4\theta}\log n$, as featured in \cref{thm:main}. 

The intuition as to why $T_{\corr}$ can govern the mixing time at \emph{sufficiently small} $\theta$ is the following.  Consider again the single-site marginals. When $\x_0$ is monochromatic, these are only affected by the noisy updates (at rate $\theta$), whereas in general, a random walk on the values of $\x_0$ (at rate $1-\theta$) can also contribute to mixing (e.g., consider 
a ``random-like'' $\x_0$, where the colors are arranged so as to resemble the effect of a noisy update). In view of their rates, the latter mechanism can have a stronger effect when $\theta$ is small, whence the single-site marginals can mix faster than the time it takes the correlations to be created. In that scenario, $T_{\x_0}<T_\corr$, which means that at time $T_{\x_0}$ the vertices are not sufficiently correlated and thus $T_{\x_0}$ is only an under-estimate of $\tmix$. \Cref{thm:main} confirms this intuition and shows that waiting for (the maximum between) $T_{\x_0}$ and $T_\corr$ is all that is needed for mixing to occur.

\subsubsection*{Proof ideas, upper bounds}

The work ~\cite{CPS16} proved an optimal worst-case upper bound for a large class of $2$-state noisy voter models, but one cannot hope for that approach to be applicable for general $\x_0$, as we next explain.
At the heart of that argument was a result from ~\cite{evans2000broadcasting} on noisy trees to view the backward cluster decomposition of the noisy voter model (explained in \cref{sec:duality}) as a projection of an equivalent chain on ``stringy trees,'' breaking the dependencies between sites. This elegant reduction provided a direct upper bound on the total variation distance via that of a chain with $n$ independent bits that can only change their initial value at rate $\theta$ (i.e., a random walk on the hypercube). At that point, $\x_0$ becomes irrelevant, as by symmetry every initial condition is equivalent to the monochromatic $\x_0=\bfone$ after said reduction.

Instead, our proof follows the approach of the later work~\cite{LubetzkySly21}, which analyzed the alternating and bi-alternating initial configurations for the 1D Ising model. The strategy there was to develop the backward dynamics of the Ising model for a short time, which suffices for the noise mechanism to decouple poly-logarithmic clusters of sites; then, the model was treated as a product chain on such smaller clusters, which could later be simplified to having size $\leq L$ for some fixed $L>0$. Thereafter, the $L^1$-$L^2$ reduction mentioned at the beginning of this section allowed the authors to 
conclude the proof, thanks to the regularity of the lattice and the initial states $\x_0$ considered there.

In our setting, the sub-exponential growth of balls assumption \cref{eq:expansion} allows us (with mostly minor changes) to make a similar decomposition as in ~\cite{LubetzkySly21}---reducing to sub-polynomial clusters, then to clusters of size $O(1)$---thus reduce the analysis of said clusters after an $L^1$-$L^2$ reduction. However, in the second part of the proof, the potential irregularity of $G$ and $\x_0$ and the general $q\geq 2$ form a nontrivial obstacle. The work~\cite{LubetzkySly21} showed that the mixing of these $O(1)$ clusters is governed by a competition between $1$-point and $2$-point correlations functions, which is easily resolved when the initial condition $\x_0$ is periodic; this is no longer the case for general $G$ and $\x_0$. 
Our strategy is to not treat these competitions separately per cluster, but rather combine them, reducing the cumulative $1$-point contributions to the autocorrelation function $t\to\autII_t(\x_0)$. The latter is shown to be completely monotone in \cref{sec:autocorrelation}, and that allows us to identify the global dominant term out of $\autII_t(\x_0)$ and $e^{-4\theta t}$, leading to an upper bound of either $T_{\x_0}$ or $\frac{1}{4\theta}\log n$, respectively.

\subsubsection*{Proof ideas, lower bounds}

The $\tmix^{\x_0} \gtrsim T_{\x_0}$ part of \cref{eq:mixing} is achieved by taking a distinguishing statistics that yields an expected difference of $\autII_t(\x_0)$ under $\P_{\x_0}^t$ and $\mu$. 

The $\tmix^{\x_0} \gtrsim \frac{1}{4\theta} \log n$ part of \cref{eq:mixing} is more complicated. In~\cite{LubetzkySly21}, the analogous bound for the alternating and bi-alternating conditions for $1$D Ising was established using the Hamiltonian of the model $\sum_i\sigma(i)\sigma(i+1)$ as a distinguishing statistics. The proof relied on the fact that the alternating configuration $x_\alt$ was a minimizer of the Hamiltonian, whereas general $\x_0$ can lead to both smaller and larger expected values of the Hamiltonian at finite time $t$ compared to equilibrium. Thus, this approach did not generalize to arbitrary $\x_0$, and it was left as a conjecture that the bound would hold more generally. Here we handle this lower bound in terms of site correlations, rather than the analogous spin agreements or disagreements measured by the Hamiltonian:
\begin{enumerate}
\item We consider a distinguishing statistics that measures the covariance between adjacent sites, and show that the latter are always more correlated at equilibrium than at time $t$, with a discrepancy of at least $\P\big(u \overset{>t}{\leftrightsquigarrow} v\big)$, uniformly over $\x_0$;
\item Using a result from ~\cite{OveisTrevisan12} that bounds the escape probability of random walk from a set in terms of its conductance, we show that under the sub-exponential growth of balls assumption \cref{eq:expansion} the probabilities above are equal to $e^{-2\theta t+o(t)}$ for $t \sim \log n$.
\end{enumerate}

\subsection{Organization of the paper}
In \Cref{sec:autocorrelation}, we prove properties of the two autocorrelation functions $\autI_t(\x_0),\autII_t(\x_0)$ that will be needed throughout the proofs of 
the main theorems. In \Cref{sec:applications}, we discuss specific applications of \Cref{thm:main} on $\Z_n^d$, as well as prove \Cref{thm:rainbow}. The proofs of \Cref{thm:main,cor:uniform} will be split between \Cref{sec:upperbound}, where we prove the upper bounds, and \Cref{sec:lowerbounds}, where we prove the lower bounds. Finally, \Cref{sec:open-prob} is devoted to concluding remarks on the new results, and open problems on generalizations / refinements thereof.

\section{Properties of autocorrelation in the noisy voter model}\label{sec:autocorrelation}

\subsection{Notation and preliminaries}

We refer to the spins/opinions/states at each vertex as \emph{colors}.
So far we denoted the $q$ colors by $\{0,\ldots,q-1\}$. In \Cref{sec:autocorrelation,sec:lowerbounds} it will be convenient to use as labels the $q$-th roots of unity:
\begin{equation}\label{eq:Cq-def}
\Cq = \left\{\omega_j = e^{2\pi i j/q} \,:\; \ j =0,\ldots,q-1\right\}\,.
\end{equation}

Given a graph $G=(V,E)$, a deterministic color configurations on $G$ will be denoted as $\x$, or $\x_0$ when we wish to emphasize that the configuration should be thought of as an initial condition. In contrast, states of the noisy voter chain at time $t$ will be denoted $X_t$ when starting from a deterministic initial condition $\x_0$, and we will often use $Y_t$ when starting from $\mu_G$, the stationary measure of the noisy voter model on $G$. We also use $Y$ to indicate a single sample from $\mu_G$.

\subsubsection{Duality with coalescing random walks}\label{sec:duality}
We now explain the classical duality between the (noisy) voter model and a system of (killed) coalescing random walks (cf.~\cite{GranovskyMadras95}).
Fix $t_0>0$, and view the dynamics resulting in the state $X_{t_0}$ as a space-time slab with $\x_0$ at the bottom and $X_{t_0}$ at the top. Moving in space (on the graph $G$) will be dubbed \emph{horizontal}, and moving in time will be dubbed \emph{vertical}.
The following description constructs $X_{t_0}=\{X_{t_0}(v)\}_{v \in V}$ by looking at the dynamics top-to-bottom (``backward in time''), rather than bottom-to-top (``forward in time'').

Let us start with a single vertex $v\in V$. We can reconstruct the value $X_{t_0}(v)$ of the state of $v$ at time $t_0$ as follows: starting from $v_0=v$ at the top, move down (backward in time) \emph{vertically} until the largest time $t_1<t_0$ where the $v_0$ received an update in the (forward) dynamics. (If no such update exists, we conclude that $X_{t_0}(v)=\x_0(v)$.)
If said update was noisy (which occurs with probability $\theta$) and resulted in a new color $\omega$, we conclude that $X_{t_0}(v)=X_{t_1}(v)=\omega$ and we can stop; if instead the update was a vote-copying one (with probability $1-\theta$), then $X_t(v)$ must be equal to $X_{t_1}(v_1)$, where $v_1$ is the neighbor of $v_0$ whose color was copied. We register this as a \emph{horizontal} move from $v_0$ to $v_1$, and proceed to recover $X_{t_1}(v_1)$ recursively (by examining the largest $t_2<t_1$ where $v_1$ was updated, and so forth).
Letting $s_0=0$ and $s_i = t_i - t_{i-1}$ ($i\geq 1$), the process defined by $Z^{v}_s = v_i$ for $s\in [s_i,s_{i+1})$  is a  continuous-time simple random walk on $G$ that starts at $v$, moves at rate $1-\theta$ and gets killed at rate $\theta$. Hence, if we let $\omega$ be a uniform color from $\{0,\ldots,q-1\}$ independent of $Z^v$, the value $X_t(v)$ can be coupled so that
\[
X_{t_0}(v) = \begin{cases}
\omega & \text{if $Z_s^v$ gets killed before time $t_0$}\\
\x_0(Z^v_{t_0}) & \text{if $Z_s^v$ survives to time $t_0$}
\end{cases}\,.
\]
(Note that the probability that $Z_{s}^v$ survives the the bottom is $e^{-\theta t_0}$.) 
All is left is to describe how the different walks $\{Z^v\}_{v\in V}$ interact with each other to determine the joint law of the $\{X_{t_0}(v)\}_{v \in V}$: for any two vertices $v,w$, as long as $Z_s^v$ and $Z_s^w$ do not meet each other, their steps (or killings) are due to updates occurring in the histories of distinct vertices and thus are independent of each other. However, if $Z_s^v$ and $Z_s^w$ meet for the first time at time $\tau$ (time $t_0-\tau$ in the forward dynamics), the process described above forces them to go through the exact same path (or be killed simultaneously) from $s=\tau$ until time $s=t_0$ (time $0$ in the slab). Hence, if two walks $Z^v$ and $Z^w$ ever meet, they proceed together from that point onward and thus \textit{coalesce into a single walk}. In particular, if the walks from $v$ and $w$ coalesce before time $s=t_0$ (that is, before reaching the bottom), then $v$ and $w$ inevitably obtain the same color (i.e., $X_{t_0}(v)=X_{t_0}(w)$). This holds in the same way for any number of vertices, meaning that any number of walks can coalesce into a single one, and the vertices such walks started from all receive the same color (a uniform one if the single walk resulted from the merging encounters a noisy update before reaching the bottom, or the value of $\x_0$ at the final location of the walk otherwise). See \Cref{fig:nvm-sim} for an illustration of this dual description.

\subsubsection{Dependence on the initial condition and coupling from the past} \label{sec:couplingfromthepast}

The duality explained above shows that $X_t$ depends on the initial condition $\x_0$ through the final locations of all the surviving walks (and their merging histories). In particular, if the walk $Z_s^v$ started from a vertex $v$ gets killed before time $t$ (i.e., before reaching the bottom), then the value $X_t(v)$ (as well as the values $X_t(w_i)$ at time $t$ of all the vertices $w_i$ whose walks $Z_s^{w_i}$ merged with $Z_s^v$ before the killing) is \textit{independent} of the initial condition $\x_0$. As a consequence, if all the walks started from $V$ die out (with or without coalescing with the other walks) before reaching the bottom, then $X_t$ becomes independent of $\x_0$.

This observation gives us a recipe to perfectly simulate the stationary measure $\mu_G$, a special case of the  \textit{coupling from the past} (CFTP) method (see \cite[\S25]{LevinPeres17}): take the walks $\{Z^v_s\}_{v \in V}$ defined above (that is, coalescing random walks with killing, running backwards in time from all vertices), but let them run until killed (without stopping them at time $t$ whence they would reach the bottom slab, hitting $\x_0$). With probability $1$, all of them will have died out at some point in the past (either coalescing with other walks or not). Then, the resulting $Y=\{Y(v)\}_{v \in V}$ obtained by assigning a uniform color to each cluster (each group of vertices whose associated random walks coalesced before dying) will be a perfect sample from $\mu_G$. (From the perspective of CFTP, the  coalescing random walk representation of the noisy voter model defines a grand coupling of the chains from all initial states, and if $t-T$ is a time where all walks coalesce, so do all the chains at time $t$.)

This description of $\mu_G$ allows us to couple $\P_{\x_0}(X_t\in\cdot)$ to $\mu_G$ by simply running the same coalescing random walks backward dynamics from time $t$, stopping at time $0$ to realize $X_t$ from $\x_0$ and then continuing back in the past to realize $Y\sim \mu_G$. In doing so, $X_t$ and $Y$ will be identical at the vertices whose associated random walks have died before reaching time $0$, but possibly not at the vertices $v$ that reached it, since in that case $X_t(v)$ takes the $\x_0$ value at the final location of $Z^v_s$ while $Y(v)$ will take an independent uniform value. As observed in~\cite{LubetzkySly16}, when starting from the uniform initial state $\cU$, the value of $X_t(v)$ can be coupled to $Y(v)$ even if $Z^v_s$ survived to time $0$, so long as it does not get to coalesce with another particle before being killed (one can simply match the value of $\cU$ to that of $Y$); this governed the behavior of $\tmix^\cU$ in Glauber dynamics for the  Ising model on $\Z/n\Z$, and will play the same role here for every $q\geq 2$ and every $G$ as per \cref{cor:uniform}.

\subsection{Alternative representations of autocorrelation} \label{sec:alternativerepresentations}

Let $G=(V,E)$ be a finite connected graph, and let $\pi_G$ be the stationary distribution of simple random walk on $G$, i.e.,
$\pi_G(v) = \frac{d_G(v)}{2|E|}$ for $v\in V$. We consider the $L^2$-inner product space on $\C^V$ with respect to $\pi_G$, via the Hermitian inner product
\begin{equation}\label{eq:inner-prod}
\langle f, g \rangle_{\pi_G} := \sum_{v \in V} f(v)\bar{g(v)}\pi_G(v)\qquad\mbox{for $f,g: V \to \mathbb{C}$}\,.
\end{equation}
Given $\x \in \Cq^V$, we define its $k$-th power  $\x^k = (\x(v)^k)_{v \in V}$ as the configuration obtained by taking $k$-powers component-wise (NB.\ $\Cq^V$ is closed under this operation, as $\Cq$ is a multiplicative group).

\begin{proposition}\label{prop:autocorrelation}
Let $G=(V,E)$ be a connected graph on $n$ vertices. Fix an initial state $\x_0 \in \Cq^{V}$. Let $\psi_1,\ldots,\psi_{n}$ be an $L^2(\pi_G)$-orthonormal basis of eigenfunctions of the transition matrix $P$ of the discrete-time simple random walk on~$G$, and let $\lambda_1,\ldots,\lambda_{n}$ be their corresponding eigenvalues. Further define
\begin{equation}\label{eq:gamma-def} \gamma_l = 1-(1-\theta)\lambda_l
\qquad\mbox{for $l=1,\ldots,n$}\,.
\end{equation}
Then, for every $t>0$,
\begin{align}
\autII_t(\x_0) = \frac1q \sum_{k=1}^{q-1} \sum_{v \in V} \pi_G(v)\big|\E_{\x_0}[X_t(v)^k]\big|^2
&= \frac1q \sum_{k=1}^{q-1} \sum_{l=1}^{n} |\langle \x_0^k, \psi_l\rangle_{\pi_G}|^2 e^{-2 \gamma_l t}\,.
\label{eq:At2-rep}
\end{align}
Consequently, for any $s \geq 0$, we have
\begin{align}
e^{-(4-2\theta)s}  \leq &\frac{\autII_{t+s}(\x_0)}{\autII_t(\x_0)} \leq e^{-2\theta s} \,,\label{eq:A2submultiplicativity1}
\end{align}
and in particular $\frac{q-1}q e^{-(4-2\theta)t} \leq \autII_t (\x_0) \leq \frac{q-1}q e^{-2\theta t}$.
\end{proposition}
\begin{proof}
Recalling the definition of $\autII_t$ from \cref{def:autocorrelationL2}, one has
\begin{align}
\autII_t(\x_0)
= \sum_{v \in V}\pi_G(v)\Big(\big(\sum_{\omega \in \Cq}\P_{\x_0}\left(X_t(v) = \omega\right)^2\big) - \frac{1}{q}\Big)\,.
\label{eq:A2-2}
\end{align}
Note that if $\tilde X_t$ denotes an i.i.d.\ copy of $X_t$ (started from the same $\x_0$), we have
\[
\sum_{\omega \in \Cq}\P_{\x_0}\left(X_t(v) = \omega\right)^2 = \sum_{\omega \in \Cq}\P_{\x_0}\left(X_t(v) = \tilde X_t(v) =  \omega\right) = \P_{\x_0,\x_0}(X_t(v) = \tilde X_t(v))\,.
\]
Since $\sum_{k=0}^{q-1} \omega^k = q\one_{\{\omega = 1\}}$ for $\omega\in\Cq$, taking
$\omega = X_t(v)\tilde X_t(v)^{-1} = X_t(v) \overline{\tilde X_t(v)}$ (recall $|\tilde X_t(v)|=1$) shows that
\[
\P_{\x_0,\x_0}(X_t(v) = \tilde X_t(v)) = \frac{1}{q}\sum_{k=0}^{q-1} \E[X_t(v)^k \overline{\tilde X_t(v)^{k}}] = \frac{1}{q}\sum_{k=0}^{q-1} \big|\E[X_t(v)^k]\big|^2\,.
\]
Combining the last two displays with \cref{eq:A2-2}, the $k=0$ summand cancels the $-\frac1q$ term, yielding
\[
\autII_t(\x_0) = \frac1q \sum_{k=1}^{q-1} \sum_{v \in V} \pi_G(v)\big|\E_{\x_0}[X_t(v)^k]\big|^2\,.
\]
Recalling the coalescing random walks representation of \cref{sec:duality}, on the event that the continuous-time random walk from $v$ is killed before reaching time $0$, the contribution to $\E_{\x_0}X_t(v)^k$ is $\frac1q \sum_{\omega\in\Cq}\omega^k = 0$. Thus, if $W_t$ denotes a continuous-time simple random walk moving at rate $1-\theta$ (and without killing), and $H_t$ the semigroup associated to such a walk with rate $1$, then
\[
\big|\E_{\x_0}[X_t(v)^k]\big|^2 = e^{-2\theta t}\big|\E_{v}[\x_0(W_t)^k]\big|^2 = e^{-2\theta t} (H_{(1-\theta)t}\x_0^k)(v)\overline{(H_{(1-\theta)t}\x_0^k)(v)}\,,
\]
which, when combined with the last display, gives
\[
\autII_t(\x_0) = \frac1q e^{-2\theta t}\sum_{k=1}^{q-1} \langle H_{(1-\theta)t} \x_0^k, H_{(1-\theta)t} \x_0^k\rangle_{\pi_G}\,.
\]
Recall that $\psi_l$ is an eigenfunction of the discrete-time transition matrix $P$ with eigenvalue $\lambda_l$; hence, it is an eigenfunction for $H_t = e^{-t(I-P)}$ with eigenvalue $e^{-(1-\lambda_l)t}$. Writing $\x_0^k = \sum_{l=1}^{n}\langle \x_0^k, \psi_l\rangle_{\pi_G}\psi_l$ via the orthonormal basis $\{\psi_l\}$ in $L^2(\pi_G)$, we arrive at
\[
\autII_t(\x_0) = \frac1q \sum_{k=1}^{q-1}\sum_{l=1}^{n}|\langle \x_0^k, \psi_l\rangle_{\pi_G}|^2 e^{-2\theta t-2(1 -\lambda_l)(1-\theta)t} = \frac1q \sum_{k=1}^{q-1} \sum_{l=1}^{n} |\langle \x_0^k, \psi_l\rangle_{\pi_G}|^2 e^{-2(1 -(1-\theta)\lambda_l)t}\,,
\]
establishing \cref{eq:At2-rep}.
The inequalities in \cref{eq:A2submultiplicativity1} readily follow by noting that $e^{-2 \gamma_l s}$ is minimized at $\lambda_l=-1$ (whence $\gamma_l = 2-\theta$) and maximized at $\lambda_l=1$ (whence $\gamma_l=\theta$). Finally, the upper and lower bounds on $\autII_t(\x_0)$ follow from \cref{eq:A2submultiplicativity1} (applying it for $s=t_0$ and $t=0$ to bound $\autII_{t_0}(\x_0)$ using that $\autII_0(\x_0)=(q-1)/q$ by definition).
\end{proof}

Two simple consequences of \Cref{prop:autocorrelation} are the following.
\begin{remark}[$T_{\x_0}$ maximized at the monochromatic $\bfone$]\label{rem:T-1}
As the monochromatic initial state $\bfone$ is an eigenfunction of the discrete-time random walk with eigenvalue $\lambda=1$, the identity in \cref{eq:At2-rep} (noting $\bfone^k = \bfone$ for all $k=1,\ldots,q-1$), along with the bounds on $\autII_t(\x_0)$ in \cref{prop:autocorrelation},
show
\[ \autII_{t}(\bfone) = \frac{q-1}q e^{-2\theta t} \geq \autII_t(\x_0)\qquad\mbox{for every $\x_0$ and $t>0$}\,;\]
hence, by definition, $T_{\x_0} \leq T_{\bfone}\sim \frac1{2\theta}\log n$ holds for every $\x_0$. 
\end{remark}

\begin{remark}[$T_{\x_0}$ minimized at $\x_\alt$ for a bipartite $G$ and $q=2$] \label{rem:T-xalt}
Since $G$ is bipartite, $\x_\alt$ is an eigenfunction of the discrete-time random walk with eigenvalue $\lambda=-1$, thus \Cref{prop:autocorrelation} shows
\[
 \autII_{t}(\x_{\alt}) = \frac{1}2 e^{-(4-2\theta)t} \leq \autII_t(\x_0)\quad\mbox{for every $\x_0$ and $t>0$}\,;\]
hence, by definition, $T_{\x_0} \geq T_{\x_{\alt}} \sim \frac1{4-2\theta}\log n$ holds for every $\x_0$. 
\end{remark}

Plugging the lower bound on $T_{\x_0}$ from \Cref{rem:T-xalt} in the asymptotics for $\tmix^{\x_0}$ given in \Cref{thm:main}  establishes \Cref{cor:bipartite}.
Combining the upper bound on $T_{\x_0}$ from \Cref{rem:T-1} with \Cref{thm:main} implies \cref{eq:tmix-asymp} in \cref{cor:uniform} for all $q\geq 3$ (the case $q=2$ was established in \cite{CPS16}). 

\subsection{Eigenfunctions of the noisy voter model}\label{sec:eigenfunctions}
As in \cref{prop:autocorrelation}, let  $\psi_1,\ldots,\psi_n$ be an $L^2(\pi_G)$-orthonormal basis of eigenfunctions of the ($1$-step) transition matrix $P$ of discrete-time simple random walk on $G$ with corresponding eigenvalues $\lambda_1,\ldots,\lambda_n$, and let $\gamma_l = 1-(1-\theta)\lambda_l$. For every $l=1,\ldots,n$ and $k=1,\ldots,q-1$, define a function $\Psi_l^{(k)}:\Cq^V \to \mathbb{C}$ by
\begin{equation}\label{eq:Psi-def}
\Psi_l^{(k)}(\x) = \langle \x^k, \psi_l\rangle_{\pi_G} \,.
\end{equation}
With this notation, \cref{eq:At2-rep} reads
\begin{equation}\label{eq:At2-rep-2}
\autII_t(\x_0) = \frac1q \sum_{l=1}^n \sum_{k=1}^{q-1}|\Psi_l^{(k)}(\x_0)|^2
 e^{-2\gamma_l t}\,.
\end{equation}
While the noisy voter model is, in general, a nonreversible Markov chain, the following result shows that $e^{-\lambda_l}$,  inherited from the spectrum of the random walk on $G$, appear as eigenvalues for it.
\begin{proposition}\label{prop:eigenfunc}
For each $l=1,\ldots,n$ and $k=1,\ldots,q-1$, the function $\Psi_l^{(k)}$ is an eigenfunction of the noisy voter model, with eigenvalue $1-\gamma_l/n$ for the discrete-time chain (each step updates a uniformly chosen vertex), or $e^{-\gamma_l}$ for the continuous-time model (each vertex is updated at rate $1$).
\end{proposition}
\begin{proof}
Let $l\in\{1,\ldots,n\}$ and $k\in\{1,\ldots,q-1\}$, and first consider the discrete-time noisy voter model, letting $\mathscr{P}$ be the $1$-step transition matrix of this chain.
If $\x \in\C_q^V$ is any initial state, then
\begin{align*}
(\mathscr{P} \Psi_l^{(k)})(\x) 
&= \E_\x\big[\Psi_l^{(k)}(X_1)\big] =
 \big\langle \E_\x\big[ X_1^k\big],\psi_l\big\rangle_{\pi_G}
= \Big\langle \Big(1-\frac{1}{n}\Big)\x^k+\frac{1-\theta}{n}P\x^k, \psi_l\Big\rangle_{\pi_G}\,,
\end{align*} 
where the last equality used that, in each coordinate $v\in V$ of the vector $\E_\x[X_1^k]$, the following holds:
(a)~with probability $1-\frac1n$, a different vertex is selected for the update, whence
$X_1(v)=\x(v)$; (b)~with probability $\frac{\theta}n$, a noisy update occurs at $v$, contributing $\frac1q \sum_{\omega\in\Cq}\omega^k = 0$ to the expectation; (c)~with probability $\frac{1-\theta}n$, a vote-copying update occurs at $v$, contributing $(P \x^k)(v)$ to the expectation. Recalling that $P$ is real symmetric (thus Hermitian), we have that
\[ \big\langle P\x^k,\psi_l\big\rangle_{\pi_G} =
\big\langle \x^k,P\psi_l\big\rangle_{\pi_G}
= \lambda_l \big\langle \x^k,\psi_l\big\rangle_{\pi_G}
\,,
\]
and combining the last two displays now shows that
\[ 
(\mathscr{P} \Psi_l^{(k)})(\x) = 
\Big(1-\frac{1}{n} + \frac{(1-\theta)\lambda_l}n\Big)\langle \x^k, \psi_l\rangle_{\pi_G} = \Big(1-\frac{\gamma_l}n\Big) \Psi_l^{(k)}(\x)\,.
\]
For the continuous-time noisy voter model, where the updates occur at rate $1$ per vertex, the associated semigroup is $\mathscr{H}_t = \exp[-(I-\mathscr{P})nt]$, thus $\Psi_l^{(k)}$ is an eigenvector with eigenvalue 
\[ \exp\Big[-\Big(1-\big(1-\frac{\gamma_l}n\big)\Big)n t\Big] = \exp\left(-\gamma_l t\right)\,,\]
as claimed.
\end{proof}
Note that, while the eigenfunctions $\{\psi_l\}$ are $L^2(\pi_G)$-normalized, the eigenfunctions $\{\Psi_l^{(k)}\}$ are not normalized in $L^2(\mu_G)$ (their variance under $\mu_G$ is not $1$). The following gives an exact expression for the variance of $\Psi_l^{(k)}$ and shows that when the vertices in $G$ have bounded degrees (and more generally, if the ratio of maximal and minimal degree is $O(1)$), all these variances have order $\frac1n$.
\begin{proposition}\label{prop:variance}
For each $l=1,\ldots,n$ and $k=1,\ldots,q-1$, we have
\begin{equation}\label{eq:varianceequality}
\Var_{\mu_G}(\Psi_l^{(k)}) = \frac{1}{\gamma_l}\sum_{v\in V}\pi_G(v)^2| \psi_l(v)|^2 \big(1-(1-\theta)h_G(v)\big)\,,
\end{equation}
with $h_G(v) = \frac{1}{d_G(v)}\sum_{w\sim v}\P(v \leftrightsquigarrow w)$ and $\{v\leftrightsquigarrow w\}$ the event that two independent continuous-time random walks from $v$ and $w$, which move at rate $1-\theta$ and are killed at rate $\theta$, coalesce before dying. 
In particular,
\begin{equation}\label{eq:varianceinequality1}
\frac{\theta}{\gamma_l}\sum_{v\in V}
\pi_G(v)^2| \psi_l(v)|^2\leq \Var_{\mu_G}(\Psi_l^{(k)}) \leq \frac{1}{\gamma_l}\sum_{v\in V}\pi_G(v)^2| \psi_l(v)|^2
\end{equation}
Furthermore, if we let $M:= \max_v \d_G(v)/ \min_v \d_G(v) $, then we have
\begin{equation}\label{eq:varianceinequality2}
\frac{\theta}{2M}\leq \frac{\theta}{M\gamma_l}\leq  n \Var_{\mu_G}(\Psi_l^{(k)})  \leq \frac{M}{\gamma_l}\leq \frac{M}{\theta}\,.
\end{equation}
\end{proposition}
\begin{proof}
Let $\mathscr{P}$ be the transition matrix of the discrete-time noisy voter model (where each step updates a uniformly chosen vertex), and recall from \cref{prop:eigenfunc} that  $\Psi_l^{(k)}$ is an eigenfunction of~$\mathscr{P}$ with eigenvalue $1-\frac{\gamma_l}{n}$. 
Let $Y_0$ be distributed according to the equilibrium measure~${\mu_G}$, and let~$Y_1$ be the configuration after one step of the discrete chain started from $Y_0$.
Letting 
\[\sigma^2 :=\Var_{{\mu_G}}(\Psi_l^{(k)}) = \E_{\mu_G}\big[|\Psi_l^{(k)}(Y_0)|^2\big]\]
(using here that $\E_{\mu_G}[\Psi_l^{(k)}(Y_0)]=\left\langle \E[Y_0^k],\psi_l\right\rangle_{\pi_G}=0$ as $Y_0(v)$ is uniform on $\Cq$ for each $v$), we have
\begin{align*}
\E_{\mu_G}\big[|\Psi_l^{(k)}(Y_1)-\Psi_l^{(k)}(Y_0)|^2\big] &= \E_{\mu_G}\big[|\Psi_l^{(k)}(Y_0)|^2\big] + \E_{\mu_G}\big[|\Psi_l^{(k)}(Y_1)|^2\big] - 2 \Re\Big( \E_{\mu_G}\Big[\Psi_l^{(k)}(Y_0)\overline{\Psi_l^{(k)}(Y_1)}\Big]\Big) \\
&= 2\sigma^2-2\Re\bigg(\E_{\mu_G}\Big[\Psi_l^{(k)}(Y_0)\,\E_{\mu_G}\Big[\overline{\Psi_l^{(k)}(Y_1)}\;\big|\; Y_0\Big]\Big]\bigg)\,.
\end{align*}
Since $\E_{\mu_G}[ \Psi_l^{(k)}(Y_1) \mid Y_0] = \E_{Y_0}[\Psi_l^{(k)}(Y_1)] = (1-\frac{\gamma_l}n)\Psi_l^{(k)}(Y_0) $ by \cref{prop:eigenfunc}, we deduce that
\begin{align*}
\E_{\mu_G}\big[|\Psi_l^{(k)}(Y_1)-\Psi_l^{(k)}(Y_0)|^2\big] &= 2\sigma^2-2(1-\frac{\gamma_l}{n})\sigma^2 = \frac{2\gamma_l}{n}\sigma^2\,,
\end{align*}
that is,
\begin{equation}
\Var_{\mu_G}(\Psi_l^{(k)}) = \frac{n}{2\gamma_l}\E_{\mu_G}\big[|\Psi_l^{(k)}(Y_1)-\Psi_l^{(k)}(Y_0)|^2\big]\,. \label{eq:varianceisincrement}
\end{equation}
To compute the latter, denote for brevity
\[ \Upsilon := \Psi_l^{(k)}(Y_1)-\Psi_l^{(k)}(Y_0) = \big\langle Y_1^k-Y_0^k,\psi_l\big\rangle\]
and consider the two events $\cE_v, \cE'_v$ describing an update at $v\in V$ in the first step: a noisy update per~$\cE_v$ and a vote-copying update per $\cE'_v$. Since $\Upsilon\one_{\cE_v} = \big(U^k - Y_0(v)\big) \pi_G(v)\psi_l(v)$, with $U$ independent and uniformly distributed on $\Cq$, we have $|\Upsilon|^2\one_{\cE_v} = |(U/Y_0(v))^k-1|^2 (\pi_G(v)\psi_l(v))^2$, thus
\[ \E_{\mu_G} |\Upsilon|^2 \one_{\cE_v} = \pi_G(v)^2| \psi_l(v)|^2 \frac1q \sum_{\omega\in\Cq}|\omega^k-1|^2 = 2 \pi_G(v)^2| \psi_l(v)|^2\,.\]
On the event $\cE'_v$, we have $\Upsilon\one_{\cE'_v} = (Y_0(J)^k-Y_0(v)^k)\pi_G(v)\psi_l(v)$, with $J$ independent and uniformly distributed on the neighbors of $v$ in $G$, thus
\begin{align*} \E_\mu |\Upsilon|^2\one_{\cE'_v} &= \pi_G(v)^2| \psi_l(v)|^2
\frac1{d_G(v)}\sum_{w\sim v} \E_{\mu_G}\big[|Y_0(w)^k-Y_0(v)^k|^2\big] \\
&=
2\pi_G(v)^2| \psi_l(v)|^2
\frac1{d_G(v)}\sum_{w\sim v} \P(v\nleftrightsquigarrow w)\,,
\end{align*}
since sampling $Y_0$ via CFTP (as described in \cref{sec:couplingfromthepast}) results either in $\{v\leftrightsquigarrow w\}$, on which $Y_0(v)=Y_0(w)$, or in $\{v\nleftrightsquigarrow w\}$, on which $Y_0(w),Y_0(v)$ are independent and uniformly distributed on~$\Cq$, contributing $2$ to the expectation (recall the case of $\cE_v$). Combining the last two displays,
\begin{align*}
\E_{\mu_G} |\Upsilon|^2 &= \frac2n \sum_{v} \pi_G(v)^2| \psi_l(v)|^2
\Big( \theta + (1-\theta)\frac1{d_G(v)}\sum_{w\sim v}\P(v\nleftrightsquigarrow w)\Big) 
\\ &= 
\frac2n \sum_{v} \pi_G(v)^2| \psi_l(v)|^2
\Big( 1 - (1-\theta)\frac1{d_G(v)}\sum_{w\sim v}\P(v\leftrightsquigarrow w)\Big)\,,
\end{align*}
which, via \cref{eq:varianceisincrement}, yields \cref{eq:varianceequality}. Since $0 \leq h_G(v) \leq 1$, we get \cref{eq:varianceinequality1}, and \cref{eq:varianceinequality2} follows from the facts $\frac1{M n}\leq \pi_G(v) \leq \frac{M}{n}$, $\theta \leq \gamma_l \leq 2-\theta\leq 2$ and $\sum_v \pi_G(v)|\psi_l(v)|^2=\|\psi_l\|_{L^2(\pi_G)}^2 = 1$.
\end{proof}
\begin{remark}
Consider the ``magnetization'' eigenfunction $\Psi$ associated to the eigenfunction $\psi=\bfone$ of the random walk transition matrix $P$, corresponding to $\lambda=1$ (and $\gamma=\theta$):
\[\Psi(\x)=\sum_v\x(v)\pi_G(v)\,.\]
As noted in the proof of \cref{prop:variance}, we have $\E_{\mu_G} \Psi(Y)=0$, and so
\[ \Var_{\mu_G}(\Psi(Y))= \E_{\mu_G}\big[|\Psi(Y)|^2\big] = \sum_{v,w}\pi_G(v)\pi_G(w) \E\big[ Y(v)\overline{Y(w)}\big] 
= \sum_{v,w}\pi_G(v)\pi_G(w) \P(v\leftrightsquigarrow w)
\,, \]
with the last identity using that  $Y(v),Y(w)$ are independent and uniform on~$\Cq$ when $\{v\nleftrightsquigarrow w\}$, so
$\E Y(v)\overline{Y(w)}\one_{\{v\nleftrightsquigarrow w\}}=0$, and
 elsewhere $Y(v)=Y(w)$, so that 
$\E Y(v)\overline{Y(w)}\one_{\{v\leftrightsquigarrow w\}}=\P(v \leftrightsquigarrow w)$.
By \cref{prop:variance}, we have $\Var_{\mu_G}(\Psi) \leq (1/\theta)\sum_v \pi_G(v)^2$, thus we can conclude that
\begin{equation}\label{eq:magnetizationvariance}
\sum_{v,w}\pi_G(v)\pi_G(w)\P(v \leftrightsquigarrow w) \leq \frac{1}{\theta}\sum_{v}\pi_G(v)^2\,.
\end{equation}
This will be used in \Cref{sec:lowerbounds} to bound the variance of certain statistics.
\end{remark}

\subsection{Classical autocorrelation}\label{sec:classical-autocor}
Fix an initial configuration $\x_0$. The point-wise autocorrelation of $X_t$ is routinely defined as
\[
\mathsf{A}(0,t) = \sum_v \pi_G(v)\Big(\one_{\{X_t(v)=\x_0(v)\}}-\frac1q\Big)\,,
\]
which in the case $q=2$, where $\Cq=\{\pm1\}$, can also be written as $\frac12 \sum_v \pi_G(v) \x_0(v) X_t(v)$. The function $\autI_t(\x_0)$ from \cref{def:autocorrelationL1} is nothing but $\E_{\x_0} [\mathsf{A}(0,t)]$, and we will now argue that
\begin{equation}\label{eq:autI-autII}\autII_t(\x_0) = \autI_{2t}(\x_0)\,.
\end{equation}
To see this, recall that $\one_{\{X_t(v)=x_0(v)\}} = \frac1q \sum_{k=0}^{q-1} X_t(v)^k\overline{\x_0(v)^k}$
(used in the proof of \cref{prop:autocorrelation}), 
via which we can rewrite $\mathsf{A}(0,t)$ as
\[ \mathsf{A}(0,t) = \frac1q \sum_{k=1}^{q-1} \langle X_t^k,\x_0^k\rangle_{\pi_G}\,;\]
this then becomes, using the representation $\x^k = \sum_l \Psi_l^{(k)}(\x)\psi_l$ for the orthonormal basis $\{\psi_l\}$, 
\[ \mathsf{A}(0,t) = \frac1q \sum_{k=1}^{q-1} \sum_{l=1}^n \Psi_l^{(k)}(X_t) \overline{\Psi_l^{(k)}(\x_0)}\,.\]
By \cref{prop:eigenfunc}, we have that 
\[\E_{\x_0}[\Psi_l^{(k)}(X_t)] = e^{-\gamma_l t} \Psi_l^{(k)}(\x_0)\,,\]
and it thus follows that
\[ 
\autI_t(\x_0) = \E_{\x_0} [\mathsf{A}(0,t)] = \frac1q \sum_{k=1}^{q-1} \sum_{l=1}^n |\Psi_l^{(k)}(\x_0)|^2 e^{-\gamma_l t} \,,
\]
establishing \cref{eq:autI-autII} in view of \cref{eq:At2-rep-2}.
(In particular, one can express all our results in terms of $\autI_t$ rather than $\autII_t$; we chose to use $\autII_t$ to emphasize the role of $L^2$-mixing in the analysis.)

\section{Applications for the lattice}\label{sec:applications}

\subsection{Lattice patterns}\label{sec:latticepatterns}
In this section, we let $G$ be the $d$-dimensional lattice $(\Z/n\Z)^d$. Recall from \cref{eq:Cq-def} that we are identifying the $q$ colors with the set $\Cq$ of $q$-th roots of unity. We can define a general class of configurations that are periodic in every lattice direction as follows: for a given vector $\v = (\v_1,\ldots,\v_d)$ with $\v_i\in\{0,\ldots,q-1\}$, let
\begin{equation}
\x_\v(j_1,\ldots,j_d) = 
e^{2\pi i (j_1\v_1+\ldots+j_d\v_d)/q}\,, \qquad j=(j_1,\ldots,j_d)\in\Z_n^d
\end{equation}
(NB.\ we assume that $n$ is a multiple of $q$, so that the above patterns do not break at the edges). For example, taking $\v=\bfone$, we see that $\x_\bfone$ (not to be confused with the monochromatic state $\bfone)$ is equivalent to the rainbow configuration from \Cref{eq:rainbow}:
\begin{equation}
\x_\rainbow(j_1,\ldots,j_d) = \x_\bfone(j_1,\ldots,j_d) = e^{2\pi i (j_1+\ldots+j_d)/q}\,, \qquad (j_1,\ldots,j_d)\in\Z_n^d\,.
\end{equation}
Observe that for every $\v$ and $k$, one has $\x_\v^k = \x_{k\v}$, and that  $\x_\v^k$ is an eigenfunction of the discrete-time simple random walk on $G$ (normalized in the sense that $\|x_\v^k\|_{L^2(\pi_G)}=1$) with eigenvalue 
\[
\lambda_{k\v} = \frac{1}{d}\sum_{i=1}^d\cos\Big(\frac{2\pi k \v_i}{q}\Big)\,.
\]
Hence, for configurations of the form $\x_\v$, \Cref{prop:autocorrelation} yields the exact expression for $\autII_t$:
\begin{equation}
    \label{eq:aut2-pattern}
\autII_t(\x_\v) = \frac1q\sum_{k=1}^{q-1}e^{-2(1-(1-\theta)\lambda_{k \v})t}\,.
\end{equation}
For a given $\v$, if we let \[\lambda_\v^*:=\max\{\lambda_{k\v}\,:\;k=1,\ldots,q-1\}\,,\] then by our definition  of $T_{\x_0}$ in \cref{def:Tx0} we have that
\[
T_{\x_\v} = \frac{d}{2(1-(1-\theta)\lambda_\v^*)}\log n+O(1)\,,
\]
and Theorem \ref{thm:main} will imply that
\begin{equation}\label{eq:tmix-lattice-pattern}
\tmix^{\x_\v} \sim \frac{d}{2\min\{1-(1-\theta)\lambda_\v^*,2\theta\}}\log n = \begin{cases}
\frac{d}{4\theta}\log n & 0 < \theta \leq \theta_{\v}:=1-\frac{1}{2-\lambda_\v^*}\,,\\
\frac{d}{2(1-(1-\theta)\lambda_\v^*)}\log n & \theta_\v \leq \theta \leq 1\,.
\end{cases}
\end{equation}
This shows that as long as $\lambda_\v^*<1$ (whence $\x_\v$ is faster than worst-case), the behavior of $\tmix^{\x_\v}$ as a function of $\theta$ always changes at  $\theta_\v = (1-\lambda_\v^*)/(2-\lambda_\v^*)\in(0,1)$ (a dynamical phase transition, marking the value of $\theta$ below which the correlation quantity $T_{\corr}$, rather than $T_{\x_\v}$, governs $\tmix^{\x_\v}$). 

\begin{example}[Rainbow initial condition on $\Z_n^d$]
For any $d\geq 1$ and $q\geq 2$, the rainbow initial condition $\x_\rainbow=\x_\bfone$ for the noisy voter model on $\Z_n^d$ has $\lambda_\bfone^*=\cos(\frac{2\pi}{q})$, so by \cref{eq:tmix-lattice-pattern},
\[
\tmix^{\x_\rainbow} \sim \frac{d}{2\min\big\{1-(1-\theta)\cos(\frac{2\pi}q),2\theta\big\}}\log n = \begin{cases}
\frac{d}{4\theta}\log n & 0 < \theta \leq \theta_\bfone\\
\frac{d}{2-2(1-\theta)\cos(\frac{2\pi}q)}\log n & \theta_\bfone \leq \theta \leq 1
\end{cases}\,,
\]
with $\theta_\bfone = 1-1/(2-\cos(\frac{2\pi}q))$.

In the special case $q=2$, $\x_\rainbow$ is the alternating condition $\x_\alt=\x_\bfone$ and has $\lambda_\bfone^*=-1$, the lowest possible value of any $\lambda_\v^*$; thus $\x_\bfone$ achieves the optimal $T_{\x_\bfone}\sim\frac{d}{4-2\theta}\log n$ for $q=2$ (as we have already established, more generally for any bipartite graph $G$, in \cref{rem:T-xalt}).
\end{example}
(NB.\ $\tmix^{\x_\rainbow}$ is asymptotically optimal in dimension $d=1$ over all choices of a (deterministic) initial state $\x_0$, as was stated in \Cref{thm:rainbow}; we will establish this in the next subsection.)

In higher dimensions that is no longer the case; e.g., as mentioned in the introduction, for $d=2$ and $q=5$ one can consider the ``knight'' initial condition $\x_\v$ for $\v=(1,2)$; there, one has $\lambda_\v^* = \frac12(\cos(\frac{2\pi}5)-\cos(\frac\pi5))=-\frac14$ (with an equal $\lambda_{k\v}$ for all $k=1,2,3,4$), whereas $\lambda_\bfone^* = \cos(\frac{2\pi}5)$. More generally, we have the following.

\begin{example}[Knight initial condition on $\Z_n^d$]  For every $(d,q)$ with $d\geq 2$ and $q\geq 5$, as well as for $q=4$ in dimensions $d\geq 3$, the knight initial condition $\x_{\knight} = \x_\v$ for $\v = (1,\ldots,1,2)$ has $\lambda^*_\v < \lambda^*_\bfone$. Thus, by \cref{eq:tmix-lattice-pattern}, $\tmix^{\x_0}$ from an initial state $\x_0=\x_\knight$ is asymptotically faster than from $\x_0=\x_\rainbow$. 

To see this, write $c_k = \cos(\frac{2\pi k}q)$, so that $\lambda_{k\bfone} = c_k$ and $\lambda^*_{\bfone} = c_1$. With this notation, 
\[\lambda_{k\v} = \tfrac{d-1}d c_k + \tfrac1d c_{2k}\,.
\]
Let $k' = 2k \!\pmod q$, and note that if $k'\neq 0$ and $c_{k'}\neq c_k$ then $\lambda_{k\v} < \max\{c_k,c_{k'}\} \leq \lambda^*_{\bfone}$. Indeed:
\begin{itemize}[leftmargin=0.2in, nosep]
    \item The case $k'=0$ occurs if and only if $k=q/2$ (for $q$ even), whence $c_k=-1$ and $c_{k'}=1$, so $\lambda_{k\v} = -1+\frac2d \leq 0$ (strictly if $d\geq 3$) and $\lambda^*_{\bfone} = c_1 =\cos(\frac{2\pi}q) \geq 0$ (strictly if $q> 4$); thus, $\lambda_{k\v}< c_1$.
    \item The case 
$c_{k'} = c_k = c_1$ can occur only if $2\pi/q = \frac23 \pi$, i.e., when $q=3$, precluded by assumption.
\end{itemize}
Note that, for $q=4$, one can readily write the explicit values of $\lambda_{k\v}$ for $\v=(1,\ldots,1,2)$ (as $c_k$ is simply $0$ if $k$ is odd and $(-1)^{k/2}$ if it is even), yielding $\lambda^*_\v = -1/d$ for $d\geq 3$ (and $\lambda^*_\v=0$ for $d=2$). 
\end{example}

\subsection{Optimal initial condition in 1D}\label{subsec:opt-1d}

In this section we establish  \Cref{thm:rainbow}, showing that the rainbow initial state $\x_\rainbow$ is optimal in dimension $d=1$ for all $q\geq 2$.

\begin{proof}[Proof of \cref{thm:rainbow}]    Once we prove that $\autII_t(\x_\rainbow) = O(\autII_t(\x_0))$ for every $\x_0$ as per
\cref{eq:rainbowlb}, the fact that $T_{\x_\rainbow} \leq T_{\x_0}+O(1)$ will readily follow from \cref{eq:A2submultiplicativity1} of \Cref{prop:autocorrelation}. 

From the discussion in \Cref{sec:latticepatterns} (specifically, \cref{eq:aut2-pattern}), we have
\[
\autII_t(\x_\rainbow) = \frac1q \sum_{k=1}^{q-1} e^{-2(1-(1-\theta)\cos(\frac{2\pi k}{q }))t} \leq  e^{-2(1-(1-\theta)\cos(\frac{2\pi}{q }))t}\,.
\]
Hence, to prove \cref{eq:rainbowlb} it remains to show that, for some absolute constant $c_0>0$,
\begin{equation}\label{eq:rainbowgoal0}
\autII_t(\x_0) \geq \frac{c_0}{ q^2}\,e^{-2(1-(1-\theta)\cos(\frac{2\pi}{q}))t} \,.
\end{equation}
Note that \Cref{prop:autocorrelation} already gives such a bound for $q=2$, as we have $\autII_t(\x_0) \geq \frac{q}{q-1} e^{-(4-2\theta)t}$. Hence, we can assume $q \geq 3$ through the remainder of this proof. 

We will identify the vertices of $G=\Z/n\Z$ with $0,\ldots,n-1$, and let
\[
\psi_l = \Big(\exp\big(\tfrac{2\pi i jl}{n}\big)\Big)_{j=0}^{n-1}\qquad(l=0,\ldots,n-1)
\]
(the $l$-th $n$-dimensional Fourier vector), which form an $L^2(\pi_G)$-orthonormal basis of eigenfunctions for simple random walk on $G$, with corresponding eigenvalues
\[
\lambda_l = \cos\big(\tfrac{2\pi l}{n}\big)\qquad(l=0,\ldots,n-1)\,.
\]
Next, let
\begin{equation}\label{def:probabilityAmplitude}
\nu^{(k)}= \left(\langle \x_0^k,\psi_l\rangle_{\pi_G}\right)_{l=0}^{n-1}\,,
\qquad(k=0,1,\dots,q-1)\,,
\end{equation}
i.e., $\nu^{(k)}$ is the discrete Fourier transform (DFT) of $\x_0^k$, up to a factor of $1/n$ from $\pi_G$ in $\langle\cdot,\cdot\rangle_{\pi_G}$. 
That is, if $\sF$ is the DFT matrix, i.e., the symmetric matrix whose columns are $\bar{\psi_l}$, then $\nu^{(k)} = \frac1n \sF \x_0^k$.
We stress that $\nu^{(k)}$ have are unit vectors in $L^2$ by Parseval's identity:
\[
\sum_{l=0}^{n-1} |\nu^{(k)}(l)|^2=\sum_{l=0}^{n-1} |\langle \x_0^k, \psi_l\rangle_{\pi_G}|^2 = \langle \x_0^k,\x_0^k\rangle_{\pi_G} = \sum_{j=0}^{n-1} 
|\x_0^k(j)|^2
\pi_G(j) = \sum_{j=0}^{n-1} \pi_G(j) = 1\,.
\]
Recalling \cref{eq:At2-rep-2}, which wrote $\autII_t$ in terms of $\Psi_l^{(k)}(\x) = \langle \x^k,\psi_l\rangle_{\pi_G}$ from \cref{eq:Psi-def}, we have
\begin{equation}
    \label{eq:A2-nu}
\autII_t(\x_0) = \frac1q\sum_{k=1}^{q-1} \sum_{l=0}^{n-1} | \nu^{(k)}(l)|^2 e^{-2(1-(1-\theta)\lambda_l)t}\,.
\end{equation}
Our goal will roughly be to show that the coefficients $| \nu^{(k)}(l)|^2$ for $l\leq (1+o(1))n/q$ have large enough mass $c$, so that by the monotonicity of $\lambda_l$ we would get $\autII_t(\x_0) \geq c \exp[-2(1-(1-\theta)\lambda_{n/q})t]$.

Now, let $\Theta_l$ denote the circular shift $(\Theta_l \x)(j) = \x(j-l)$, and define 
\[
\nu_l^{(k)} := \Theta_l \nu^{(k)}\,.
\]
We want to exploit the fact that the state $\x_\rainbow$ yields an optimal basis in the following sense:
\[
\Big\{\nu_l^{(k)}\Big\}_{\substack{k=0,\ldots,q-1 \\ l=0,\ldots,n/q-1}} = \Big\{\delta_{k\frac{n}{q}+l}\Big\}_{\substack{k=0,\ldots,q-1\\ l=0,\ldots,n/q-1}} \qquad \text{if $\x_0=\x_\rainbow$}
\]
(in the above and in what follows, $\delta_l$ is the indicator function $\one_{\{\cdot=l\}}$),
while taking the same shifts with other initial configurations necessarily yields correlated vectors in the sector $\{\delta_0,\ldots\delta_{\frac{n}{q}}\}$ (which, as we will soon see, correspond to the coefficients that we are trying to bound from below). 

To do so, let $\circledast$ denote circular convolution (i.e., $(\x \circledast y)(m)=\sum_{j=0}^{n-1}\x(j)\y(m-j)$ with the index of $\y$ taken modulo $n$), so that
 \[ \nu_l^{(k)} = \nu^{(k)} \circledast \delta_l\,,\]
and recall that, if $\widehat \x=\sF \x$ and $\widehat \y = \sF \y$ are the DFTs of $\x$ and $\y$, then $\langle \widehat \x,\widehat \y\rangle = n \langle\x,\y\rangle$, and
$\sF (\x\circledast \y) = \widehat\x \cdot \widehat\y$, the point-wise product of $\widehat\x$ and $\widehat\y$. In the right hand above, $\nu^{(k)} = \frac1n \sF \x_0^k$ and $\delta_l = \langle\psi_l,\psi_l\rangle_{\pi_G} = \frac1n \sF \psi_l$, so $\sF \bar{\nu^{(k)}} = \bar{\x_0^k}$ and $\sF \delta_l = \bar{\psi_l}$, implying that
$ \sF \bar{\nu_l^{(k)}} = \bar{\x_0^k} \cdot \bar{\psi_l}$. 
It then follows that
\begin{align}
\langle \nu_{l_1}^{(k_1)}, \nu_{l_2}^{(k_2)} \rangle &= \frac1n \overline{\langle \sF \bar{\nu^{(k_1)}_{l_1}} , \sF \bar{\nu^{(k_2)}_{l_2}} \rangle} = \frac1n \overline{\langle \sF(\bar{\nu^{(k_1)}} \circledast \delta_{l_1}), \sF(\bar{\nu^{(k_2)}} \circledast \delta_{l_2}) \rangle} = 
\frac1n \langle \x_0^{k_1} \cdot \psi_{l_1}, \x_0^{k_2} \cdot \psi_{l_2} \rangle \nonumber\\
&= \frac1n\langle \x_0^{k_1-k_2}, \psi_{l_2-l_1} \rangle = \nu^{(k_1-k_2)}(l_2-l_1)\,,
\label{eq:nu^k_l-corr}
\end{align}
where we are reading both $k_1-k_2$ and $l_1-l_2$ modulo $n$. This shows that the Fourier coefficients $\nu^{(k)}(l)$ describe the correlations between the vectors $\{\nu^{(k)}_l\}$. To leverage this, consider the collection
\begin{equation}\label{eq:Bq-def}
\mathcal{B}_{N} := \{ \nu_l^{(k)}\,:\; k=0,1,\ldots,q-1\,,\; l=0,\ldots,N-1\}\,,
\end{equation}
for $N \asymp n$ to be fixed later (in \cref{eq:N-def}. The correlations between $\u,\v\in\mathcal{B}_{N}$ satisfy the following:
\begin{align}
\sum_{\substack{\u,\v \in \mathcal{B}_{N}\\ \u \neq \v}} |\langle \u, \v \rangle|^2 
&=
\sum_{\substack{k_1,k_2\in\{0,\ldots, q-1\}\\ l_1,l_2\in\{0,\ldots, N-1\} \nonumber\\
(k_1,l_1)\neq(k_2,l_2)}} |\nu^{(k_1-k_2)}(l_2-l_1)|^2 
=
\sum_{\substack{k_1,k_2\in\{0,\ldots, q-1\}\\ l_1,l_2\in\{0,\ldots, N-1\} \\
k_1\neq k_2}} |\nu^{(k_1-k_2)}(l_2-l_1)|^2 
\\ 
&= q\sum_{k=1}^{q-1} \sum_{l=-(N-1)}^{N-1} (N-|l|)|\nu^{(k)}(l)|^2\,,    \label{eq:B-corr}
\end{align}
where the first equality is by \cref{eq:nu^k_l-corr}; the second one used that $\nu^{(0)}(l)=\delta_0$ (as $\x_0^0=\bfone=\psi_0$), so we only get a contribution from $k_1=k_2$ if $l_1=l_2$ , disallowed in the first summation; finally, the transition between the lines counted the multiplicity of each $k=k_1-k_2$ and $l=l_2-l_1$ as follows:
\begin{itemize}
    \item there are $q$ pairs $(k_1,k_2)$  realizing every $k=k_1-k_2\neq 0$ (for every $k_1$, take $k_2=k+k_1$);
    \item there are $N-|l|$ pairs $(l_1,l_2)$ realizing every $l=l_2-l_1\in\{-(N-1),\ldots,N-1\}$  (for $l\geq 0$, take $l_1=0,\ldots,N-1-l$ and $l_2=l_1+l$; for $l< 0$, take $l_2=0,\ldots,N-1-|l|$ and $l_1=l_2-l$). 
\end{itemize}

To get a lower bound on the above coefficients, we use the following simple lemma, which is a variant of, e.g., \cite[Lem.~9.1]{Alon09} and \cite[Lem.~2]{Tao13} (those lemmas give bounds on $\max_{i\neq j}|\langle \v_i,\v_j\rangle|$ for unit vectors $\v_i\in\R^n$ as opposed to the version below that considers $\sum_{i\neq j} \langle |\v_i,\v_j\rangle|^2$ for $\v_i\in \C^n$, but the same proof applies to our setting here, and we include it for completeness).
\begin{lemma}\label{lem:terry}
Let $\v_1,\dots,\v_{n+r}\in\C^{n}$ with $\|\v_i\|=1$ for all $i$. Then
$
\sum_{i \neq j} |\langle \v_i, \v_j \rangle|^2 \geq r
$.
\end{lemma}
\begin{proof}
Consider the Gram matrix $G = (\langle \v_i, \v_j \rangle)_{i,j=1}^{n+r}$. As it is Hermitian with rank at most $n$ ($G=B^* B$ for an $n\times (n+r)$ matrix $B$), the matrix $A=G-I$ has real eigenvalues $\{\lambda_i\}_{i=1}^{n+r}$ in which $\lambda=-1$ has multiplicity at least $(n+r)-n=r$. By the assumption that the $\v_i$'s are unit vectors, its diagonal is all-zero, whence its Frobenius norm 
 $\|A\|_{\textsc f} = \sqrt{\operatorname{tr}(A^* A)}$ has $\|A\|^2_{\textsc f} =  \sum_{i\neq j} |\langle\v_i,\v_j\rangle|^2 $. At the same time, $\|A\|^2_{\textsc f} =  \sum_{i}\lambda_i^2 \geq r$, as required.
\end{proof}
Choose now 
\begin{equation}\label{eq:N-def}
N=\frac{n}{q}+\frac{n}{15+(1-\theta)t}\,,
\end{equation}
noting that
\[ \frac{2\pi}q  \leq \frac{2\pi N}n \leq \frac{2\pi}q + \frac{2\pi}{15} \leq \frac{4\pi}5\]
since $q\geq 3$. 
Then $\mathcal{B}_{N}$ from \cref{eq:Bq-def} has $qN $ unit vectors in $\C^n$, and therefore, by \cref{lem:terry},
\[
\sum_{\substack{\u,\v \in \mathcal{B}_{N}\\ \u \neq \v}} |\langle \u, \v \rangle|^2 \geq qN - n = \frac{qn}{15+(1-\theta)t}\,,
\]
which, when combined with \cref{eq:B-corr}, gives
\begin{equation}\label{eq:fourier1}
\sum_{k=1}^{q-1} \sum_{l=-(N-1)}^{N-1} (N-|l|)|\nu^{(k)}(l)|^2 \geq \frac{n}{15+(1-\theta)t}\,,
\end{equation}
as well as, after plugging the trivial bound $N-|l| \leq N $ into the last display,
\begin{equation}\label{eq:fourier2}
\sum_{k=1}^{q-1} \sum_{l=-(N-1)}^{N-1} |\nu^{(k)}(l)|^2 \geq \frac1N \cdot\frac{n}{15+(1-\theta)t)} = \frac{q}{15+q+(1-\theta)t}\,.
\end{equation}
In order to use these two estimates, we first bound $\autII_t(\x_0)$ from \cref{eq:A2-nu} via
\begin{align*}
\autII_t(\x_0) &\geq \frac1q \sum_{k=1}^{q-1} \sum_{l=-(N-1)}^{N-1} | \nu^{(k)}(l)|^2 e^{-2(1-(1-\theta)\lambda_l)t}\,,
\end{align*}
and observe that, for each $-N<l<N$,
\begin{align*}
e^{-2(1-(1-\theta)\lambda_l)t} &=  e^{-2(1-(1-\theta)\cos(\frac{2\pi N}{n}))t}  \exp\bigg(2\Big(\cos\Big(\frac{2\pi |l|}{n}\Big) - \cos\Big(\frac{2\pi N}{n}\Big)\Big)(1-\theta)t\bigg)\\
&\geq  
e^{-2(1-(1-\theta)\cos(\frac{2\pi N}{n}))t}  \bigg(1+2\Big(\cos\Big(\frac{2\pi |l|}{n}\Big) - \cos\Big(\frac{2\pi N}{n}\Big)\Big)(1-\theta)t\bigg)\,.
\end{align*}
Writing $x=2\pi N/n \in [2\pi/q, 4\pi/5]$ and $y:=2\pi(N-|l|)/n \in[2\pi/n,x]$, in this range we have
\[ \frac{\pi}{4 q} y \leq \cos(x-y) - \cos(x) \leq y\,.\]
(The right inequality holds as  $\cos(\cdot)$ is $1$-Lipschitz; for the left one, when  $y\leq x/2$, the mean value theorem supports the lower bound $\delta_q y$ for $\delta_q=\min\{\sin \xi: \frac{\pi}q<\xi<\frac45\pi\}\geq \pi/(2q)$, and if $y\geq x/2$ then we have the lower bound $\cos(x/2)-\cos x \geq \delta_q x/2 \geq \delta_q y/2$ for the same $\delta_q$.)
It follows that
\begin{align*}
\autII_t(\x_0) &\geq \frac1q e^{-2(1-(1-\theta)\cos(\frac{2\pi N}{n}))t} \sum_{k=1}^{q-1} \sum_{l=-(N-1)}^{N-1} | \nu^{(k)}(l)|^2 \left(1+\frac{N-|l|}n\cdot\frac{\pi^2}{q}(1-\theta)t\right)\\
&\geq \frac1q e^{-2(1-(1-\theta)\cos(\frac{2\pi N}{n}))t} \left(\frac{q}{15+q+(1-\theta)t} + \frac{\pi^2}{q}\cdot \frac{(1-\theta)t}{15+(1-\theta)t}\right)\,,
\end{align*}
where between the lines we used \cref{eq:fourier1,eq:fourier2}.
Finally, as $\cos(
\cdot)$ is $1$-Lipschitz, 
\[
\cos\Big(\frac{2\pi N}{n}\Big) \geq \cos\Big(\frac{2\pi}{q}\Big) - \frac{2\pi}{15+(1-\theta)t}\,,
\]
and therefore
\[
\autII_t(\x_0) \geq \frac1q e^{-2(1-(1-\theta)\cos(\frac{2\pi}{q}))t} \max\left\{\frac{q}{15+q+(1-\theta)t} ,\,\frac{\pi^2}{q}\cdot\frac{(1-\theta)t}{15+(1-\theta)t}\right\} e^{-4\pi \frac{(1-\theta)t}{15+(1-\theta)t}}\,.
\]
If $(1-\theta)t \leq 1$ then the first term in the maximum is at least $1/7$, whereas if $(1-\theta)t \geq 1$ then the second term in the maximum is at least $\pi^2/(16q)$. The last exponent is at least $e^{-4\pi}$, and when combined, this establishes \cref{eq:rainbowgoal0} and thus concludes the proof.
\end{proof}

\section{Upper bounds}\label{sec:upperbound}
In this section, we prove the upper bound in \Cref{thm:main}, from which the upper bound in \cref{cor:uniform} will readily follow (see \Cref{remark:uniform}).

Fix $\theta\in (0,1]$ and $q \geq 2$. Let $G=(V,E)$ be a graph on $n$ vertices satisfying \cref{eq:expansion} for some fixed $\gAlpha,\gKap$, and let $\x_0$ be an initial configuration on $G$. (We do not require that $G$ to be connected for this part of the proof.)
Define
\[ d^{\x_0}_\tv(t) := \left\|\P_{\x_0}(X_{t} \in \cdot)- \mu_G\right\|_\tv\]
and 
\begin{align}
t_0 &:= \max\Big\{T_{\x_0}, \frac{1}{4\theta}\log n\Big\}\,,\label{eq:t0-def}\\
t_1 &:= t_0 + b_1 (\log n)^{1-\frac{\gAlpha}{2}}\quad\mbox{for}\quad b_1=3/\theta\,,\label{eq:t1-def}\\
t_2 &:= t_1 + b_2 (\log n)^{1-\frac{\gAlpha}{2}} \quad \mbox{for}\quad b_2=2/\theta\,.
\label{eq:t2-def}
\end{align}
The goal will be to show that
\begin{equation}\label{eq:ubgoal}
d^{\x_0}_\tv(t_2) \leq  C e^{-\frac12(\log n)^{1-\frac{\gAlpha}2}}\,,
\end{equation}
for some $C>0$ depending on $\theta,q,\gAlpha,\gKap$ (but not $\x_0$ and $G$). 

\begin{remark}\label{rem:n0-large-enough}
If one wanted to write an upper bound that would hold for every $n$ (as opposed to $n$ large enough), and at the same time keeps the constant prefactor $C>0$ reasonably small, the following could be done. Using the representation detailed in \cref{sec:alternativerepresentations}, one can show that $\tmix^{\x_0}(\epsilon)\leq \theta^{-1}(\log n + \log(1/\epsilon))$, since we can couple $X_t$ to the stationary distribution on the event that all the walks $(Z_s^v)_{v\in V}$ have died (occurring with probability at least $1-ne^{-\theta t}$).
\end{remark}

To prove \cref{eq:ubgoal}, we follow the recipe of the analogous upper bound proof in \cite{LubetzkySly21}, with some modifications due to the fact that we are working with a slightly more general model and on an arbitrary geometry, rather than on a lattice. The last part of the proof instead relies on the complete monotonicity of $\autII_t(\x_0)$, proven in \Cref{prop:autocorrelation}.

Throughout this section, we will be making ample use of the duality with coalescing random walks and the couplings discussed in \Cref{sec:duality,sec:couplingfromthepast}. Hence, for the fixed time $t_2$, we can look at the dynamics on $[0,t_2]$ backwards in time and analyze the walks $\{Z_s^v\}_{v \in V, s \in [0,t_2]}$. We refer to the locations of the walks at some time $s$ (or $t_2-s$ in the forward dynamics) as the \textit{histories} of the walks at time $t_2-s$. We also refer to the collective locations of the walks that have not been killed by time $s$ (or $t_2-s$  in the forward dynamics) as the \textit{surviving} histories at time $t_2-s$.

\subsection{Reducing the total variation distance at time \texorpdfstring{$t_2$}{t2} to a product measure at time~\texorpdfstring{$t_1$}{t1}}\label{subsec:tv-t2-product-t1}
In this subsection we will establish the following result that, by analyze the backward dynamics in the interval $[t_1,t_2]$ from \cref{eq:t1-def,eq:t2-def}, reduces $d_{\tv}^{\x_0}(t_2)$ to the total variation distance between the projection of $X_{t_1}$ on certain subsets $V_i$ and the product measure of $\mu_G\restriction_{V_i}$.

\begin{lemma}
    \label{lem:reduce-X-t1(V_i)-product-mu}
Call a collection of subsets of vertices $\{V_i\}_{i=1}^m$ (for $m=m(n)$) {\upshape\good} if, for all $i,j$,
\[|V_i| \leq \gKap \exp\big[(\log n)^{1-\frac{\gAlpha}{2}}\big]\,, \quad \dist_G(V_i,V_j) \geq (\log n)^{1+\beta}\,,\]
where $\beta:=\frac{1}{2}\frac{\gAlpha^3}{1-\gAlpha}$. 
If $X_t^{(i)}$ are i.i.d.\ copies of $X_t$, and $Y^{(i)}$ are i.i.d.\ copies of $\mu_G$, then
\[
d_\tv^{\x_0}(t_2) \leq \max_{\{V_i\}_{i=1}^m\text{ \upshape\good}}
\Big\|\prod_{i=1}^m\P_{\x_0}\big(X_{t_1}^{(i)}(V_i) \in \cdot\big) - \prod_{i=1}^m \P\big(Y^{(i)}(V_i)\in\cdot\big) \Big\|_\tv + O(n^{-9})\,.
\]
\end{lemma} 
\begin{remark}\label{rem:E[Vi]-rather-than-max}
We will in fact show a stronger statement, where the maximum over $\{V_i\}$ is replaced by an expectation over random collections $\{V_i\}$ that are \good with an extra indicator that the surviving history at time $t_1$ is precisely $\bigcup_i V_i$ (see the event $E$ defined in \cref{lem:E} and \cref{eq:dtv-t2-random-Vi-E}).
\end{remark}
\begin{remark}\label{rem:Vi-balls}
We will reduce to \good collections $\{V_i\}$ where each $V_i$ is a subset of some ball $B_v(r)$ whose size is at most $\gKap \exp[(\log n)^{1-\frac\gAlpha2}]$. Since the proofs in later subsections will only rely on the sizes of the $V_i$'s and their pairwise distances, we did not include this in the criteria for being \good.
\end{remark}

We begin with an elementary observation that, in order for the random walk $(Z_s^v)$ to reach distance at least $r$ from its origin $v$, its first $r$ updates must be moving updates (having rate $1-\theta$) as opposed to a killing one (having rate $\theta$), the probability of which is at most $(1-\theta)^r\leq e^{-\theta r}$.
\begin{fact}\label{fact:obvious}
For every fixed $\kappa>0$ and $\beta>0$, the probability that that at least one of the $n$ random walks $(Z_s^v)_{v\in V}$ travels a distance of at least $\kappa(\log n)^{1+\beta}$ is at most $n \exp[-\theta \kappa (\log n)^{1+\beta}]$.
\end{fact}

The next lemma will show that the surviving histories at time $t_1$ are typically \good.

\begin{lemma}\label{lem:E}
Let $\{V_i\}$ be the sets obtained by starting with a singleton per  surviving walk at time~$t_1$, and repeatedly merging $V_i,V_j$ if $\dist_G(V_i,V_j)< (\log n)^{1+\beta}$. Let $E$ be the event that $\{V_i\}$ is {\upshape\good}. Then, we have
$\P(E) \geq 1-O(n^{-9})$.
\end{lemma}
\begin{proof}
For any $v\in V$, let $F_v$ be the event that $Z_s^v$ (the random walk started at $v$ in the backward dynamics) does not leave $B_v( \frac{1}{2}\log^{1+\beta} n)$ for $s\in[0,b_2(\log n)^{1-\frac{\gAlpha}{2}}]$, and let $F = \bigcap_{v \in V} F_v$. 
By \Cref{fact:obvious},
\[
\P(F) = 1-O(n^{-10})\,.
\]

Next, we want to show that with high probability, there exists an annulus of logarithmic thickness around any $v\in V$ such that all walks started at the vertices in such annulus are killed by time $b_2(\log n)^{1-\frac{\gAlpha}{2}}$ (that is, before reaching $t_1$ backwards in time). For $\gamma:=\frac{\gAlpha}{2}+\frac{\gAlpha^2}{2}$, define the annuli
\[
A_v(k) = B_v(2k \cdot 2(\log n)^{1+\beta})\setminus B_v((2k-1)\cdot 2(\log n)^{1+\beta})\,, \quad k=1,\ldots, \tfrac{1}{4}(\log n)^{\gamma}\,.
\]
Note that for fixed $v$, the annuli $\{A_v(k)\}_{k \geq 1}$ are at graph distance at least $ 2(\log n)^{1+\beta}$ from each other, and that
\[
|A_v(k)| \leq |B_v(4k(\log n)^{1+\frac{\gAlpha}{2}})| \leq \gKap \exp\big(\big[2k\cdot2(\log n)^{1+\beta}\big]^{1-\gAlpha}\big) \leq \gKap e^{(\log n)^{(1+\beta+\gamma)(1-\gAlpha)}} = \gKap e^{(\log n)^{1-\frac{\gAlpha}{2}}}
\]
by \cref{eq:expansion} and using $k \leq \frac{1}{2}(\log n)^{\gamma}$ and $(1+\beta+\gamma)(1-\gAlpha)=1-\frac{\gAlpha}{2}$. For each $v$ and $k$, let $G_{v,k}$ denote the event that all the walks started at $A_v(k)$ are killed by time $b_2(\log n)^{1-\frac{\gAlpha}{2}}$. Then, since a single walk survives in such time with probability $\exp[-\theta b_2(\log n)^{1-\frac{\gAlpha}{2}}]$, a union bound yields
\[
\P(G^c_{v,k}) \leq |A_v(k)|e^{-\theta b_2(\log n)^{1-\frac{\gAlpha}{2}}} \leq \gKap e^{-(\log n)^{1-\frac{\gAlpha}{2}}}
\]
via the aforementioned bound on $|A_v(k)|$ and the choice $b_2 = 2/\theta$. It follows that
\[
\P(G_{v,k}^c \mid F) \leq \frac{\P(G^c_{v,k})}{\P(F)} = 2\gKap \exp\big[-(\log n)^{1-\frac{\gAlpha}{2}}\big]\,.
\]
for large enough $n$. Let $G_v = \bigcup_{k=1}^{\frac{1}{4}(\log n)^\gamma} G_{v,k}$ be the event that all the walks started from at least one of the annuli $A_v(k)$ are killed by time $b_2(\log n)^{1-\frac{\gAlpha}{2}}$. Then, noting the events $G_{v,k}$ are conditionally independent given $F$ (as the random walks from different annuli cannot coalesce under $F$), we get
\[
\P(G_v^c \mid F) = \prod_{k=1}^{\frac{1}{4}(\log n)^\gamma}\P(G^c_{v,k}\mid F) \leq \left(2 \gKap e^{-(\log n)^{1-\frac{\gAlpha}{2}}}\right)^{\frac{1}{4}(\log n)^\gamma} = e^{-(\frac14-o(1))(\log n)^{1+\gAlpha^2/2}}
= O(n^{-10})\,,
\]
where we used the fact that $1-\gAlpha/2 + \gamma = 1+\gAlpha^2/2$. 

Hence, if $G= \bigcap_{v \in V} G_v$ is the event that every vertex $v\in V$ is surrounded by some annulus $A_v(k)$ (with $1 \leq k \leq \frac{1}{4}(\log n)^\gamma$) whose random walks are killed by time $b_2(\log n)^{1-\frac\gAlpha2}$, a union bound gives
\[
\P(F^c \cup G^c) \leq \P(F^c) + \sum_{v\in V}\P(G_v^c \mid F) = O(n^{-9})\,.\]
On the event $F \cap G$, at time $t_1$ every surviving walk must be
\begin{enumerate}
    \item within $\frac{1}{2}(\log n)^{1+\beta}$ of its initial position, and 
    \item surrounded by an annulus $A$ of thickness $2(\log n)^{1+\beta}$, that is contained in a ball $B$ of size at most $ \gKap e^{(\log n)^{1-\frac{\gAlpha}{2}}}$, such that all random walks started from $A$ were killed by time~$t_1$.
\end{enumerate} This implies that, on $F\cap G$, every surviving vertex is within a ball of size at most $ \gKap e^{(\log n)^{1-\frac{\gAlpha}{2}}}$ and is surrounded by an annulus, vacant from walks that survived to time $t_1$, of thickness at least $(2-\frac{1}{2}-\frac{1}{2})(\log n)^{1+\beta} = (\log n)^{1+\beta}$. It follows that $F \cap G \subset E$, and so $\P(E) \geq 1-O(n^{-9})$.
\end{proof}

To reduce $\|\P_{\x_0}(X_{t_2} \in \cdot) - \mu_G\|_\tv
$
to $
\|\P_{\x_0}(X_{t_1}(\bigcup V_i) \in \cdot) - \mu_G\|_\tv
$ for the well-separated surviving components $\{V_i\}$ whose existence was addressed in \cref{lem:E}, we will use the \emph{update support} notion from \cite{LubetzkySly13,LubetzkySly14}. This notion takes a much simpler form in our setting of the noisy voter model compared to the Ising model. 
We include the proof of the following simple lemma---taken verbatim from \cite{LubetzkySly14}---both for completeness, and to justify a slightly generalized version of it (for a product of multiple chains), which we will later need and follows from the same argument.
\begin{lemma}[Special case of {\cite[Lem.~3.3]{LubetzkySly14}}] \label{lemma:supportinequality}
For any $t \geq \hat t \geq 0$ and any set $W$ of vertices at time~$t$ in the space-time slab, if $\hat W$ denotes the (random) history of $W$ at time $\hat t$, we have
\[
\left\|\P_{\x_0}(X_{t}(W)\in \cdot)-\mu_G\restriction_{W}\right\|_\tv \leq \E\left[\big\|\P_{\x_0}(X_{\hat t}(\hat W)\in \cdot)-\mu_G\restriction_{\hat W}\big\|_\tv\right]\,,
\]
where the expectation is with respect to the update sequences along $[\hat t,t]$. Furthermore, if $\{X_t^{(i)}\}$ are i.i.d.\ copies of the dynamics, and $\tau_i\in[0,t]$ are stopping times for their backward dynamics, then
\[
\left\|\prod \P_{\x_0}\big( X_{t}^{(i)}(W_i)\in \cdot\big)-\prod \mu^{(i)}_G\restriction_{W_i}\right\|_\tv \leq \E\left[\big\|\prod \P_{\x_0}(X_{t-\tau_i}(\hat W_i)\in \cdot)-\prod \mu_G\restriction_{\hat W_i}\big\|_\tv\right]\,.
\]
\end{lemma}
\begin{proof}
Consider two instances of the dynamics, $X_t$ and $Y_t$, started at $\x_0$ and $\mu_G$ respectively.
As explained in \Cref{sec:couplingfromthepast}, we couple the backward dynamics of $X_t,Y_t$ using the same update sequence. Given the cumulative set of updates $\cF_{\hat t,t}$ in the interval $[\hat t,t]$, the configurations $X_{t}(W)$ and $Y_{t}(W)$ at time $t$ become deterministic functions $f_{\cF_{\hat t,t}}(X_{\hat t}(\hat W))$ and $f_{\cF_{\hat t,t}}(Y_{\hat t}(\hat W))$ (with the same $f_{\cF_{\hat t,t}}$) of the configurations at time $\hat t$ restricted to the surviving history $\hat W$ (as the killed histories are accounted for by $\cF_{\hat t,t}$). Hence, for any set of configurations $\Gamma$ and a realization of $\cF_{\hat t,t}$,
\begin{align*}
\P_{\x_0}(X_{t}(W) \in \Gamma\mid\cF_{\hat t,t})  &= \P_{\x_0}(f_{\cF_{\hat t,t}}(X_{\hat t}(\hat W)) \in \Gamma)\,,\\
\P_{\mu_G}(X_{t}(W) \in \Gamma\mid\cF_{\hat t,t}) &= \P_{\mu_G}(f_{\cF_{\hat t,t}}(Y_{\hat t}(\hat W)) \in \Gamma)\,.    
\end{align*}
Thus,
$\|\P_{\x_0}(X_{t}(W)\in \cdot)-
\P_{\mu_G}(Y_t(W)\in\cdot)\|_\tv =\max_{\Gamma}\big(\P_{\x_0}(X_{t}(W) \in \Gamma) - \P_{\mu_G}(Y_{t}(W) \in \Gamma)\big)$ becomes
\begin{align*}
 \max_{\Gamma} &\int\left(\P_{\x_0}(f_{\cF_{\hat t,{t}}}(X_{\hat t}(\hat W)) \in \Gamma) - \P_{\mu_G}(f_{\cF_{\hat t,{t}}}(Y_{\hat t}(\hat W)) \in \Gamma)\right)\,\d\P(\cF_{\hat t,{t}})\\
\leq &\int\max_{\Gamma} \left(\P_{\x_0}(f_{\cF_{\hat t,{t}}}(X_{\hat t}(\hat W)) \in \Gamma) - \P_{\mu_G}(f_{\cF_{\hat t,{t}}}(Y_{\hat t}(\hat W)) \in \Gamma)\right)\,\d\P(\cF_{\hat t,{t}})\\
\leq &\int\left\|\P_{\x_0}(X_{\hat t}(\hat W)\in \cdot)-\P_{\mu_G}(Y_{\hat t}(\hat W)\in \cdot)\right\|_\tv \,\d \P(\cF_{\hat t,{t}})\,,
\end{align*}
where in the last inequality we used that a projection can only reduce total variation distance.

For the generalization, note first that we can replace $\hat t$ by $t-\tau$ follows from exactly the same proof, as we may replace $\cF_{\hat t,t}$ everywhere with the stopped $\sigma$-algebra $\cF_{\tau}$ for the backward dynamics.
(Issues pertaining to the event $\{\tau=s\}$ having probability $0$ can easily be avoided by running the discrete backward dynamics while tracking the cumulative associated rescaling time in tandem.)
For a product chain $(X_t^{(i)})_{i=1}^m$ with individual stopping times $\tau_i$ for the backward dynamics $Z_s^{(i)}$, we expose the stopped chain $Z^{(i)}_{s\wedge \tau_i}$ in each coordinate, and the required result follows.
\end{proof}

\begin{proof}[Proof of \cref{lem:reduce-X-t1(V_i)-product-mu}]
Applying \cref{lemma:supportinequality} with $t'=t_1$, $t=t_2$ and $W=V$, we
write $W' = \bigcup V_i$ as prescribed in \cref{lem:E}, and further apply that lemma to find that
\begin{align}
\label{eq:dtv-t2-random-Vi-E}
d_\tv^{\x_0}(t_2) &\leq \E\bigg[\Big\|\P_{\x_0}\big(X_{t_1}\big(\bigcup V_i\big) \in \cdot\big) - \mu_G\restriction_{\bigcup V_i}\Big\|_\tv\one_E \bigg] + O(n^{-9})\,,
\end{align}
and in particular,
\begin{align}\label{eq:dtv-t2-max-proj}
d_\tv^{\x_0}(t_2) &\leq \max_{\{V_i\}\text{ \good}}\Big\|\P_{\x_0}\big(X_{t_1}\big(\bigcup V_i\big) \in \cdot\big) - \mu_G\restriction_{\bigcup V_i}\Big\|_\tv + O(n^{-9})\,.
\end{align}
Consider some realization of a \good $\{V_i\}_{i=1}^m$, and let us examine the  dynamics in the interval $[0,t_1]$.
Running the backward dynamics from time $t_1$, we couple the vector
$(X_{t_1}(V_i))_{i=1}^m$ with
$(X_{t_1}^{(i)}(V_i))_{i=1}^m$, where the $X_t^{(i)}$ are i.i.d., by letting the random walks and noisy updates generating the two vectors be identical as long as no walks from two different $V_i$'s meet. Should the walks from $V_i$ and $V_j$ meet, we let them coalesce normally under $(X_{t_1}(V_i),X_{t_1}(V_j))$, and instead let them stay independent under $(X_{t_1}^{(i)}(V_i),X_{t_1}^{(j)}(V_j))$. Consequently, the two vectors are identical if no two walks from different $V_i$'s meet, which happens for example if no walk travels more than $(\log n)^{1+\beta}$. Hence, by \Cref{fact:obvious},
\[
\Big\|\P_{\x_0}\big((X_{t_1}(V_1),\ldots,X_{t_1}(V_m))\in\cdot\big)-\P_{\x_0}\big((X_{t_1}^{(1)}(V_1),\ldots,X_{t_1}^{(m)}(V_m)) \in \cdot\big)\Big\|_\tv \leq O(n^{-10})\,.
\]
For the same reason (using $\P_{\mu_G}$ in lieu of $\P_{\x_0}$ for the $\{Y(V_i)\}$), we get that
\[
\left\|\mu_G\restriction_{\bigcup_{i=1}^m V_i} - (\mu_G\restriction_{V_1}\times\ldots\times\mu_G\restriction_{V_m})\right\|_\tv \leq O(n^{-10})\,.
\]
Combining the last two inequalities shows that for every \good $\{V_i\}_{i=1}^m$, 
\begin{align*}
\Big\|\P_{\x_0}\big(X_{t_1}\big(\bigcup V_i\big) \in \cdot\big) - \mu_G\restriction_{\bigcup V_i}\Big\|_\tv
\leq
\Big\|\prod_{i=1}^m\P_{\x_0}\big(X_{t_1}^{(i)}(V_i) \in \cdot\big) - \prod_{i=1}^m \P\big(Y^{(i)}(V_i)\in\cdot\big) \Big\|_\tv + O(n^{-10})\,.
\end{align*}
which together with \cref{eq:dtv-t2-max-proj} concludes the proof.
\end{proof}

\subsection{\texorpdfstring{$L^1$}{L1}-\texorpdfstring{$L^2$}{L2} reduction at time \texorpdfstring{$t_1$}{t1}}\label{subsec:l1tol2} 
Thus far we have reduced $d_{\tv}^{\x_0}(t_2)$ to the total variation distance between two product measures, namely $(X_{t_1}^{(i)}(V_i))_{i=1}^m$ and $(\mu_G\restriction_{V_i})_{i=1}^m$ for a \good collection $\{V_i\}_{i=1}^m$.
In this subsection, we bound the latter as follows.
\begin{lemma}\label{lem:Xt-from-V-to-U}
The following holds for every {\upshape\good} collection $\{V_i\}_{i=1}^m$. Let $\tau_i$ be the minimum $s\geq 0$ where the backward dynamics, started at time $t_1$ from every $v\in V_i$, has its surviving history reach size at most $3$ at time $t_1-s$, or $\tau_i=t_1$ if the backward dynamics reaches time $0$ before this happens. Further let $U_i$ be the surviving history at time $\tau_i$. Then there exists some $C=C(q,\theta)>0$ such that
\begin{align*}\Big\|\prod_{i=1}^m\P_{\x_0}\big(X_{t_1}^{(i)}(V_i) \in \cdot\big) - \prod_{i=1}^m \mu_G\restriction_{V_i} \Big\|_\tv^2
\leq C\, &\E\bigg[
\sum_{i=1}^m
\Big\| \P_{\x_0}\big( X_{t_1-\tau_i}^{(i)}(U_i)\in\cdot\big) - 
\mu_G\restriction_{U_i} \Big\|_\tv^2 \one_{\{\tau_i<t_1\}}
\bigg] \\
&+ O\big(e^{-8(\log n)^{1-\frac\gAlpha2}}\big)\,.
\end{align*} 
\end{lemma}
\begin{proof}
We first examine the event $\{\tau_i< t_1\}$. The probability that $4$ fixed vertices $v_1,\ldots,v_4$ in $V_i$ have each of the $Z_s^{v_{j}}$ survive to time $0$ without ever coalescing is at most 
\[ e^{-4 \theta t_1}\leq \frac1n e^{-12(\log n)^{1-\frac\gAlpha2}}\,,\]
using \cref{eq:t1-def}. Since $|V_i|\leq \gKap \exp\big[(\log n)^{1-\frac\gAlpha2}\big]$, we thus have
\[ \P(\tau_i = t_1) \leq \binom{|V_i|}4 e^{-4\theta t_1} \leq \frac{\gKap^4}n e^{-8(\log n)^{1-\frac\gAlpha2}}\,,
\]
and can conclude from a union bound over the $m\leq n$ sets $V_i$ that
\[ \P\left(\max_{1\leq i\leq m}\tau_i<t_1\right) = 1-O\big(e^{-8(\log n)^{1-\frac\gAlpha2}}\big)\,.\]
We now apply \cref{lemma:supportinequality}, using the stopping time $\tau_i$ for $ X_t^{(i)}$, inferring that
\begin{align}
\Big\|\prod_{i=1}^m \P_{\x_0}&\big( X_{t_1}^{(i)}(V_i)\in\cdot\big) - 
\prod_{i=1}^m \mu_G\restriction_{V_i} \Big\|_\tv \leq \E\left[
\Big\|\prod_{i=1}^m \P_{\x_0}\big(X_{t_1-\tau_i}^{(i)}(U_i)\in\cdot\big) - 
\prod_{i=1}^m \mu_G\restriction_{U_i} \Big\|_\tv
\right] \nonumber\\
&\leq \E\left[
\Big\|\prod_{i=1}^m \P_{\x_0}\big(X_{t_1-\tau_i}^{(i)}(U_i)\in\cdot\big) - 
\prod_{i=1}^m \mu_G\restriction_{U_i} \Big\|_\tv \one_{\{\max_{i}\tau_i<t_1\}} 
\right] + O\big(e^{-8(\log n)^{1-\frac\gAlpha2}}\big)
\,.\label{eq:Xt-Vi-Ui}
\end{align}
We now wish to use the $L^1$-$L^2$ reduction for product chains (we only need one side of it here):
\begin{proposition}[{\cite[Prop.~7]{LubetzkySly14}}]\label{prop:l1-l2} Let $X_t=(X_t^{(1)},\ldots,X_t^{(m)})$ for mutually independent ergodic chains with stationary distributions $\pi_i$. Let $\pi=\prod_{i=1}^m \pi_i$ and  
$ \fM_t = \sum_{i=1}^m \|\P(X_t^{(i)}\in\cdot)-\pi_i\|_{L^2(\pi_i)}^2$.
Then $\|\P(X_t\in\cdot)-\pi\|_{\tv}\leq \sqrt{\fM_t}$ for any $t>0$; in particular, if  $\fM_t\to 0$ then $\|\P(X_t\in\cdot)-\pi\|_\tv \to 0$. Conversely,  if $\fM_t\to \infty$ then $\|\P(X_t\in\cdot)-\pi\|_\tv \to 1$.
\end{proposition}
\begin{remark}\label{rem:l1l2-upper}
The upper bound in \cref{prop:l1-l2} is straightforward and, as seen in that proof, holds more generally: if $\nu= \prod_{i=1}^m\nu_i$ and $\pi=\prod_{i=1}^m\pi_i$ then $\|\nu-\pi\|_\tv \leq (\sum_{i=1}^m \|\nu_i-\pi_i\|_{L^2(\pi_i)}^2)^{1/2}$.
\end{remark}
By \cref{prop:l1-l2} (we apply its upper bound as per \cref{rem:l1l2-upper}), we infer from \cref{eq:Xt-Vi-Ui} along with Jensen's inequality that
\begin{equation}
    \label{eq:Xt-Mt}
\Big\|\prod_{i=1}^m \P_{\x_0}\big( X_{t_1}^{(i)}(V_i)\in\cdot\big) - 
\prod_{i=1}^m\mu_G\restriction_{V_i}\Big\|_\tv^2 \leq \E\left[\fM \one_{\{\max_i\tau_i<t_1\}}\right] + O\big(e^{-8(\log n)^{1-\frac\gAlpha2}}\big)\,,
\end{equation}
for the random variable
$\fM = \sum_{i=1}^m \|\nu_i-\pi_i\|_{L^2(\pi_i)}^2$ with $\nu_i = \P_{\x_0}( X_{t_1-\tau_i}(U_i) \in \cdot)$  and $\pi_i = 
\mu_G \restriction_{U_i}$.
Observe that we deterministically have 
\begin{align*}
\|\nu_i - \pi_i\|_{L^2(\pi_i)}^2  &=\sum_{\x\in\Cq^{U_i}}\frac{|\nu_i(\x)-\pi_i(\x)|^2}{\pi_i(\x)} 
\leq \frac{\|\nu_i-\pi_i\|_\infty\|\nu_i-\pi_i\|_1}{\min_\x \pi_i(\x)}
 \leq 2\Big(\frac q\theta\Big)^{|U_i|}\|\nu_i-\pi_i\|_\tv^2\,,    
\end{align*}
using that $\min_\x \pi_i(\x)\geq (\theta/q)^{|U_i|}$, since any configuration $\x$ can be obtained by letting the first update in the backward dynamics of each vertex $v\in U_i$ be a noisy update (with probability $\theta$) that assigns the new value $\x(v)$ (with probability $1/q$). 
As $|U_i|\leq 3$ when  $\tau_i < t_1$, we get
\begin{align*}
\E\left[\fM \one_{\{\max_i\tau_i<t_1\}}\right] 
& \leq 2\frac{q^3}{\theta^{3}}\,\E\left[
\sum_{i=1}^m 
\Big\| \P_{\x_0}\big(X_{t_1-\tau_i}^{(i)}(U_i)\in\cdot\big) - 
\mu_G\restriction_{U_i} \Big\|_\tv^2
\one_{\{\tau_i<t_1\}}\right] 
 \,,
\end{align*}
which, together with \cref{eq:Xt-Mt}, completes the proof.
\end{proof}

\subsection{Reduction to the autocorrelation \texorpdfstring{$\autII_t$}{A2}} \label{subsec:reductiontoA2}
With \cref{lem:Xt-from-V-to-U} in mind, define
\begin{equation}
    \label{eq:dtv-U-t-def}
d_\tv^{\x_0}(U_i,t) := \Big\|\P_{\x_0}\big(X_{t_1-\tau_i}^{(i)}(U_i)\in\cdot\big) - 
\mu_G\restriction_{U_i} \Big\|_\tv^2\,.
\end{equation}
The final ingredient in the proof will be to establish the following.
\begin{lemma}\label{lem:dtv-U-bound}
The following holds for every {\upshape\good} collection $\{V_i\}_{i=1}^m$. Let $U_i$ and $\tau_i$ be as in \cref{lem:Xt-from-V-to-U}. Then, for some $C=C(q,\theta,\gKap,\gAlpha)>0$, the $d_\tv^{\x_0}(U_i,t)$ from \cref{eq:dtv-U-t-def} satisfy 
\[
\E\left[\sum_{i=1}^m d_\tv^{\x_0}(U_i,t_1-\tau_i)^2 \one_{\{\tau_i<t_1\}} \right] \leq C(1+t_1^3)e^{4(\log n)^{1-\frac{\gAlpha}{2}}}n (\autII_{t_1}(\x_0) \vee e^{-4\theta t_1}) + O(n^{-9})\,.
\]
\end{lemma}

To obtain this bound, we will aim to reduce the quantity $\sum_{i=1}^m d_\tv^{\x_0}(U_i,t_1-\tau_i)^2 \one_{\{\tau_i<t_1\}}
$ to one involving single sites, rather than subsets $U_i$ of size at most $3$. 

\begin{remark}\label{rem:intuition-A2}
Roughly put, the main obstacle in this part of the proof is that, given the generality of $\x_0$ and $G$, the function $s\mapsto d_\tv^{\x_0}(U_i,t_1-s)$ can be highly irregular. Rather than treating the terms in this sum individually, our strategy will be to reduce their total to the function $\autII_t(\x_0)$, which, as we showed in \Cref{prop:autocorrelation}, is completely monotone. To do so, we will need to sum $d_\tv^{\x_0}(U_i,\cdot)$ at identical times, and so we need to account for the fact that the times $\{\tau_i\}$ have different distributions. We overcome this by first replacing the $U_i$ with deterministic sets, then bounding the densities of the $\tau_i$, and finally reducing the analysis to single sites, thus to $\autII_t(\x_0)$.
    
\end{remark}

While the sets $V_i$ were well-separated (at distance at least $ (\log n)^{1+\beta}$), the $U_i$ might not be. However, they are so with high probability; namely, if $V_i'$ is the $\frac{1}{4}(\log n)^{1+\beta}$-thickening of~$V_i$, i.e.,
\[
V_i':=\bigcup_{v \in V_i}B_v\left(\tfrac{1}{4}(\log n)^{1+\beta}\right)\,,
\]
then \Cref{fact:obvious} (used with $\kappa=\frac14$) shows that every $U_i$ is contained in $V_i'$ with probability $1-O(n^{-10})$.
Note that because $\dist_G(V_i,V_j) \geq (\log n)^{1+\beta}$, the $V_i'$ satisfy
\[
\dist_G(V_i',V_j') \geq \tfrac{1}{2}(\log n)^{1+\beta}\,.
\]
Then, if $U_i^{s} \subset V'_i$ denotes the subset of $V'_i$ that maximizes $d_\tv^{\x_0}(U_i,t_1-s)^2$ (note that $U_i$ is not determined by $\tau_i$, while $U_i^{\tau_i}$ is), we have
\begin{equation}\label{eq:dtv-Ui-bnd-1}
\E\left[\sum_{i=1}^m d_\tv^{\x_0}(U_i,t_1-\tau_i)^2 \one_{\{\tau_i<t_1\}} \right] \leq \sum_{i=1}^m\E\left[d_\tv^{\x_0}(U_i^{\tau_i},t_1-\tau_i)^2 \one_{\{\tau_i<t_1\}}\right] + O(n^{-10})\,.
\end{equation}
Next, we would like to integrate over the $\tau_i$ and reduce to common times for every $i$. 
The probability of $\tau_i \geq s$ is at most bounded by the probability that at least $4$ walks started at $V_i$ survive up to time $s$, which is at most $\binom{|V_i|}4 e^{-4\theta s}$.
Then, 
if we are to have $\tau_i\in (s,s+h)$, it must be that along that interval we remain with exactly $4$ surviving walks (a single coalescence or noisy update before $\tau_i$ occurs), which then receive at least one update. The number of updates from that point is at most a \text{Poisson}($4h$) variable, and hence the probability it is nonzero is at most its mean, $4 h$. Therefore,
\begin{align}\label{eq:Fi-upper-bnd}
\P\left(\tau_i \in (s,s+h)\right) 
&\leq 4 \gKap^{4}e^{4(\log n)^{1-\frac{\gAlpha}{2}}}e^{-4\theta s}h\,.
\end{align}
Hence, if $F_i$ denotes the distribution function of $\tau_i$, then 
\begin{align}\label{eq:integration}
\E\left[d_\tv^{\x_0}(U_i^{\tau_i},t_1-\tau_i)^2 \one_{\{\tau_i<t_1\}}\right] &\leq \int_0^{t_1}d_\tv^{\x_0}(U_i^s,t_1 - s)^2\d F_i(s) \nonumber\\
&\leq 4\gKap^{4}e^{4(\log n)^{1-\frac{\gAlpha}{2}}}\int_0^{t_1}d_\tv^{\x_0}(U_i^s,t_1 - s)^2e^{-4\theta s}\d s
\end{align}
(note that in the latter expression, $U_i^s$ is a deterministic set for each $s$ given $V_i'$).
Hence, 
\begin{align}\label{eq:lastbound}
\sum_{i=1}^m\E\left[d_\tv^{\x_0}(U_i^{\tau_i},t_1-\tau_i)^2 \one_{\{\tau_i<t_1\}}\right] \leq 4\gKap^4 e^{4(\log n)^{1-\frac{\gAlpha}{2}}}\int_0^{t_1}\left(\sum_{i=1}^m d_\tv^{\x_0}(U_i^s,t_1 - s)^2\right)e^{-4\theta s}\d s\,.
\end{align}
We now use the following:

\begin{lemma}\label{lemma:reductiontoA2}
Let $\{W_i\}_{i=1}^m$ be subsets with $|W_i|\leq 3$ and $\dist_G(W_i,W_j) \geq \frac{1}{2}(\log n)^{1+\beta}$, and let $\Delta$ denote the maximum degree in $G$. Then for every $t\geq 0$,
\[
\sum_{i=1}^m d_\tv^{\x_0}(W_i,t)^2 \leq 18q\Delta n \left( t^2\autII_t(\x_0) + (t+1)e^{-4 \theta t}\right) + O(n^{-9})\,.
\]
\end{lemma}

\begin{proof}
Let $\xi_i$ denote the minimum $r\geq 0$ where the backwards dynamics, started at time $t$ from every $v\in W_i$, has its surviving history shrink to a single (random) vertex $v_i$ at time $t-r$. We have that $\P(\xi_i>t)\leq \binom{|W_i|}2 e^{-2\theta t}$, and applying  \cref{lemma:supportinequality}, we get
\begin{align*}
d_\tv^{\x_0}(W_i,t) \leq \E\left[d_\tv^{\x_0}(v_i,t-\xi_i)\one_{\{\xi_i \leq t\}}\right] + \P(\xi_i > t) \leq \E\left[d_\tv^{\x_0}(v_i,t-\xi_i)\one_{\{\xi_i \leq t\}}\right] + 3 e^{-2 \theta t}\,.
\end{align*}
Just as at the beginning of \Cref{subsec:reductiontoA2}, if $W'_i$ is a $\frac{1}{8}(\log n)^{1+\beta}$-thickening of $W_i$, that is,
\[
W_i':=\bigcup_{v \in W_i}B_v\left(\tfrac{1}{8}(\log n)^{1+\beta}\right)\,,
\]
then $v_i \in W_i'$ for every $i$, with probability at least $1-O(n^{-10})$. Thus, if $v_i^r$ denotes the vertex $v\in W_i'$ that maximizes $d_\tv^{\x_0}(v,t-r)$, then we can bound
\[
\E\left[d_\tv^{\x_0}(v_i,t-\xi_i)^2\one_{\{\xi_i \leq t\}}\right] \leq \E\left[d_\tv^{\x_0}(v_i^{\xi_i},t-\xi_i)^2\one_{\{\xi_i \leq t\}}\right] + O(n^{-10})\,.
\]
Exactly as in \cref{eq:Fi-upper-bnd,eq:integration}, we have
$ \P\left(\xi_i \in (s,s+h)\right)\leq 2h \binom{|W_i|}2 e^{-2\theta s}$, and so, if we let $G_i$ denote the distribution function of $\xi_i$, then 
\[
\E\left[d_\tv^{\x_0}(v_i^{\xi_i},t-\xi_i)\one_{\{\xi_i \leq t\}}\right] \leq \int_0^{t} d_\tv^{\x_0}(v_i^r,t - r) \d G_i(r) \leq 6 \int_0^{t} d_\tv^{\x_0}(v_i^r,t - r) e^{-2\theta r} \d r\,.
\]
Combining the last 4 displays, along with $(a+b)^2 \leq 2(a^2+b^2)$ and Jensen's inequality, we get
\begin{align*}
d_\tv^{\x_0}(W_i,t)^2 &\leq 72 \left(\int_0^{t} d_\tv^{\x_0}(v_i^r,t - r) e^{-2\theta r} \d r\right)^2 + 18 e^{-4 \theta t} +O(n^{-10}) \\
&\leq 72 t\int_0^{t} d_\tv^{\x_0}\left(v_i^r,t - r\right)^2 e^{-4\theta r} \d r + 18 e^{-4 \theta t}  + O(n^{-10})\,.
\end{align*}
This yields
\begin{equation}
    \label{eq:dtv-Wi-m}
\sum_{i=1}^m d_\tv^{\x_0}(W_i,t)^2 \leq 72 t\int_0^{t} \Big(\sum_{i=1}^m d_\tv^{\x_0}(v_i^r,t - r)^2\Big) e^{-4\theta r} \d r + 18 ne^{-4 \theta t}  + O(n^{-9})\,.
\end{equation}
Since $\dist_G(W_i,W_j) \geq \frac{1}{2}(\log n)^{1+\beta}$, we have $\dist_G(W_i',W_j') >\frac{1}{4}(\log n)^{1+\beta}>0$
by our definition of the $W_i'$. In particular, the vertices $\{v_i^r\}$ are all distinct for any fixed $r$, and so we may replace the sum over $v_i^r$ ($i=1,\ldots,m$) by one running over all $n$ vertices:
\[
\sum_{i=1}^m d_\tv^{\x_0}\left(v_i^r,t - r\right)^2 \leq \sum_{v}d_\tv^{\x_0}\left(v,t - r\right)^2 
\leq 2|E|\sum_v \pi_G(v) d_\tv^{\x_0}\left(v,t-r\right)^2
\,,\]
using that $2|E| \pi_G(v) = d_G(v) \geq 1$.
Observe that $d^{\x_0}_\tv(v,s)$ is nothing but the total variation distance of the single-site marginal $X_t(v)$ from the uniform distribution:
\[ d^{\x_0}_\tv(v,s)^2 = \bigg(\frac12\sum_{\omega\in\Cq}\Big|\P_{\x_0}(X_s(v)=\omega)-\frac1q\Big|\bigg)^2 \leq \frac{q}4 \sum_{\omega\in\Cq} \Big(\P_{\x_0}(X_s(v)=\omega)-\frac1q\Big)^2\,,\]
which, upon recalling the definition of $\autII_t(\x_0)$ from \cref{def:autocorrelationL2}, implies that
\[ \sum_{i=1}^m d_{\tv}^{\x_0}\left(v_i^r,t-r\right)^2 \leq \frac q2 |E| \autII_{t-r}(\x_0)\,. \]
Revisiting \cref{eq:dtv-Wi-m}, and writing $|E| \leq \Delta n/2$ (noting in passing that the maximum degree $\Delta$ is less than $\gKap e$ by \cref{eq:expansion}), we get
\[
\sum_{i=1}^m d_\tv^{\x_0}(W_i,t)^2 \leq 18 q \Delta n \left(t\int_0^{t} \autII_{t-r}(\x_0) e^{-4\theta r} \d r + e^{-4 \theta t} \right) + O(n^{-9})\,.
\]
Now, by Proposition \ref{prop:autocorrelation}, there exist $\alpha_l \geq0$ with $\sum_l \alpha_l = \autII_0(\x_0) = \frac{q-1}q$ and exponents $\gamma_l\in[\theta,2-\theta]$ such that $\autII_{t-r}(\x_0) = \sum_{l=1}^n \alpha_l e^{-2\gamma_l(t-r)}$, and so it follows that
\begin{align*}
\int_0^{t} \autII_{t-r}(\x_0) e^{-4\theta r} \d r &= \sum_l \alpha_l\int_0^{t} e^{-2\gamma_l(t-r)-4\theta r} \d r \leq \sum_l \alpha_l\int_0^{t} (e^{-2\gamma_l t} + e^{-4\theta t})\d r  \\&= t\autII_t(\x_0) + \frac{q-1}{q} t e^{-4\theta t}\,,    
\end{align*}
where we used $e^{-2\gamma_l(t-r)-4\theta r} \leq e^{-(2\gamma_l\,\wedge\, 4\theta)t} $. This concludes the proof.
\end{proof}

\begin{proof}[Proof of \cref{lem:dtv-U-bound}]
Using \cref{lemma:reductiontoA2}, we deduce from \cref{eq:lastbound} that
\begin{align}\label{eq:sum-E-Ui-t1-tau}
\sum_{i=1}^m\E\left[d_\tv^{\x_0}(U_i^{\tau_i},t_1-\tau_i)^2 \one_{\{\tau_i<t_1\}}\right] \leq &Ct_1^2\, e^{4(\log n)^{1-\frac{\gAlpha}{2}}}n\int_0^{t_1}(\autII_{t_1 - s}(\x_0)+e^{-4\theta (t_1-s)})e^{-4\theta s}\d s
 + O(n^{-9})\,.
\end{align}
Put $s = b_1 (\log n)^{1-\frac\gAlpha2}$ for $b_1=3/\theta$ as per \eqref{eq:t1-def}, whence $t_1 = t_0 + s$. Writing $\autII_t(\x_0) = \sum_{l=1}^n \alpha_l e^{-2\gamma_l t}$ as in the above lemma, with $\alpha_l\geq 0$ and $\sum_l \alpha_l = \frac{q-1}q < 1$, we use that $1+\frac{q}{q-1}\leq 3$ to bound
\begin{align*}
\autII_{t_1 - s}(\x_0)+e^{-4\theta (t_1-s)} &\leq 3\sum_{l=1}^n \alpha_l e^{-(2\gamma_l\;\wedge\; 4\theta)(t_1-s)} = 3\sum_{l=1}^n \alpha_l\left(e^{-2\gamma_l t_1} \vee e^{-4\theta t_1}\right)e^{(2\gamma_l\;\wedge\; 4\theta) s}\\
&\leq 3 \left(\autII_{t_1}(\x_0)\;\vee\; ne^{-4\theta t_1}\right)  e^{4\theta s}\,.
\end{align*}
Plugging this in \cref{eq:sum-E-Ui-t1-tau}, the integral there is bounded from above by $3 t_1 (\autII_{t_1}(\x_0) \vee e^{-4\theta t_1})$, so
\[ \sum_{i=1}^m\E\left[d_\tv^{\x_0}(U_i^{\tau_i},t_1-\tau_i)^2 \one_{\{\tau_i<t_1\}}\right] \leq C't_1^3\, e^{4(\log n)^{1-\frac{\gAlpha}{2}}}n \left(\autII_{t_1}(\x_0) \;\vee\; e^{-4\theta t_1}\right) + O(n^{-9})\,. \]    
Combined with \cref{eq:dtv-Ui-bnd-1}, this completes the proof.
\end{proof}

\subsection{Proof of upper bound in \cref{thm:main}}
Combining \cref{lem:reduce-X-t1(V_i)-product-mu,lem:Xt-from-V-to-U,lem:dtv-U-bound}, we get
\[
d^{\x_0}_\tv(t_2)^2 \leq 
C t_1^3\,e^{4(\log n)^{1-\frac{\gAlpha}{2}}}n\left(\autII_{t_1}(\x_0)\;\vee\;e^{-4\theta t_1}\right) + O\big(e^{-8(\log n)^{1-\frac\gAlpha2}}\big)\,.
\]
By definition of $t_1$ in \cref{eq:t1-def} and $T_{\x_0}$ in \cref{def:Tx0}, using \cref{eq:A2submultiplicativity1} from \Cref{prop:autocorrelation} yields 
\[n
\left(\autII_{t_1}(\x_0)\;\vee\;e^{-4\theta t_1}\right) \leq e^{-2\theta b_1(\log n)^{1-\frac{\gAlpha}{2}}}\,.\] Hence, the choice $b_1=(3/\theta)(\log n)^{1-\frac\gAlpha2}$ shows that
\[
d^{\x_0}_\tv(t_2)^2 \leq  O\big(t_1^3 e^{-2(\log n)^{1-\frac\gAlpha2}}\big) < C e^{-(\log n)^{1-\frac\gAlpha2}}\,,\]
for some other $C>0$ depending only on $\theta,q,\gAlpha,\gKap$. This establishes \cref{eq:ubgoal} and concludes the proof of the upper bound.
\qed

\begin{remark}\label{remark:uniform}
The proof holds verbatim if we replace $\x_0$ with the uniform initial condition $\cU$.
 However, as $\autII_t(\cU)=0$, the $T_{\x_0}$ part can be dropped, thus the upper bound in \cref{eq:uniformmixing} follows.
\end{remark}

\section{Lower bounds}\label{sec:lowerbounds}

\subsection{Autocorrelation lower bound}
In this section, we show the $\tmix^{\x_0} \gtrsim T_{\x_0}$ part of \cref{eq:mixing}. 

Let $G$ be a graph satisfying \cref{eq:expansion} for some $\gAlpha,\gKap$, and fix a configuration $\x_0$. We will show a stronger bound: there exists some $C^\dagger=C^\dagger(\theta,q,\gAlpha,\gKap)>0$ such that, for every $\epsilon>0$, if
\begin{equation}
    \label{eq:t-dagger-def}
t^\dagger = T_{\x_0} - \frac1{2\theta}\log(1/\epsilon)-C^\dagger
\end{equation}
(where $T_{\x_0}$ is as in \cref{def:Tx0}), then
\begin{equation}\label{eq:lbautocorrelationgoal}
\|\P_{\x_0}(X_{t^\dagger}\in\cdot)-\mu_G\|_\tv >1-\epsilon\,,
\end{equation}
and in particular, $\tmix^{\x_0}(1-\epsilon)>t^\dagger$.

To establish this, define the statistics
\begin{equation}\label{eq:R(X)-stat}
\cR_t^{\x_0}(\x) = \sum_{v\in V} \pi_G(v)\sum_{\omega\in\Cq}a_{v,\omega,t}\Big(\one_{\{\x(v)=\omega\}}-\frac1q\Big)\,,
\end{equation}
where 
\[
a_{v,\omega,t} = \P_{\x_0}(X_t(v)=\omega)-\frac{1}{q}\qquad\mbox{for all $v\in V$ and $\omega\in\Cq$}\,.
\]
Notice that $\E_{\mu_G}[\cR_t^{\x_0}(Y)] = 0$ (since $\P_{\mu_G}(Y(v)=\omega)=1/q$ for all $v,\omega$), whereas
$\E_{\x_0}[\cR_t^{\x_0}(X_t)]$ is nothing but the autocorrelation:
\begin{align*}
\E_{\x_0}[\cR_t^{\x_0}(X_t)] &= \sum_{v\in V}\pi_G(v)\sum_{\omega\in\Cq} a_{v,\omega,t}^2 = \sum_{v\in V} \pi_G(v)\sum_{\omega\in\Cq}\left(\P_{\x_0}(X_t(v)=\omega)-\frac{1}{q}\right)^2 
= \autII_t(\x_0)
\end{align*}
as given in \cref{def:autocorrelationL2}. 

We next need a bound on the variance of $\cR_t^{\x_0}(X_t)$:
\begin{claim}\label{clm:var-R(Xt)}
There exists $C_0 = C_0(\theta,q,\gAlpha,\gKap)$ such that
\[
\Var_{\x_0}(\cR_t^{\x_0}(X_t)) \leq \frac{C_0}{n}\autII_t(\x_0) \,, \quad \Var_{\mu_G}(\cR_t^{\x_0}(Y)) \leq \frac{C_0}{n}\autII_t(\x_0)\,.
\]
\end{claim}
\begin{proof}
It will be convenient to work with the following formula for $\cR_t^{\x_0}(X_t)$, equivalent to \cref{eq:R(X)-stat}:
\[
\cR_t^{\x_0}(\x) = \sum_{v\in V} \pi_G(v)\sum_{\omega\in\Cq}a_{v,\omega,t}\one_{\{\x(v)=\omega\}}\,;
\]
indeed, the equivalence is due to the fact that $\sum_{\omega} a_{v,\omega,t} = 0$ for every $v$ and $t$. We thus have
\begin{align*}
\Var_{\x_0}(\cR_t^{\x_0}(X_t)) &= \sum_{u,v\in V}\sum_{\omega_1,\omega_2\in \Cq}\pi_G(u)\pi_G(v)a_{u,\omega_1,t}a_{v,\omega_2,t}\Cov_{\x_0}(\one_{\{X_t(u)=\omega_1\}},\one_{\{X_t(v)=\omega_2\}})\\ 
&\leq \sum_{u,v} \sum_{\omega_1,\omega_2}\pi_G(u)\pi_G(v)\left(\frac{a_{u,\omega_1,t}^2+a_{v,\omega_2,t}^2}{2}\right)\left|\Cov_{\x_0}(\one_{\{X_t(u)=\omega_1\}},\one_{\{X_t(v)=\omega_2\}})\right|\,,
\end{align*}
which, by the symmetry between $(u,\omega_1)$ and $(v,\omega_2)$, is in turn equal to 
\begin{align*}
\sum_{u,\omega_1}&\pi_G(u)a_{u,\omega_1,t}^2 \sum_{v,\omega_2}\pi_G(v)\left|\Cov_{\x_0}(\one_{\{X_t(u)=\omega_1\}},\one_{\{X_t(v)=\omega_2\}})\right|\\
&\leq \frac{\Delta q}{n}\sum_{u,\omega_1}\pi_G(v)a_{v,\omega,t}^2\max_{\omega_2}\sum_v\left|\Cov_{\x_0}(\one_{\{X_t(u)=\omega_1\}},\one_{\{X_t(v)=\omega_2\}})\right|\,.
\end{align*}
(In the last inequality we used $\pi_G(v) \leq \Delta/n$ where $\Delta\leq \gKap e$ is the maximum degree of $G$.) 

We next argue that the last sum over $v$ is bounded by some $C(\theta,\gKap,\gAlpha)$ for all $u,\omega_1,\omega_2$. To see this, let $X_t'$ be a an independent copy of $X_t$, whence
\[ \left|\Cov_{\x_0}(\one_{\{X_t(u)=\omega_1\}},\one_{\{X_t(v)=\omega_2\}})\right|=\left|\P(X_t(u)=\omega_1,\,X_t(v)=\omega_2)-\P(X_t(u)=\omega_1,\, X'_t(v)=\omega_2)\right|\,,\]
and let $\tau$ be the minimum time $s\geq0$ in the backward dynamics where the walks $(Z^u_s)$ and $(Z^v_s)$ coalesce. We can couple $(X_t(u),X_t(v))$ with $(X_t(u),X_t'(v))$ in such a way that the two pairs are identical if $\tau>t$ (the walks do not meet until until the backward dynamics reaches time $0$). Under this coupling, the expression for 
$\left|\Cov_{\x_0}(\one_{\{X_t(u)=\omega_1\}},\one_{\{X_t(v)=\omega_2\}})\right|$ simplifies to
\begin{align*} 
\left|\P(X_t(u)=\omega_1 ,\,X_t(v)=\omega_2,\,\tau\leq t)-\P(X_t(u)=\omega_1,\, X'_t(v)=\omega_2,\tau\leq t)\right| \leq \P(\tau\leq t)\,.
\end{align*} 
Note that $\P(\tau \leq t) < \P(\tau < \infty)= \P(u \leftrightsquigarrow v)$ (the probability that the walks started at $u$ and $v$ meet before any of them dies). If $\dist_G(u,v)=r$, then in order for the walks at $u$ and $v$ to possibly meet, the first $r$ updates that either of them receive must be non-killing (with probability $1-\theta$ each), and so $\P(u \leftrightsquigarrow v) \leq (1-\theta)^r \leq e^{-\theta r}$. Hence, using \cref{eq:expansion},
\begin{align*}
\max_{\omega_2}\sum_v\left|\Cov_{\x_0}(\one_{\{X_t(u)=\omega_1\}},\one_{\{X_t(v)=\omega_2\}})\right| &\leq \sum_{r \geq 0}|\{v \, : \; \dist_G(u,v)=r\}|e^{-\theta r} \\ &\leq \gKap \sum_{r \geq 0}e^{ r^{1-\gAlpha}-\theta r }=:C_1\,.
\end{align*}
In conclusion,
\[
\Var_{\x_0}(\cR_t^{\x_0}(X_t)) < \frac{\Delta C_1 q}n\sum_{v}\pi_G(v)a_{v,\omega,t}^2 = \frac{\Delta C_1 q}{n}\autII_t(\x_0)\,.
\]
The same proof works for $\Var_{\mu_G}(\cR_t^{\x_0}(Y))$, noting directly that under the same coupling as above 
\[
|\Cov_{\mu_G}(\one_{\{Y(u)=\omega_1\}},\one_{\{Y(v)=\omega_2\}})| \leq \P(\tau <\infty) = \P(u \leftrightsquigarrow v) \leq \exp[-\theta \dist_G(u,v)]\,,
\]
the same upper bound we used for the covariance when estimating $\Var_{\x_0}(\cR_t^{\x_0}(X_t))$.
\end{proof}
We are ready to prove \cref{eq:lbautocorrelationgoal}. Recall that $\E_{\x_0}[\cR_t^{\x_0}(X_t)] = \autII_t(\x_0) > 0 = \E_{\mu_G}[\cR_t^{\x_0}(Y)]$, and define the set of configurations
\[
E^\dagger = \bigg\{\x \in \Cq^V \, : \;  \cR_{t^\dagger}^{\x_0}(\x) < \frac12 \autII_{t^\dagger}(\x_0)\bigg\}\,.
\]
By Chebyshev's inequality,
\begin{align*}
\P_{\x_0}(X_{t^\dagger} \in E^\dagger) = \P_{\x_0}\Big(\cR_{t^\dagger}^{\x_0}(X_{t^\dagger}) -\E_{\x_0}[\cR_{t^\dagger}^{\x_0}(X_{t^\dagger})] < -\frac{\autII_{t^\dagger}(\x_0)}2\Big)
\leq \frac{4\Var_{\x_0}(\cR_{t^\dagger}^{\x_0}(X_{t^\dagger}))}{\autII_{t^\dagger}(\x_0)^2} \leq \frac{4 C_0}{n\autII_{t^\dagger}(\x_0)}\end{align*}
using \cref{clm:var-R(Xt)} (with $C_0$ from that claim).
Recalling from \cref{eq:t-dagger-def} that $t^\dagger = T_{\x_0}-\frac1{2\theta} \log(1/\epsilon) - C^\dagger$, we see that if $C^\dagger = \frac1{2\theta}\log(8C_0)$ then the sub-multiplicativity of $\autII$, as per \cref{eq:A2submultiplicativity1}, yields
\[ \P_{\x_0}(X_{t^\dagger}\in E^\dagger) 
\leq \frac{4C_0}{n\autII_{T_{\x_0}}(\x_0) e^{2\theta (T_{\x_0}-t^\dagger)}} = \frac{\epsilon}2\,,
\]
with the last equality following from the definition of $\autII$ in \cref{def:autocorrelationL2}.
Similarly, 
\begin{align*}
\P_{\mu_G}(Y \notin E^\dagger) = \P_{\mu_G}\Big(\cR_{t^\dagger}^{\x_0}(Y)-\E_{\mu_G}[\cR_{t^\dagger}^{\x_0}(Y)] > \frac{\autII_{t^\dagger}(\x_0)}2\Big)  
\leq \frac{4\Var_{\mu_G}(\cR_{t^\dagger}^{\x_0}(Y))}{\autII_{t^\dagger}(\x_0)^2} &\leq \frac{\epsilon}2\,.
\end{align*}
Combining the last two displays gives
\[
\|\P_{\x_0}(X_{t^\dagger} \in \cdot) - \mu_G\|_\tv \geq \P_{\mu_G}(Y \in E^\dagger)  - \P_{\x_0}(X_{t^\dagger} \in E^\dagger) \geq 1-\epsilon\,,
\]
thereby establishing \cref{eq:lbautocorrelationgoal}.
\qed

\subsection{Correlation lower bound}
In this section, we show the $\tmix^{\x_0} \gtrsim \frac{1}{4\theta}\log n$ part of \cref{eq:mixing} and \cref{eq:uniformmixing}. Because the proofs of the two are essentially the same, we simply prove \cref{eq:mixing} and describe in \Cref{remark:uniformlb} the single modification needed in the case of a uniform initial state $\cU$.

Let $G=(V,E)$ be a connected graph with $|V|=n$ satisfying \cref{eq:expansion} for some constants $\gAlpha,\gKap$. We will show that there exists a constant $C^\ddagger>0$ depending only on $(\theta,\gAlpha,\gKap)$ such that if
\begin{equation}\label{eq:t-ddagger-def}
t^\ddagger = \frac{1}{4\theta}\log n -C^\ddagger (\log n)^{1-\frac\gAlpha2}\,,
\end{equation}
then, for all initial states $\x_0$,
\begin{equation}\label{eq:lbsecondordergoal}
\|\P_{\x_0}(X_t\in\cdot)-\mu_G\|_\tv \geq 1- O(e^{-2 (\log n)^{1-\frac\gAlpha2}})\,,
\end{equation}
implying in particular that, for every $0<\epsilon<1$ fixed, $\tmix^{\x_0}(1-\epsilon)>t^\ddagger$ provided $n$ is large enough.

For fixed $\x_0,t$, we consider the following distinguishing statistics:
\begin{equation}\label{def:secondorderstatistics}
\cR_t^{\x_0}(\x) = \sum_{u \sim v} \left(\one_{\{\x(u)=\x(v)\}} - \P_{\x_0}\big(X_t(v)=\x(u)\big)\right)\,.
\end{equation}
(Notice $\E_{\x_0}[\cR_t(X_t)]$ does not simplify to $0$ but to
$ \sum_{u\sim v}[\P_{\x_0}(X_t(u)=X_t(v)) - \P_{\x_0}(X_t(u)=X'_t(v))]$ for an i.i.d.\ instance of the chain $X'_t$.) 
By the routine fact $\sum_{k=0}^{q-1} \omega^k = q\one_{\{\omega=1\}}$ for all $\omega \in \Cq$, we can equivalently write $\cR_t^{\x_0}(\x)$ as
\begin{equation}\label{eq:Rt-ham-def}
\cR_t^{\x_0}(\x) = \frac{1}{q}\sum_{k=1}^{q-1}\sum_{u \sim v} \left(\x(u)^k\bar{\x(v)^k} - \x(u)^k\bar{\E_{\x_0}[X_t(v)^k]}\right)
\end{equation}
(NB.\ we omitted the $k=0$ term, as it has no contribution). Then, as $\E_{\mu_G}[Y(v)^k]=\frac{1}{q}\sum_{\omega\in\Cq} \omega^k=0$ for all $k=1,\ldots,q-1$, we have
\begin{equation}\label{eq:meandifference}
\E_{\mu_G}[\cR_t^{\x_0}(Y)] - \E_{\x_0}[\cR_t^{\x_0}(X_t)] = \frac{1}{q}\sum_{k=1}^{q-1}\sum_{u \sim v}\left(\Cov_{\mu_G}(Y(u)^k,Y(v)^k) - \Cov_{\x_0}(X_t(u)^k,X_t(v)^k)\right)\,.
\end{equation}
While the left-hand is clearly real-valued (e.g., via \cref{eq:Rt-ham-def}), a-priori the individual terms we sum over in \cref{eq:meandifference} might not be. The next claim shows that they  are nonnegative real numbers bounded from below by the probability that two random walks coalesce after time $t$.
\begin{claim}\label{claim:covariance}
For every $\x_0 \in \Cq^V$, every two vertices $u,v \in V$ and all $k=1,\ldots,q-1$, 
\begin{equation}\label{eq:covariance}
\Cov_{\mu_G}(Y(u)^k, Y(v)^k) - \Cov_{\x_0}(X_t(u)^k,X_t(v)^k) - \P\big(u \overset{>t}{\leftrightsquigarrow} v\big)\in\R_+\,, 
\end{equation}
where $\big\{u \overset{>t}{\leftrightsquigarrow} v\big\}$ denotes the probability that two random walks $(Z_s^u),(Z_s^v)$ started at $u$ and $v$ with killing rate $\theta$ (and moving rate $1-\theta$) do not coalesce until time $t$, and then coalesce after time $t$.
\end{claim}
\begin{proof}
Let $\tau$ be the first time that the walks started from $u$ and $v$ coalesce in the backward dynamics (with $\tau=+\infty$ if either walk gets killed before coalescing). If $X_t'(v)$ denotes a variable that has the same distribution as $X_t(v)$ but is independent of $X_t(u)$, we can couple $(X_t(u),X_t(v))$ with $(X_t(u),X_t'(v))$ by letting the two corresponding random walks $Z^v_s,Z^{\prime\, v}_s$ in the backward dynamics started at $v$ be identical for $s \leq \tau$, and then run independently for $s > \tau$. When doing so, one has $(X_t(u),X_t(v))=(X_t(u),X_t'(v))$  when $\tau>t$, and thus
\begin{align*}
\Cov_{\x_0}(X_t(u)^k,X_t(v)^k) 
&= \E_{\x_0}\left[X_t(u)^k\bar{X_t(v)^k}-X_t(u)^k\bar{X_t'(v)^k}\right]\\
&= \E_{\x_0}\left[\left(X_t(u)^k\bar{X_t(v)^k}-X_t(u)^k\bar{X_t'(v)^k}\right)\one_{\{\tau \leq t\}}\right]\\
&= \P(\tau \leq t)  - \E_{\x_0}\left[X_t(u)^k\bar{X_t'(v)^k}\one_{\{\tau \leq t\}}\right]\,,
\end{align*}
where the final identity follows from the fact that $X_t(u)=X_t(v)$ on $\one_{\{\tau\leq t\}}$.

On the other hand, recall that $\mu_G$ can be perfectly simulated by letting the walks from $V$ run forever in the past until all of them are killed, and noting that if two walks do not coalesce (because one gets killed first) they independently get a uniform spin (with $0$ mean, again  by $\frac{1}{q}\sum_{\omega\in\Cq} \omega=0$). Thus,
\[
\Cov_{\mu_G}(Y(u)^k,Y(v)^k) = \E_{\mu_G}\left[Y(u)^k\bar{Y(v)}^k\right] =\P(\tau < \infty) \,.
\]
By definition, $\P(t < \tau < \infty) = \P\big(u \overset{>t}{\leftrightsquigarrow} v\big)$, so it follows that
\[
\Cov_{\mu_G}(Y(u)^k, Y(v)^k) - \Cov_{\x_0}(X_t(u)^k,X_t(v)^k) = \P\big(u \overset{>t}{\leftrightsquigarrow} v\big) + \E_{\x_0}[X_t(u)^k\bar{X_t'(v)^k}\one_{\{\tau \leq t\}}]\,.
\]
The proof will be finished if we prove that $\E_{\x_0}[X_t(u)^k\bar{X_t'(v)^k}\one_{\{\tau \leq t\}}]$ is real-valued and nonnegative. Indeed, if $\cF_\tau$ denotes the $\sigma$-algebra associated with the stopping time $\tau$ for the backward dynamics, we have
\begin{align*}
\E_{\x_0}\left[X_t(u)^k\bar{X_t'(v)^k}\one_{\{\tau \leq t\}}\right] &= \E_{\x_0}\left[\E\left[X_t(u)^k\bar{X_t'(v)^k}\mid\cF_\tau\right]\one_{\{\tau \leq t\}}\right] \\&= \E_{\x_0}\left[\E\left[X_{t-\tau}(w_\tau)^k\bar{X_{t-\tau}'(w_\tau)^k}\mid\cF_\tau\right]\one_{\{\tau \leq t\}}\right]\,,
\end{align*}
where $w_\tau$ denotes the vertex at which the walks from $u$ and $v$ meet at time $\tau$ (note that $w_\tau$ is $\cF_\tau$-measurable). Finally, since $X_{t-\tau}(w_\tau)$ and $X_{t-\tau}'(w_\tau)$ are conditionally independent given $\cF_\tau$,
\[
\E_{\x_0}\left[\E\left[X_{t-\tau}(w_\tau)^k\bar{X_{t-\tau}'(w_\tau)^k}\mid \cF_\tau\right]\one_{\{\tau \leq t\}}\right] = \E_{\x_0}\Big[\big|\E[X_{t-\tau}(w_\tau)^k\mid\cF_\tau]\big|^2\one_{\{\tau \leq t\}}\Big] \in \R_+\,,
\]
as claimed.
\end{proof}

\begin{remark}\label{remark:uniformlb}
To handle the uniform initial condition $\cU$ rather than a deterministic $\x_0$, redefine the statistics $\cR_t^{\x_0}(\x)$ in \cref{def:secondorderstatistics}  simply as
\[
\cR(\x)=\sum_{u \sim v} \one_{\{\x(u) = \x(v)\}}\,.
\]
Then,  \cref{eq:meandifference} holds true for the new statistics when we replace $\x_0$ by $\cU$. Moreover, in that case we have, in lieu of \cref{claim:covariance}, that
\[\Cov_{\mu_G}(Y(u)^k, Y(v)^k) - \Cov_{\cU}(X_t(u)^k,X_t(v)^k) = \P\big(u \overset{>t}{\leftrightsquigarrow} v\big)\,.\]
The remaining part of the proof in this section then holds verbatim once we replace $\x_0$ by $\cU$.
\end{remark}

Applying \cref{claim:covariance} to \cref{eq:meandifference} shows that, for all $q\geq 2$ (replacing $\frac{q-1}q$ by $\frac12$ in the lower bound),
\begin{equation}\label{eq:meandifference-intermediate}
\E_{\mu_G}[\cR_t^{\x_0}(Y)] - \E_{\x_0}[\cR_t^{\x_0}(X_t)] \geq \frac12 \sum_{u\sim v} \P\big(u \overset{>t}{\leftrightsquigarrow} v\big)\,.
\end{equation}
In light of \cref{eq:meandifference-intermediate}, 
we aim to show that $\P\big(u \overset{>t}{\leftrightsquigarrow}v\big)\approx e^{-2\theta t}$ for every $uv\in E$ (\cref{claim:meetinglater} below), so that one has
$\sum_{u\sim v} \P\big(u \overset{>t}{\leftrightsquigarrow} v\big) \approx n e^{-2\theta t}$, whereas 
$\Var_{\x_0}(\cR_t^{\x_0}(X_t))$ and $\Var_{\mu_G}(\cR_t^{\x_0}(Y))$ are $O(n)$ (see \cref{clm:var-R(Xt)} below); at that point, the proof will be concluded by Chebyshev's inequality.

Note that in order for the walks started from $u$ and $v$ to coalesce after $t$, both of them need to survive separately until time $t$, which incurs in a cost of $e^{-2\theta t}$ for survivability alone. Our next goal is to show that under the assumption of \cref{eq:expansion} and connectedness, for nearby vertices $u,v$ the cost for the remaining conditions (the two walks avoiding each other until time $t$ and coalescing later) is sub-exponential in $t\asymp \log n$. This, together with \cref{claim:covariance}, will yield the required lower bound. The key for this will be a result from \cite{OveisTrevisan12}, which bounds from below the probability of a random walk failing to escape a set $S$ through the conductance of $S$. 

Before applying said result in our setting, we prove the existence of ``good'' balls with limited expansion around every $v \in V$.
For a set of vertices $S$, the conductance $\Phi(S)$ of $S$ is the ratio
\[
\Phi(S) = \frac{|E(S,S^c)|}{\Vol(S)}\,,
\]
where $|E(S,S^c)|$ is the number of edges connecting $S$ to its complement $S^c$, and $\Vol(S)$ is the sum of the degrees of vertices in $S$. The next lemma shows that, for every graph $G=(V,E)$ of subexponential growth of balls as in \cref{eq:expansion} and every vertex $v\in V$, we can find a ball $B_v(r)$ in every scale with conductance at most $O(r^{-\gAlpha})$.

\begin{lemma}\label{lemma:trappingball}
Let $G$ be a graph satisfying \cref{eq:expansion} for some $\gKap > 0$ and $\gAlpha\in(0,1)$. Then, for every integer $\Rone \geq \gKap$ and every $v \in V$, there exists $r\in[\Rone, 2\Rone]$ such that
\[
\frac{|\partial^+ B_v(r)|}{|B_v(r)|} \leq \frac{8}{r^\gAlpha}\,,
\]
where $B_v(r)$ is the ball of radius $r$ in $G$ centered at the vertex $v$, and $\partial^+ B_v(r) = B_v(r+1)\setminus B_v(r)$.
In particular, its conductance has $\Phi(B_v(r)) \leq 8\Delta r^{-\gAlpha} $, where $\Delta \leq \gKap e$ is the maximum degree in $G$.
\end{lemma}
\begin{proof}
Denote $\Rtwo =2\Rone$, and suppose
by contradiction that the claim is false, so that there exist some $v\in V$ and $r\geq \gKap$ such that
\[
\frac{|\partial^+ B_v(r)|}{|B_v(r)|} > \frac{8}{r^\gAlpha}\,, \qquad \mbox{for all $r\in[\Rone, \Rtwo]$}\,.
\]
We can then write 
\begin{align*}
|B_v(\Rtwo)| 
= |B_v(\Rtwo-1)|\left(1+\frac{|\partial^+ B_v(\Rtwo-1)|}{|B_v(\Rtwo-1)|}\right) > |B_v(\Rtwo-1)|\left(1+\frac{8}{(\Rtwo-1)^{\gAlpha}}\right)\,,
\end{align*}
and iterating this argument shows that
\begin{align*}
|B_v(\Rtwo)| >  |B_v(\Rone)|\prod_{r=\Rone}^{\Rtwo-1}\left(1+\frac{8}{r^{\gAlpha}}\right)
 \geq \Rone \exp\Big[\sum_{r=\Rone}^{\Rtwo-1}\log\left(1+\frac{8}{r^{\gAlpha}}\right)\Big]\,.
\end{align*}
Noting that $\log(1+8x) \geq 2x$ for all $0\leq x\leq 1$ and that $\frac{1-2^{-(1-\gAlpha)}}{1-\gAlpha} \geq \frac12$ for all $0\leq \gAlpha\leq 1$, we find
\[
e^{\sum_{r=\Rone}^{\Rtwo-1}\log\left(1+8 r^{-\gAlpha}\right)} \geq e^{2\sum_{r=\Rone}^{\Rtwo-1}r^{-\gAlpha}} \geq e^{2\int_{\Rone}^{\Rtwo}r^{-\gAlpha}\d r} = e^{2 \frac{1-2^{-(1-\gAlpha)}}{1-\gAlpha}(\Rtwo)^{1-\gAlpha}} \geq e^{(\Rtwo)^{1-\gAlpha}}\,.
\]
Having arrived at $|B_v(\Rtwo)| > \Rone \exp[(\Rtwo)^{1-\gAlpha}]$ with $\Rone\geq \gKap$, this contradicts \cref{eq:expansion} as desired.
\end{proof}

For a set of vertices $S\subset $V, an integer $k$ and a distribution $\nu$ on $S$, define
\[ \rho_{\nu}^{\textsf{lazy}}(S,k)=\P_\nu(\mbox{lazy random walk remains in $S$ for $k$ steps})\,.\]
(Here and in what follows, a lazy random walk refers to one that stays in place with probability $\frac12$.)
Oveis Gharan and Trevisan~\cite[Prop.~8]{OveisTrevisan12} proved that on for any graph $G$, 
if $\pi_S$ is the restriction of the stationary distribution $\pi_G$ to $S$ (i.e., $\pi_S(v)= d_G(v)/\Vol(S)$ for $v\in S$), then
\begin{equation*}
\rho_{\pi_S}^{\textsf{lazy}}(S,k) \geq \left(1-\frac{\Phi(S)}{2}\right)^k\,.
\end{equation*}
This extends to a continuous-time random walk with rate $\lambda$: by a time-rescaling, said walk is equivalent to a continuous-time lazy random walk with update rate $2\lambda$. There, the number of updates by time $t$ is $N_t\sim\text{Poisson}(2\lambda t)$, and given $N_t=k$ one can apply the bound on $\rho_{\pi_S}^{\textsf{lazy}}(S,k)$. It follows that  the probability $\rho_{\pi_S}^{\textsf{cont}}(S,t,\lambda)$ that a continuous-time random walk with rate $\lambda$ started from $v\in\pi_S$ stays completely inside $S$ until time $t$ satisfies
\begin{equation}\label{eq:ctsconductancebound}
\rho_{\pi_S}^{\textsf{cont}}(S,t,\lambda) \geq \E\left[\left(1-\Phi(S)/2\right)^{N_t}\right] = e^{-\Phi(S)\lambda t}\,.
\end{equation}
We will use this to prove the following:

\begin{claim}\label{claim:meetinglater}
For every $t=t(n) \geq 0$, any edge $u v\in E$ and any integer $\Rone$ with $\gKap \leq \Rone \leq \frac{1}{10}\diam(G)$, 
\[
\P\big(u \overset{>t}{\leftrightsquigarrow} v\big) \geq 
\exp\big[-2\theta t -C_1 \Rone - C_2 \Rone^{-\gAlpha}\, t\big] \,,
\]
where $C_1,C_2>0$ are constants depending only on $\theta,\gAlpha,\gKap$.
\end{claim}
\begin{proof}

The strategy will be the following: let the walks $Z_s^u$ and $Z_s^v$ from $u$ and $v$, respectively, move quickly to two suitable locations, confine them to separate balls until time $t$ using \cref{eq:ctsconductancebound} for each of them and then let them quickly coalesce after time $t$. 

Applying \cref{lemma:trappingball} for $u$ and $\Rone$ yields $r \in [\Rone,2\Rone]$ such that $S := B_u(r)$ has $\Phi(S) \leq  8 \Delta r^{-\gAlpha}$.
Let $\delta_x$ denote the point mass distribution at $x$, and consider $\rho_{\delta_{x}}^{\textsf{cont}}(S,t,1-\theta)$, the probability that the continuous-time random walk with moving rate $1-\theta$, started at $x\in S$, remains in $S$ for time~$t$. From \cref{eq:ctsconductancebound}, we see that the vertex $u_*$ maximizing $\rho_{\delta_{x}}^{\textsf{cont}}(S,t,1-\theta)$ over $x\in S$ must satisfy
\begin{equation}\label{eq:ctsbounds-u}\rho_{\delta_{u_*}}^{\textsf{cont}}(S,t,1-\theta) \geq 
\rho_{\pi_S}^{\textsf{cont}}(S,t,1-\theta) \geq e^{-8 \Delta r^{-\gAlpha} (1-\theta)t}\,.
\end{equation}
By connectedness of $G$ and the fact that $\diam(G)\geq 10 \Rone$, there exists a (shortest) path from the edge $uv$ to a vertex at some distance $>4\Rone$. Assume without loss of generality that such a path starts at $v$ and (being a shortest path) does not visit $u$, and let $w$ be a vertex along this path with $\dist_G(u,w)=4\Rone+1$. 
Applying \cref{lemma:trappingball} for $w$ and $\Rone$ yields some $r' \in [\Rone,2\Rone]$ such that $S' := B_w(r')$ has $\Phi(S) \leq  8 \Delta r^{-\gAlpha}$
(note that $S \cap S' = \emptyset$ thanks to the condition $\dist_G(u,w)=4\Rone+1$).
As argued above for $B_u(r)$, there must exist $w_*\in S'$ such that
\begin{equation}\label{eq:ctsbounds-w}\rho_{\delta_{w_*}}^{\textsf{cont}}(S',t,1-\theta) \geq e^{-8 \Delta (r')^{-\gAlpha} (1-\theta)t}\,.
\end{equation}

We can now realize a sub-event of $\big\{u \overset{>t}{\leftrightsquigarrow} v\big\}$ as follows: 
\begin{enumerate}[(i)]
\item force the first $\dist_G(v,w)+\dist_G(w,w_*) + \dist_G(u,u_*) \; (\leq (4\Rone+1)+2\Rone+2\Rone = 8\Rone+1)$ cumulative updates that $Z_s^u$ and $Z_s^v$ receive in the time interval $s\in [0,t]$ to move $Z_s^v$ from $v$ to $w$ to $w_*$ (while avoiding~$u$), and then move $Z_s^u$ from $u$ to $u_*$ (the probability that each update is non-killing and moves the specified random walk along the specified path is at least $(1-\theta)\cdot\frac{1}{2}\cdot\frac{1}{\Delta}$, where $\Delta$ is the maximum degree of $G$);
\item for the remaining time until $t$ (if not surpassed already in the previous step), force the two walks now at $u_*$ and $w_*$ to stay inside the disjoint $S=B_u(r)$ and $S'=B_{w}(r')$ and receive no killing updates (the non-escape probabilities are bounded from below by \cref{eq:ctsbounds-u,eq:ctsbounds-w}, while the non-killing probability is bounded from below by $e^{-2\theta t}$);
\item force the next cumulative updates that $Z_s^u$ and $Z_s^v$ receive (after $t$) to make the two walks coalesce by going through the minimal path between their current locations (the distance between them is at most $8\Rone+1$, and the probability that each sequential update is non-killing and makes the correct move along the minimal path is at least $(1-\theta)\cdot\frac{1}{\Delta}$).
\end{enumerate}
Hence, 
\begin{align*}
\P\big(u \overset{>t}{\leftrightsquigarrow} v\big) &\geq \left(\frac{1-\theta}{2\Delta}\right)^{8\Rone+1}  e^{-2\theta t}e^{-8\Delta \left[r^{-\gAlpha} +(r')^{-\gAlpha}\right](1-\theta)t}\left(\frac{1-\theta}{\Delta}\right)^{8\Rone+1} \\ &\geq e^{-2\theta t} \left(\frac{1-\theta}{2\Delta}\right)^{20\Rone} e^{-16\Delta \Rone^{-\gAlpha}(1-\theta)t }\,,
\end{align*}
giving the desired bound for  $C_1=20\log\left(\frac{2\Delta}{1-\theta}\right)$ and $C_2=16\Delta(1-\theta) $.
\end{proof}

Using \cref{claim:meetinglater} with 
\begin{equation}\label{eq:rn-choice}
    \Rone = (\log n)^{\frac1{1+\gAlpha}} < (\log n)^{1-\frac\gAlpha2}
\end{equation}
(noting $\Rone=o(\diam(G))$ as every graph $G$ with maximum degree $\Delta$ has $\diam(G) \geq \log_{\Delta-1}(n)-2$), 
whereby $\Rone \asymp \Rone^{-\alpha} \log n$ and so both error terms in that claim will have the same order for $t=t^\ddagger$,
it follows from \cref{eq:meandifference-intermediate} that, for all $t \leq \frac1{4\theta}\log n$ and every sufficiently large $n$,
\begin{equation}\label{eq:meandifferencefinal}
\E_{\mu_G}[\cR_t^{\x_0}(Y)] - \E_{\x_0}[\cR_t^{\x_0}(X_t)] \geq \frac{n e^{-2\theta t}}{2 e^{C (\log n)^{1-\frac{\gAlpha}{2}}}}\,,
\end{equation}
for some fixed $C=C(\theta,\gAlpha,\gKap)>0$ (e.g., take $C = C_1 + C_2/(4\theta)$ for $C_1,C_2$ from that claim).

The final ingredient is an upper bound on the variance of the distinguishing statistics $\cR_t^{\x_0}$. 
\begin{claim}\label{clm:variance-R-corr}
Set $C_0 = 2\Delta^2/\theta$, where $\Delta$ is the maximum degree in $G$. Then for every $t>0$ and $\x_0$,
\[
\Var_{\x_0}(\cR_t^{\x_0}(X_t)) \leq C_0 n\,, \quad \Var_{\mu_G}(\cR_t^{\x_0}(Y)) \leq C_0 n\,.
\]
\end{claim}
\begin{proof}
Denoting $\cR_{u,v}(X) = \one_{\{X(u)=X(v)\}} - \E_{\x_0}[\one_{\{X(u)=X_t(v)\}}]$, we have
\begin{align*}
\Var_{\x_0}(\cR_t^{\x_0}(X_t)) &= \sum_{u_1 \sim v_1,u_2 \sim v_2}\Cov_{\x_0}(\cR_{u_1,v_1}(X_t), \cR_{u_2,v_2}(X_t))\,.
\end{align*}
Let $X_t'$ be an independent copy of $X_t$. As in the proof of \cref{claim:covariance}, we couple the random walks $(Z_s^{u_2},Z_s^{v_2})$ 
with their counterparts
$(Z_s^{\prime\,u_2},Z_s^{\prime\,v_2})$ 
so that they are identical 
until time~$\tau$, the first time (if finite) that the backward dynamics 
has one of the random walks started from $\{u_1,v_1\}$ meet with one of the walks started from $\{u_2,v_2\}$. (Beyond that time, we  run each chain independently.) 
In particular, the $4$-tuples $(X_t(u_1),X_t(v_1), X_t(u_2), X_t(v_2))$ and $(X_t(u_1),X_t(v_1), X_t'(u_2), X_t'(v_2))$ are identical if $\tau>t$.
Letting $\{x \leftrightsquigarrow y\}$ denote the event that two walks started at $x$ and $y$ with killing rate $\theta$ and moving rate $1-\theta$ meet before either of them dies, we get
\begin{align*}
\Cov_{\x_0}(\cR_{u_1,v_1}(X_t), \cR_{u_2,v_2}(X_t)) &= \E_{\x_0}\left[\cR_{u_1,v_1}(X_t)\cR_{u_2,v_2}(X_t) - \cR_{u_1,v_1}(X_t)\cR_{u_2,v_2}(X_t')\right]\\
&= \E_{\x_0}\left[\left(\cR_{u_1,v_1}(X_t)\cR_{u_2,v_2}(X_t) - \cR_{u_1,v_1}(X_t)\cR_{u_2,v_2}(X_t')\right)\one_{\{\tau \leq t\}}\right]\\
&\leq 2\P(\tau \leq t)\,.
\end{align*}
Relaxing $\{\tau \leq t\}$ into $\{\tau <\infty\}$, followed by a union bound, shows 
\[
\P(\tau\leq t) \leq 
\P(u_1 \leftrightsquigarrow u_2)+\P(u_1 \leftrightsquigarrow v_2)+\P(v_1 \leftrightsquigarrow u_2)+\P(v_1 \leftrightsquigarrow v_2)\,.
\]
and combining the last three displays yields
\begin{align*}
\Var_{\x_0}(\cR_t^{\x_0}(X_t)) \leq 2\sum_{u_1, u_2}d_G(u_1)d_G(u_2)\P(u_1\leftrightsquigarrow u_2) 
\,,
\end{align*}
Using \cref{eq:magnetizationvariance}, we thus find that
\[
\Var_{\x_0}(\cR_t^{\x_0}(X_t)) 
\leq \frac{2}{\theta}\sum_v d_G(v)^2 \leq \frac{2\Delta^2}{\theta}n =  C_0n\,.
\]
As this holds for any $t$, we may take $t\to\infty$ and arrive at
\[
\Var_{\mu_G}(\cR_t^{\x_0}(Y)) \leq  C_0n\,,
\]
as required.
\end{proof}

We are ready to prove \cref{eq:lbsecondordergoal}. 
Define the set of configurations
\[
E^\ddagger = \Big\{\x \in \Cq^V \, : \;  \cR_{t^\ddagger}^{\x_0}(\x) > \big(\E_{\x_0}[\cR_{t^\ddagger}^{\x_0}(X_{t^\ddagger})]+\E_{\mu_G}[\cR_{t^\ddagger}^{\x_0}(Y)\big)/2\Big\}\,.
\]
Recalling from \cref{eq:t-ddagger-def} that $t^\ddagger=\frac1{4\theta}\log n - C^\ddagger(\log n)^{1-\frac\gAlpha2}$ and plugging it in
\cref{eq:meandifferencefinal}, we see that
\[
\E_{\mu_G}[\cR_{t^\ddagger}^{\x_0}(Y)] - 
\E_{\x_0}[\cR_{t^\ddagger}^{\x_0}(X_{t^\ddagger})] \geq \frac1{2} \sqrt{n} \exp\Big[\big(2\theta C^\ddagger-C\big) (\log n)^{1-\frac{\gAlpha}{2}}\Big]\,,
\]
for the constant $C(\theta,q,\gAlpha,\gKap)>0$ from \cref{eq:meandifferencefinal}. Letting
$ C^\ddagger = C/(2\theta) + 1$, this results in the lower bound $(1/C) \sqrt{n} \exp[(\log n)^{1-\frac\gAlpha2}]$. We now combine this with \cref{clm:variance-R-corr}, and deduce from Chebyshev's inequality that
\begin{align*}
\P_{\x_0}(X_{t^\ddagger} \in E^\ddagger) 
\leq \frac{4\Var_{\x_0}(\cR_{t^\ddagger}^{\x_0}(X_{t^\ddagger}))}{(\E_{\mu_G}[\cR_{t^\ddagger}^{\x_0}(Y)] - \E_{\x_0}[\cR_{t^\ddagger}^{\x_0}(X_{t^\ddagger})])^2} &\leq 16 C_0  \exp\Big[-2 (\log n)^{1-\frac\gAlpha2}\Big]\,. \end{align*}
Similarly,
\begin{align*}
\P_{\mu_G}(Y \notin E^\ddagger)  
\leq \frac{4\Var_{\mu_G}(\cR_{t^\ddagger}^{\x_0}(Y))}{(\E_{\mu_G}[\cR_{t^\ddagger}^{\x_0}(Y)] - \E_{\x_0}[\cR_{t^\ddagger}^{\x_0}(X_{t^\ddagger})])^2} &\leq 16 C_0 \exp\Big[-2 (\log n)^{1-\frac\gAlpha2}\Big]\,,
\end{align*}
and it follows that
\[
\|\P_{\x_0}(X_t \in \cdot) - \mu_G\|_\tv \geq \P_{\mu_G}(Y \in E^\ddagger)  - \P_{\x_0}(X_t \in E^\ddagger) \geq 1- 32 C_0 e^{-2(\log n)^{1-\gAlpha/2}}\,,
\]
thereby establishing \cref{eq:lbsecondordergoal}.
\qed

\section{Concluding remarks and open problems}\label{sec:open-prob}

Our main result, \cref{thm:main}, established that the noisy voter model, for every choice of the parameters $\theta,q$, exhibits total variation cutoff from any sequence of initial conditions $\x_0^{(n)}$ provided that the underlying graphs $G^{(n)}$ have subexponential growth of balls as per \cref{eq:expansion}, 
(Previously, cutoff was only known  for $q=2$ and only from worst-case and a few other periodic initial states.) 
We characterized the fastest initial state $\x_0$ on $\Z^d_n$ for all $d$ at $q=2$ and for $d=1$ at every $q$, demonstrated that this optimum is nontrivial for large $d,q$ (e.g., at $d=3, q=4$, or $d=2, q=5$), and showed in
\cref{cor:uniform} that it cannot asymptotically outperform the uniform initial condition~$\cU$.
It would be interesting to extend this to bounded-degree graphs that do not satisfy \cref{eq:expansion}:
\begin{question}
    \label{q:gen-graphs}
Let $G^{(n)}$ be a sequence of bounded-degree connected graphs on $n$ vertices. Does the noisy voter model on $G^{(n)}$, for all fixed $q\geq 2$ and $\theta\in(0,1)$, exhibit cutoff 
from every sequence~$\x_0^{(n)}$ of initial states (i.e., is there a function $T_n$ so that $\tmix^{\x_0}(\epsilon)=(1+o(1))T_n(\x_0^{(n)})$ for all fixed $0<\epsilon<1$)?
Is the uniform initial state always asymptotically as fast as the deterministic optimum  $\min_{\x_0}T_n(\x_0)$?
\end{question}

A concrete class of graphs to study would be a typical random $d$-regular graph for fixed $d\geq 3$.
We remark that we used the condition in \cref{eq:expansion} in an essential way in both the upper bounds and the lower bounds.
In the upper bound, it allowed us to reduce the surviving history of the sites, over a negligible time period of $o(\log n)$, to a collection of well-separated (\good) subsets (see \cref{lem:E}). This supported the coupling to a product chain in \cref{subsec:tv-t2-product-t1}, and further played a role in the reduction to the autocorrelation function in \cref{subsec:reductiontoA2}. In the lower bound, the subexponential growth was used to control the variance of the test function in the bound $\tmix^{\x_0}\gtrsim T_{\x_0}$ (see \cref{clm:var-R(Xt)}), and to identify sets $S$ with a small conductance $\Phi(S)$ so as to bound the escape probability of the random walk from them en route to the bound $\tmix^{\x_0}\gtrsim \frac1{4\theta}\log n$ (see \cref{lemma:trappingball}).

Another interesting open problem is the cutoff window. The work \cite{CPS16} showed cutoff with an $O(1)$ window from a worst-case starting state on any graph $G$ when $q=2$ (as did \cite{LubetzkySly16} when $G=\Z_n$, the 1D Ising model); i.e., $\tmix(\epsilon)-\tmix(1-\epsilon) = O(1)$ with an implicit constant depending on $\epsilon$. Our new results on a general $\x_0$ only
focused on the asymptotics of $\tmix^{\x_0}$ (as did those of~\cite{LubetzkySly21}). For example, on $G=\Z^d_n$, our proof arguments will only show that $\tmix(\epsilon)-\tmix(1-\epsilon)= O(\log\log n)$.

\begin{question}
    \label{q:window}
Let $G^{(n)}$ be a sequence of connected graphs on $n$ vertices satisfying \cref{eq:expansion}. 
Does the noisy voter model on $G^{(n)}$, for every fixed $\theta\in(0,1)$ and $q\geq 2$, and every sequence of initial conditions $\x_0^{(n)}$, satisfy $\tmix^{\x_0}(\epsilon)-\tmix^{\x_0}(1-\epsilon) = O(1)$ with an implicit constant depending on $\epsilon$?
\end{question}

We conclude with an open problem on the noisy voter model on a class of graphs $G=G^{(n)}$ in which $\max_v d_{G}(v) / \min_v d_{G}(v)\to\infty$ as $n\to\infty$. (We only considered $G$ with $\max_v d_{G}(v)= O(1)$, and even the worst-case initial state analysis of \cite{CPS16} required $\max_v d_{G}(v) / \min_v d_{G}(v) =O(1)$.) Namely, as we mentioned the random $d$-regular in the context of \cref{q:gen-graphs}, it is natural to ask what the behavior would be on its counterpart, the Erd\H{o}s--R\'enyi random graph $\cG(n,p)$ for $p=d/n$. Note that with an unbounded degree ratio, \cref{eq:varianceinequality2} no longer gives that $\Var_{\mu_G}(\Psi_l^{(k)}) \asymp 1/n$, and it is then plausible that the correct generalization of \cref{thm:main} would be given in terms of the variant of $\autII_t(\x_0)$ where the functions $\Psi_l^{(k)}$ (as per \cref{eq:At2-rep-2}) are normalized in $L^2(\mu_G)$.

\begin{question}\label{q:G(n,p)}
Address the analogue of \cref{q:gen-graphs} for noisy voter model on a typical instance of an Erd\H{o}s--R\'enyi random graph $\cG(n,p)$ with $p=d/n$, for all fixed $d>0$, $q\geq 2$ and $\theta\in(0,1)$.
\end{question}

\bibliographystyle{abbrv}
\bibliography{noisy}
\end{document}